\documentclass[pdflatex,sn-mathphys-num]{sn-jnl}

\usepackage{tcolorbox}
\usepackage{amsmath, appendix, ulem}
\usepackage{amssymb, color}
\usepackage{graphicx}
\usepackage{float}
\usepackage{epsf}

\DeclareUnicodeCharacter{202F}{{\color{red} FIX ME!!!!}}

\usepackage{titletoc}

\usepackage{subcaption}

\usepackage{subeqnarray}
\usepackage{cases}
\graphicspath{{../Figures/}}

\numberwithin{equation}{section}

\newtheorem{lem}{Lemma}[section]
\newtheorem{thm}{Theorem}[section]
\newtheorem{proposition}[thm]{Proposition}
\newtheorem{cor}[thm]{Corollary}
\theoremstyle{remark}

\newtheorem{assum}{Assumption}
\newtheorem{heuristics}{Heuristic}

\usepackage{tikz}
\usetikzlibrary{positioning, arrows.meta, calc}

\tikzset{
  box/.style = {
    rectangle,
    draw,
    rounded corners,
    minimum height=10mm,
    text width=1.2cm,     
    align=center,
    font=\sffamily
  },
  arrow/.style = {-{Stealth[length=6pt,width=6pt]}, thick}
}

\usepackage{mathtools}
\usepackage{esint}

\newcommand{\nn}{\nonumber}
\newcommand{\R}{{\mathbb R}}

\newcommand{\N}{{\mathbb N}}

\renewcommand{\tilde}{\widetilde}
\renewcommand{\hat}{\widehat}
\renewcommand{\bar}{\overline}

\newcommand{\mc}[1]{\mathcal{#1}}

\newcommand{\mb}[1]{\mathbf{#1}}

\newcommand{\EE}{\mathbb{E}}
\newcommand{\RR}{\mathbb{R}}
\newcommand{\NN}{\mathbb{N}}
\newcommand{\PP}{\mathbb{P}}

\newcommand{\DD}{\mathcal{D}}
\newcommand{\ODD}{\overline{\mathcal{D}}}

\newcommand{\OX}{\overline{X}}

\newcommand{\OL}{\overline{L}}

\newcommand{\TE}{\mathcal{E}}

\newcommand{\la}{\langle}
\newcommand{\ra}{\rangle}

\DeclareMathOperator\diag{diag}

\usepackage{graphicx}
\usepackage{caption}
\usepackage{calc}
\usepackage{ragged2e}
\usepackage{setspace}
\newcommand{\FundingLogos}{%
  \raisebox{0pt}{\includegraphics[height=1.5cm]{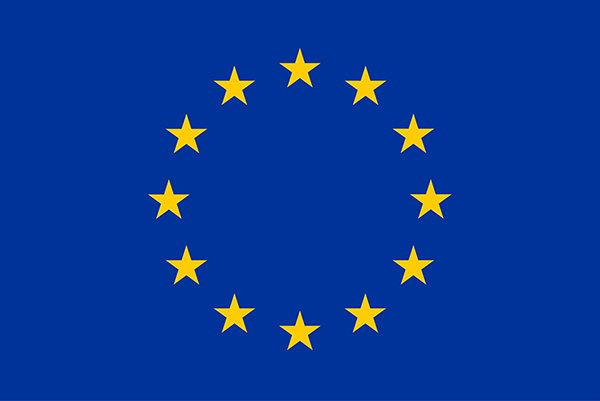}}%
  \hspace{1em}%
  \raisebox{0pt}{\includegraphics[height=1.5cm]{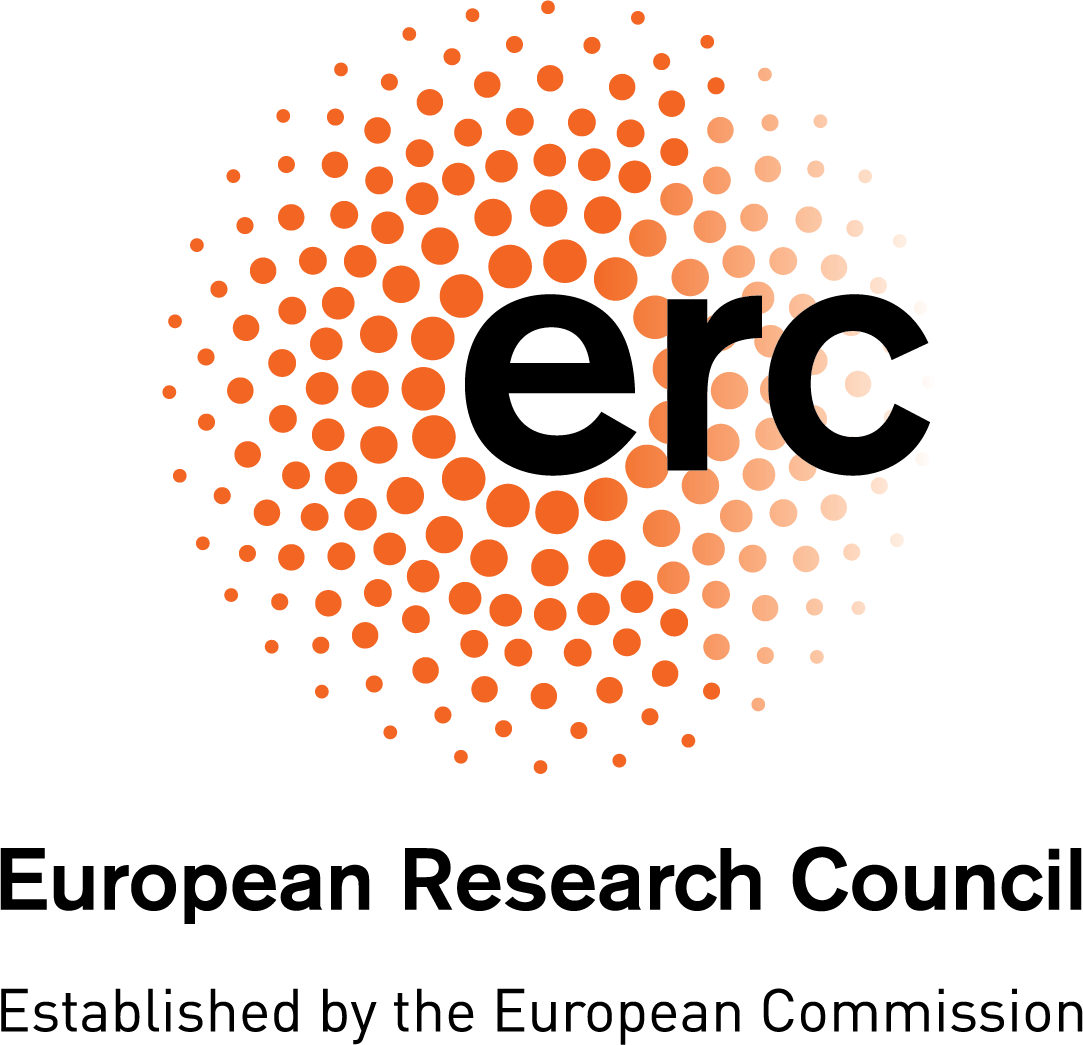}}%
}

\usepackage{tcolorbox}
\tcbset{colback=white,boxrule=0.5pt,arc=2pt,left=4pt,right=4pt,top=4pt,bottom=4pt}

\begin{document}

\title[CBO over domains with boundaries and heuristics]{Constrained Consensus-Based Optimization and Numerical Heuristics for the Few Particle Regime}

\author[1]{\fnm{Jonas} \sur{Beddrich}}\email{beddrich@ma.tum.de}
\author[2]{\fnm{Enis} \sur{Chenchene}}\email{enis.chenchene@univie.ac.at}
\author*[1,3,4]{\fnm{Massimo} \sur{Fornasier}}\email{massimo.fornasier@cit.tum.de}
\author[5]{\fnm{Hui} \sur{Huang}}\email{huihuang1@hnu.edu.cn}
\author[1]{\fnm{Barbara} \sur{Wohlmuth}}\email{wohlmuth@cit.tum.de}

\affil[1]{\orgdiv{Department of Mathematics}, \orgname{Technical University of Munich}, \orgaddress{\city{Garching by Munich}, \country{Germany}}}

\affil[2]{\orgdiv{Faculty of Mathematics}, \orgname{University of Vienna}, \orgaddress{\city{Vienna}, \country{Austria}}}

\affil[3]{\orgdiv{Munich Data Science Institute}, \orgname{Technical University of Munich}, \orgaddress{\city{Garching by Munich}, \country{Germany}}}

\affil[4]{\orgname{Munich Center for Machine Learning}, \orgaddress{\city{Munich}, \country{Germany}}}

\affil[5]{\orgdiv{School of Mathematics}, \orgname{Hunan University}, \orgaddress{\city{Changsha}, \country{China}}}

\abstract{
Consensus-based optimization (CBO) is a versatile multi-particle optimization method for performing nonconvex and nonsmooth global optimizations in high dimensions. Proofs of global convergence in probability have been achieved for a broad class of objective functions in unconstrained optimizations. In this work we adapt the algorithm for solving constrained optimizations on compact and unbounded domains with boundary by leveraging emerging reflective boundary conditions. In particular, we close a relevant gap in the literature by providing a global convergence proof for the many-particle regime comprehensive of convergence rates.
On the one hand, for the sake of minimizing running cost, it is desirable to keep the number of particles small. On the other hand, reducing the number of particles implies a diminished capability of exploration of the algorithm.  Hence numerical heuristics are needed to ensure convergence of CBO in the few-particle regime.
In this work, we also significantly improve the convergence and complexity of CBO by utilizing an adaptive region control mechanism and by choosing geometry-specific random noise. In particular, by combining a hierarchical noise structure with a multigrid finite element method, we are able to compute global minimizers for a constrained p-Allen--Cahn problem with obstacles, a very challenging variational problem.
}



\maketitle

\tableofcontents

\section{Introduction}
{
In this paper we are concerned with theoretical guarantees for global optimization, namely 
ensuring computability of global minimizers
\begin{equation}\label{globopt}
x^* \in \arg\min_{x\in\ODD}\TE(x)
\end{equation}
of a possibly nonconvex and nonsmooth objective function $\TE: \ODD \subset \mathbb R^d \to \R$. Here $\DD$ is a domain in $\mathbb R^d$ and $\ODD$ its closure. This problem can be rightfully considered the ultimate challenge of mathematical optimization, because there are no clear directional or geometrical principles  to guide the search for global minima \cite{glob24,glob22}. Local optimization techniques like gradient descent or (quasi-)Newton's methods struggle with nonconvex problems because they often get trapped in possibly numerous local minima. 
Even solving a simple polynomial equation can highlight these limitations, requiring multiple independent runs  and possibly randomizing the initial conditions to be able to find all solutions. \\
Very-many-runs first or second order methods and metaheuristics, such as Simulated Annealing, Genetic Algorithms, Particle Swarm Optimization,  Ant Colony Optimization, and other evolution strategies \cite{blum2003metaheuristics,GendPotv13}, have achieved extraordinary empirical successes in global optimization and have been extensively benchmarked, as evidenced by platforms like COCO \cite{COCO21}.
For Simulated Annealing the convergence in probability can be guaranteed by the large time analysis of the corresponding Fokker--Planck equation with vanishing diffusion by leveraging $\log$-Sobolev inequalities as forms of Polyak--Łojasiewicz conditions, see \cite{chizat2022meanfield} and the older results \cite{chiang1987diffusions,geman1986diffusions,holley1988simulated}. Unfortunately, the theoretically guaranteed rate of convergence in time $\mathcal O(\log \log(t)/\log(t))$ of Simulated Annealing is rather poor, see \cite[Theorem 4.1]{chizat2022meanfield}. 
Instead, proving the global  convergence of the widely used multi-agent methods Genetic Algorithms \cite{reeves2010genetic} and Particle Swarm Optimization \cite{kennedy1995particle} with quantitative rates has remained an elusive challenge for decades, until recent breakthroughs did connect these to multi-agent dynamics \cite{borghi2023kinetic,huang2022global,grassi2020particle,carrillo2018analytical}. More specifically, the most recent groundbreaking paradigm of Consensus-Based Optimization (CBO) \cite{pinnau2017consensus} introduced by Pinnau et al. and its first analytical description \cite{carrillo2018analytical} by Carrillo et al. have introduced the right mathematical tools for attacking the long-standing problem of rigorously proving global convergence for such methods.
CBO  aims at fusing the cooling strategy of Simulated Annealing \cite{kirkpatrick1983optimization,kirkpatrick1984optimization,chiang1987diffusions,geman1986diffusions,holley1988simulated} towards Gibbs equilibria with the space exploration by multiple particles/explorers as in Particle Swarm Optimization \cite{kennedy1995particle,zhang2015comprehensive} by taking advantage of a consensus mechanism as in the Cucker--Smale alignment model \cite{carrillo2010asymptotic} or the  Hegselmann--Krause opinion formation model \cite{hegselmann2002opinion}, in which an average orientation or opinion is obtained from the individual observations. 
For the case $\DD = \mathbb R^d$ of unconstrained optimization as first introduced and addressed in \cite{pinnau2017consensus}, the equations defining the iterates $X^i_k$ of the CBO algorithm read ($i=1,\dots,N$ labelling the particles, and $k=0,\dots,K$ denoting the iterates)
\begin{align}
\label{eq:B1-CBO1}
X^i_{(k+1)\Delta t} &= X^i_{k\Delta t} - \Delta t \lambda(X^i_{k\Delta t}-X_{\alpha}( \rho_{k\Delta t}^N)) + \sqrt{\Delta t}\sigma |X^i_{k\Delta t}-X_{\alpha}( \rho_{k\Delta t}^N)| N^i_k(0,1), 
\end{align}
where ${\mathcal E}$ is the objective function to be minimized,
$X_{\alpha}( \rho_{k\Delta t}^N) = \frac{1 }{\sum_{i=1}^N \omega_\alpha^{\mathcal E}(X_{k\Delta t}^i)} \sum_{i=1}^N X_{k\Delta t}^i \omega_\alpha^{\mathcal E}(X_{k\Delta t}^i)$ is called the instantaneous consensus-point,
 $\omega_\alpha^{\mathcal E}(x)=e^{-\alpha {{\mathcal E}}(x)}$ is  the Gibbs weight, and $N^i_k(0,1)$ is an independent standard Gaussian random vector. The other constants $\lambda,\sigma,\Delta t,\alpha$ are all positive parameters of the algorithm, describing drift strength, volatility, time discretization, and inverse temperature respectively.
The initial optimizers $X_0^i$ are drawn i.i.d. at random according to a given probability distribution $\rho_0$. The choice of the weight function $\omega_\alpha^\TE(x)$ comes from  the  well-known Laplace's principle \cite{miller2006applied,Dembo2010}, which states that for any probability measure $\mu\in\mc{P}( \RR^d )$, there holds 
\begin{equation}\label{lap_princ}
	\lim\limits_{\alpha\to\infty}\left(-\frac{1}{\alpha}\log\left(\int_{ \RR^d }\omega_\alpha^\TE(x)\mu(dx)\right)\right)=\inf\limits_{x \in \rm{supp }(\mu)} \TE(x)\,.
\end{equation}
The algorithm \eqref{eq:B1-CBO1} has a very simple formulation (it can be implemented in couple of lines of code) with a computational cost of $\mathcal O(N)$ and is easily parallelizable, see, e.g., \cite{benfenati22,JMLR:v25:23-0764}. Moreover, it is  based on sole pointwise evaluations of the objective function $\mathcal E$, hence, it does not require any higher order information. \\
 Since 2017 CBO has attracted enormous attention in virtue of its elegant formulation, its ability of breaking locality, and its astounding potential of providing a blueprint for proving convergence of other evolution strategies such as local Monte Carlo, Metropolis--Hastings, CMA-ES, see \cite{riedl2023gradient}, and Particle Swarm Optimization, see \cite{grassi2020particle,cipriani2022zero,huang2022global}. In Figure \ref{fig:CBOworld} we illustrate how CBO is collocated within the global optimization landscape, with a full explanation of the relationships with different relevant methods in the Appendix.\\

\begin{figure} 
\centering
\begin{tikzpicture}[xshift=0cm, node distance=15mm and 30mm]
  \node (mid) { }; 
  \node[box, above=of $(mid)$] (sa) {SA};
  \node[box, below=of $(mid)$] (pso) {PSO};
  \node[box, right=10mm of mid] (cbo) {CBO};
  \node[box, right=25mm of $(cbo)$] (ch) {CH};
  \node[box, below=of $(ch)$] (es) {ES};
  \node[box, right=25mm of $(ch)$] (mms) {MMS};
  \node[box, below=of $(mms)$] (tr) {TR};
  
 \node[box, right=25mm of $(mms)$] (sgd) {SGD};
 \draw[arrow] (sa.302.5) -- (cbo.112.5) 
 node[pos=0.4,fill=white] {\cite{fornasier2021consensus,pinnau2017consensus}};
  \draw[arrow] (pso.67.5) -- (cbo.247.5)
  node[pos=0.4, fill=white] {\cite{grassi2020particle,cipriani2022zero,huang2022global,pinnau2017consensus}};
  \draw[arrow] (cbo.east) -- (ch.west)
  node[pos=0.45,fill=white] {\cite{fornasier2025, riedl2023gradient}};
  \draw[arrow] (ch.south) -- (es.north)
  node[pos=0.4,fill=white] {\cite{riedl2023gradient,roith25}};
  \draw[arrow] (ch.east) -- (mms.west)
  node[pos=0.45,fill=white] {\cite{riedl2023gradient}};
  \draw[arrow] (mms.south) -- (tr.north)
  node[pos=0.4,fill=white] {\eqref{app3}-\eqref{app6} };
  \draw[arrow] (mms.east) -- (sgd.west)
  node[pos=0.45,fill=white] {\cite{riedl2023gradient}};
\end{tikzpicture}

\caption{Collocation of Consensus-Based Optimization (CBO) within the global optimization landscape. Here SA$=$``Simulated Annealing", PSO$=$``Particle Swarm Optimization", CH$=$``Consensus Hopping", ES$=$``Evolution Strategies", MMS$=$``Minimizing Movement Scheme" (or Proximal Point Method), TR=``Trust-Region", SGD=``Stochastic Gradient Descent". A full explanation of the connections is reported in the Appendix.}
\label{fig:CBOworld}
\end{figure}

 Several proofs of global convergence of CBO have been proposed \cite{carrillo2018analytical,carrillo2019consensus,ha2021convergence,doi:10.1142/S0218202522500245,huang2025faithful}, prominently the most comprehensive one of one of us and collaborators \cite{fornasier2021consensus}, which provides global convergence in probability with {\it quantitative} convergence rates depending on $\Delta t \to 0$, $N\to \infty$, $\alpha\to \infty$ under mild conditions on $\mathcal E$, namely local Lipschitz continuity and local growth at minimizers. 
This convergence result is based first on proving that the dynamics for $\Delta t \to 0$ and $N \to \infty$ in \eqref{eq:B1-CBO1} can be described by its mean-field approximation \cite{huang2022mean,gerber2023mean,huang2025uniform}
\begin{equation}\label{eq:B1-CBO2}
		d \overline{X}_t
		=-\lambda\left(\overline{X}_t-X_\alpha(\rho_t)\right)dt+\sigma |\overline{X}_t-X_\alpha(\rho_t)| d B_t.
	\end{equation}
 where  $X_\alpha(\rho)=\frac{\int x \omega^{{\mathcal E}}_\alpha(x) d\rho(x)}{\int \omega^{\mathcal E}_\alpha(x) d\rho(x)}$ and $\rho_t= \operatorname{Law}(\overline{X}_t)$, which fulfills the nonlinear Fokker--Planck equation
 \begin{align} \label{eq:B1-fokker_planck}
	\partial_t\rho_t
	= \lambda\operatorname{div} \big(\!\left(x - X_\alpha(\rho_t)\right)\rho_t\big)
	+ \frac{\sigma^2}{2}\Delta\big(|x-X_\alpha(\rho_t)|^2\rho_t\big).
\end{align}
The large time behavior of \eqref{eq:B1-fokker_planck} that finally explains the convergence of \eqref{eq:B1-CBO1} is a particularly challenging problem, because \eqref{eq:B1-fokker_planck} does not fulfill a recognizable gradient flow structure and it is not prone to more common techniques such as entropy dissipation methods. For this reason, in \cite{fornasier2021consensus} an {\it ad hoc} technique based on a quantitative version of Laplace's principle, see \cite[Proposition 21 and Proposition 23]{fornasier2021consensus}, had to be devised resulting in  proving that, for $\alpha$ large enough, the squared Wasserstein distance ${W_2}^2(\rho_t, \delta_{x^*})$ (here $x^*$ is the assumed unique global minimizer of $\TE$) is a natural Lyapunov functional for \eqref{eq:B1-fokker_planck} at finite time, with exponential decay
\begin{equation}\label{eq:contr}
W_2^2(\rho_t, \delta_{x^*}) \leq W_2^2(\rho_0, \delta_{x^*}) e^{-(1-\theta)(2\lambda - d \sigma^2) t}, 
\end{equation}
for a suitable $\theta \in (0,1)$ for $t \in [0,T^*]$ and 
$W_2^2(\rho_{T^*}, \delta_{X^*}) \leq \epsilon$. The meaning of \eqref{eq:contr} is that, remarkably, in the many-particle limit the fundamental mechanism of CBO is to transform a nonconvex optimization problem into the canonical problem of minimizing the squared distance to the global minimizer.
Hence, the average trajectories of the particles move precisely in straight lines towards the global minimizer, demonstrating robust convergence even in the presence of a nonconvexity.
Standard results of numerical approximation of  solutions of  SDEs  \cite{platen1999introduction} one obtains that the numerical scheme \eqref{eq:B1-CBO1} converges for $\Delta t \to 0$ to a system of SDE;
a quantitative mean-field limit \cite[Theorem 2.6]{gerber2023mean} for $N\to \infty$ shows that the SDE system converges in suitable sense to the mean-field model  \eqref{eq:B1-CBO2}; finally the large time analysis of \eqref{eq:B1-fokker_planck} allows to conclude \eqref{eq:contr}. The combination of these three approximation levels yields \cite[{Theorem 3.8}]{fornasier2021consensus}, showing convergence in probability of \eqref{eq:B1-CBO1}
\begin{align}\label{B1-ownresult2}
  & \mathbb P\Bigg( \Bigg| \frac{1}{N} \sum_{i=1}^N X_{T^*}^i-x^*\Bigg|^2 \leq \varepsilon\Bigg) \geq \\ \nonumber 
  & 1 - \left [ \varepsilon^{-1} (C_{\mathrm{NA}}\Delta t+C_{\mathrm{MFA}} N^{-1}+ C_0 e^{-(1-\theta) (2 \lambda - d \sigma^2) T^*})+\delta \right ],
\end{align}
where $\varepsilon,\delta,\theta$ are arbitrarily small positive constants and $T^*=\Delta t K$ is a sufficiently large running time depending on $\varepsilon$.
\\

In the meanwhile, the original CBO model~\cite{pinnau2017consensus} has been adapted to address a multitude of different problems, namely solving for multiple minimizers, saddle point problems, equilibria of games, multiobjective optimizations and more. It has been also employed with success in a  large spectrum of real-life applications. The related literature is quite vast by now. Rather than attempting to include a necessarily incomplete account of this very fast growing field,  we refer to the review paper \cite{totzeck2021trends}, and to \cite{fornasier2024polarized} for a more recent and relatively comprehensive report. Below we limit ourselves to recall the relevant results, which are more focused on the specific subject of the present work.
\\

In the following let $\DD\subsetneq \R^d$ be an open convex set with the boundary $\partial \DD$ and, by $\ODD$ we denote  again the closure of $\DD$.
The scope of the present paper is to close one of a relevant gaps left in the literature, that is providing a suitable formulation for CBO and related theoretical guarantees to solve global optimization problems of the type \eqref{globopt} constrained on convex domains with boundary. \\
It should be mentioned that the results included in this paper are not the only recent contribution aiming at formulating a CBO method for constrained optimization. \\
The first appearing work in this direction is \cite{bae2022constrained}, where the authors analyze a fully discrete CBO scheme like \eqref{eq:B1-CBO1} with a correction implemented by an orthogonal projection onto $\ODD$, see also \eqref{numericeq0} below, for the case where the domain is convex. While in the paper the authors offer guarantees  of convergence to a consensus point, they do not go so far to provide convergence rates to global minimizers in terms of number $N$ of particles as in \eqref{B1-ownresult2}. Successive contributions \cite{borghi2021constrained,carrillo2021consensus} formulated the constrained optimization problem as a unconstrained optimization with penalty. While \cite{carrillo2021consensus} is focused on the formulation of the method and numerical results, the convergence proof in mean-field law in \cite{borghi2021constrained} is based on the analysis of large time behavior of a Fokker--Planck equation as in \eqref{eq:B1-fokker_planck} and makes use of \cite[Theorem 3.7]{fornasier2021consensus}. While in this paper the domain $\DD$ does not need to be convex, the result  relies on the fundamental assumption that the penalization is exact (\cite[Assumption A1]{borghi2021constrained}) and, again, it does not achieve convergence rates in terms of number $N$ of particles as in \eqref{B1-ownresult2}. The more recent paper \cite{carrillo2024constrained} contributes with a CBO model for equality constrained optimizations (hence the domain $\DD$ is a  manifold described globally as zero set of a smooth function) and provides a proof of convergence in mean-field law, following the blueprint of 
\cite[Theorem 3.7]{fornasier2021consensus}. Finally, in the series of papers \cite{FHPS1,FHPS2,fornasier2021anisotropic}, two of us introduced and analyzed formulations of CBO for optimizations over compact hypersurfaces with no boundary, providing proof of {\it local} convergence with quantitative estimates for problems on hyperspheres in \cite{FHPS2,fornasier2021anisotropic}. 
\\

 In the present paper we analyze the  following CBO  scheme, which was already introduced in \cite{bae2022constrained}: 
 \begin{tcolorbox}[width=\linewidth, sharp corners=all, colback=white!95!black,title=Algorithm: CBO on convex domains with boundary]
 
 Given a time horizon $T>0$ and a time discretization $t_0=0<\Delta t<\dots<K\Delta t=T$ of $[0,T]$, we define
\begin{align}\label{numericeq0}
	X_{(k+1)\Delta t}^i&=\Pi_{\ODD}\Big(X_{k\Delta t}^i -\Delta t\lambda\left(X_{k\Delta t}^i-X_\alpha(\rho_{k\Delta t}^{N})\right) \\ \nonumber 
    & +\sigma D(X_{k\Delta t}^i-X_\alpha(\rho_{k\Delta t}^{N}))N^i(0,\Delta t)\Big)\notag\\
	X_0^i&\sim \rho_0,\quad i=1,\dots,N\,,
\end{align}
where $\{N^i(0,\Delta t)\}_{i=1}^N$ are independent Gaussian random vectors with zero mean and covariance matrix $\Delta t \textbf{Id}_{d}$, $D(\cdot)$ denotes the diagonal matrix produced by its argument, and $\Pi_{\ODD}$ is the orthogonal projection onto the closed convex set $\ODD$.
\end{tcolorbox}
Let us stress at this point that the only relevant assumption on the initial distribution $\rho_0$ is that it has the $16$-th moment bounded (Lemma \ref{lem:moment}). Hence, except for this concentration assumption, we can choose quite freely how the initial particles are picked.\\
For the main results we need the following assumptions, here summarized informally (for details see Assumptions \ref{assum}-\ref{assum2}):
\begin{tcolorbox}[width=\linewidth, sharp corners=all, colback=white!95!black,title=Assumptions]
\begin{itemize}
\item[(A1)]
 $\DD$ is a sufficiently regular domain (Assumption \ref{assum}). 
\item[(A2)] The objective function $\TE$ is locally Lipschitz continuous with at most linear growth of the Lipschitz constant and it enjoys a quadratic growth (Assumption \ref{assum1}). 
\item[(A3)] The objective function  $\TE$ has a polynomial growth in the vicinity of the {\it unique} minimizer $x^*$ (Assumption \ref{assum2}).
\end{itemize}
\end{tcolorbox}
A few comments on the assumptions are in order: in assumption (A2) we require at most local Lipschitz continuity of $\TE$ and its quadratic growth    in order to ensure the well-posedness of certain continuous models \eqref{particle} and \eqref{Xbareq} and for proving their mean-field relationships. This latter assumption can be relaxed to other polynomial growth, see, e.g., \cite{gerber2023mean}.
Instead, assumption (A3) is crucial  to prove global convergence to $x^*$.\footnote{While here we restrict ourselves to the case where there is only one global minimizer $x^*$, other recent papers \cite{bungert2022polarized,fornasier2024polarized} formulated CBO for solving unconstrained problems with multiple global minimizers. Paper \cite{bungert2022polarized} does not come with a proof of well-posedness, global convergence for nonconvex objectives or convergence rates. The work \cite{fornasier2024polarized} provides well-posedness and a proof global convergence with convergence rates for nonconvex objectives of the continuous model only. Hence, the extension of CBO to multiple minimizers still faces some intricate difficulties.}  
Examples of functions fulfilling Assumption (A2) are the objectives of
lasso and ridge regression,
or empirical risk functions with, for instance, the least squares loss and weight decay.
Moreover, several standard benchmark functions in optimization~\cite{JYZ}, such as the nonconvex Rastrigin or Ackley function obey also Assumption (A3).

The main result of this paper is summarized as follows:
\begin{thm}\label{thmmain0} 
Assume (A1)-(A3) and fix $\varepsilon>0$, $2\lambda > \sigma^2$, and $\alpha >0$ large enough (depending on $\varepsilon$ and $\rho_0$). Let $\{(X_{k\Delta t}^i)_{k=1,\dots,K}\}_{i=1}^N$ be the iterations generated by the Euler--Maruyama scheme \eqref{numericeq0}, where  $K\Delta t =T_*$, and $T_*$ is large enough (depending on $\varepsilon>0$). Then the final iterations fulfill the following quantitative error estimate
\begin{align}\label{quantest}
	\EE\left[\left|\frac{1}{N}\sum_{i=1}^NX_{K\Delta t}^i-x^*\right|^2\right]\ \leq 3C_{\mathrm{NA}}\Delta t\log(1/\Delta t)+3C_{\mathrm{MFA}}\frac{1}{N}+3\varepsilon\,,
\end{align}
where $C_{\mathrm{MFA}}$ comes from Theorem \ref{thmmean}, and $C_{\mathrm{NA}}$ depends on $\lambda$, $\sigma$, $\alpha$, $d$, $T_*$, $N$ and $\TE$.
\end{thm}
This  result goes beyond the above mentioned previous contributions \cite{bae2022constrained, borghi2021constrained, carrillo2024constrained} as it provides a {\it quantitative rate of convergence} with respect to relevant parameters of the algorithm, namely number of iterations $K$, time discretization step $\Delta t$, and number  of particles $N$. Moreover, differently from \cite{borghi2021constrained}, we do not need to   impose exact penalizations. Moreover, while our previous work \cite{FHPS1,FHPS2,fornasier2021anisotropic} for optimizations on  hyperspheres obtained  results of local convergence, Theorem \ref{thmmain0} ensures global convergence in probability for problems on convex domains with boundary. 
Indeed, a simple application of Markow inequality allows to obtain from \eqref{quantest} convergence in probability with a quantitative estimate of the type \eqref{B1-ownresult2}, hence we do not stress it here further.\\
The quantitative estimate \eqref{quantest} does explain well the observed behavior of the algorithm, as we will show in Section \ref{sec:numerical_experiments} with numerical experiments. Nevertheless, it is particularly meaningful for the many-particle regime $N \gg 1$, especially because the constant $C_{\mathrm{MFA}}$  shows exponential dependence on $\alpha$ and, depending on the objective function $\mathcal E$, this may be reflected in an exponential dependence on the dimension $d$, see, e.g., \cite{fornasier2021consensus}.\\
On the one hand, for the sake of minimizing running cost, it is certainly desirable to keep the number of particles quite small instead. On the other hand, reducing the number of particles implies a diminished capability of the algorithm of exploration the optimization domain and forming a consensus on the location of the global minimizer.  Hence numerical heuristics are needed to ensure that CBO in the few-particle regime mimics and behaves as the many-particle one.\\
As the second relevant contribution of this paper, in this work we also significantly improve the convergence and complexity of CBO in the few-particle regime by utilizing an adaptive region control mechanism and by choosing geometry-specific random noise. 
In particular, we show that these new heuristics are extremely powerful, allowing to solve the global optimization of the Rastrigin function in dimension $d=100$ with or without constraints, a problem that was so far completely out of reach of CBO methods, as it was becoming extremely challenging for CBO already for moderated dimensions ($d=20$), see, e.g., \cite{borghi2021constrained,BORGHI2023113859,pinnau2017consensus, totzeck2020consensus}. It remains an open and very challenging problem to provide theoretical guarantees for such heuristics for the few-particle regime.\\
Moreover, in order to test our novel approach beyond standard benchmark optimizations much tested in previous work, see, e.g., \cite{borghi2021constrained,BORGHI2023113859,pinnau2017consensus, totzeck2020consensus}, in this paper we apply for the first time CBO for solving a challenging problem in scientific computing. In particular, by integrating a hierarchical noise structure in a multigrid finite element method, we demonstrate how the CBO scheme \eqref{numericeq0} can be effectively applied to compute global minimizers for a constrained $p$-Allen--Cahn problem, both with and without obstacles. Given the existence of continua of (local) minimizers of   $p$-Allen--Cahn energies, such global variational problems are extremely challenging.  We consider this successful test a breakthrough that paves the way for further  applications of CBO in scientific computing, highlighting once more its versatility and ease of implementation. A complete theoretical study of these numerical results is left for future research.
\\

The paper is organized as follows. In Section \ref{sec:blueprint} we provide a concise explanation of the architecture of the proof of Theorem \ref{thmmain0}, which is based on the asymptotic approximation by auxiliary mean-field models for $\Delta t \to 0$ and $N\to \infty$. In Section \ref{sec:wpmfl} we address the well-posedness of such models and we quantify the approximation rate for $N\to \infty$, by establishing a quantitative mean-field limit. Section \ref{sec:globconv} is dedicated then to the proof of Theorem \ref{thmmain0}. In Section \ref{sec:num} we introduce new heuristics to improve the complexity for the few-particle regime and we illustrate their efficacy in solving constrained optimizations for well-known benchmark cases and for a constrained $p$-Allen--Cahn problem with obstacles.
}

\begin{table}[h!]
\centering
\begin{tabular}{p{2cm} p{10cm}}
\hline
\textbf{Notation} & \textbf{Definition} \\
\hline
$\TE$ & The cost function one wishes to minimize.\\
$x^*$ & The unique global minimizer of $\TE$.\\
$\ODD$ & The closure of a convex domain $\DD$. \\
$\partial \DD$ & The boundary of $\DD$. \\
$|\DD|$ & The diameter of $\DD$. \\
$n(x)$& An outward normal vector at $x\in\partial\DD$. \\
$\Pi_{\ODD}(\cdot)$ &  The orthogonal projection onto $\ODD$.\\
$|\cdot|$ & The standard Euclidean distance. \\
$\|\cdot\|_\infty$ & The $L^\infty$ norm.\\
$\mc{P}_p(\R^d)$  & The set of probability measures on $\R^d$ with a finite $p$th-moment.\\
$W_2(\mu,\nu)$ & The $2$-Wasserstein distance between probability measures $\mu, \nu\in \mc{P}_2(\R^d)$. Namely, $$W_2^2(\mu,\nu):=\inf_{\gamma\in\Gamma(\mu,\nu)}\iint_{\R^d\times\R^d}|x-y|^2\gamma(dx,dy)\,,$$
where $\Gamma(\mu,\nu)$ is the set of all couplings between $\mu$ and $\nu$.\\
$B_r(x)$ & An open ball with origin $x$ and radius $r$. \\
% ...
\hline
\end{tabular}
\caption{Notation Table. Further notations will be recalled in the text.}
\end{table}

\section{Architecture of the proof of the main result}\label{sec:blueprint}
{
While the proof of Theorem \ref{thmmain0} builds  on the blueprint of \cite{fornasier2021consensus}, 
it requires nevertheless major and nontrivial technical adaptations to deal with the boundary of $\DD$. In particular, thanks to the convexity of $\DD$ and the emergence of reflecting boundary conditions (see below), we obtain  inequalities, such as  \eqref{eq:favor}, which play at our favor in crucial estimates. Moreover, we need to devise new observations, such as Proposition \ref{propositive1} and Proposition \ref{propositive1'}, to deal with the case where the minimizer $x^*$ may lay precisely on the boundary, a situation, which is actually quite common in practical scenarios. See more details in Section \ref{sec:relevance} on the relevance of the analytic advances of this paper and their insights.\\
Now, to make sense of Theorem \ref{thmmain0} and how it is obtained, we need to introduce further notations. For $x\in\partial \DD$, we denote by $\mc{H}_x$ the set of all supporting hyperplanes of $\DD$ at $x$. By an outward normal vector $n(x)$ at $x\in\partial\DD$ one means any outward unit vector perpendicular to some $H\in \mc{H}_x$. Moreover, we denote by 
$\mc {N}_x$ the set of all outward unit normal vectors at $x\in\partial \DD$. Indeed, there is a possibility that $\#(\mc {N}_x)=\infty$ if the boundary $\partial \DD$ lacks smoothness in the vicinity of $x$. However, in our current study, we do not impose the condition of boundary smoothness. For later use let us denote $[N]:=\{1,\cdots,N\}$.

\subsection{Step 1: convergence of the numerical scheme for $\Delta t \to 0$}

The proof of Theorem \ref{thmmain0} is based on observing first that the iterations \eqref{numericeq0} converge in suitable sense for  $\Delta t \to 0$ to the solution of the following Skorokhod stochastic differential equations (SDEs)
\begin{subequations}\label{particle}
	\begin{numcases}{}
		dX_t^i=-\lambda(X_t^i-X_\alpha(\rho_t^{N}))dt+\sigma D(X_t^i-X_\alpha(\rho_t^{N}))dB_t^{i}-dL_t^i\,,\label{eqX}\\
		L_t^i=\int_0^t n(X_s^i)d|L^i|_s,\quad |L^i|_t=\int_0^t \mathbf{I}_{\partial\DD}(X_s^i)d|L^i|_s\,,\quad   i\in[N] \label{eqL}
	\end{numcases}
\end{subequations}
where $n(X_s^i)\in \mc N_{X_s^i}$ if $X_s^i\in\partial \DD$,
$\{B^{i}_{\{t\geq0\}}\}_{i=1}^N$ are $N$ independent Brownian motions and $\{L^{i}_{\{t\geq0\}}\}_{i=1}^N$ are continuous reflecting processes associated to $\{X^{i}_{\{t\geq0\}}\}_{i=1}^N$  with bounded total variation, which prevent particles leaving the domain. Moreover, $|L|_t^i$ denotes the total variation of $L_s^i$ on $[0,t]$, namely
\begin{equation}
	|L|_t^i=\sup\sum\limits_k|L_{t_k}^i-L_{t_{k-1}}^i|,
\end{equation}
where the supremum is taken over all partitions such that $0=t_0<t_1<\cdots<t_n=t$. Here, again the consensus point is defined as
\begin{equation}\label{XaN}
	X_\alpha(\rho_t^{N})=\frac{\int_{\ODD}x\omega_{\alpha}^{\mc{E}}(x)\rho^{N}(t,dx)}{\int_{\ODD}\omega_{\alpha}^{\mc{E}}(x)\rho^{N}(t,dx)}, \quad \rho^{N}(t,dx)=\frac{1}{N}\sum_{i=1}^N\delta_{X_t^i}dx\,.
\end{equation}

 \subsection{Step 2: mean-field limit for $N\to \infty$}

As $N\to\infty$, we further prove that the CBO dynamics \eqref{particle} will well approximate the solution of the following  mean-field  kinetic Mckean--Vlasov type equation
\begin{subequations}\label{Xbareq}
	\begin{numcases}{}
d\OX_t=-\lambda(\OX_t-X_\alpha(\rho_t))dt+\sigma D(\OX_t-X_\alpha(\rho_t))dB_t-d\OL_t, \label{eqXbar}\\
\OL_t=\int_0^t n(\OX_s)d|\OL|_s,\quad |\OL|_t=\int_0^t \mathbf{I}_{\partial\DD}(\OX_s)d|\OL|_s\,, \label{eqLbar}
	\end{numcases}
\end{subequations}
where 
\begin{equation}\label{Xa}
X_\alpha(\rho_t)=\frac{\int_{\ODD}x\omega_{\alpha}^{\mc{E}}(x)\rho(t,dx)}{\int_{\ODD}\omega_{\alpha}^{\mc{E}}(x)\rho(t,dx)}\,,
\end{equation}
with $\rho(t,x)$ being  required to be  the law of $\OX_t$ , which makes the set of equations \eqref{Xbareq} nonlinear and self-consistent.  

Equations \eqref{particle} and \eqref{Xbareq} are known as Skorokhod SDEs, analogous to the one-dimensional case initially studied by Skorohkhod in \cite{skorokhod1962stochastic}. The extension of Skorokhod SDEs to domains beyond just a half-line or half-space was first addressed by Tanaka in \cite{tanaka1979stochastic}, where the domain $\DD$ is assumed to be convex. This convexity condition was subsequently relaxed by Lions and Sznitman in \cite{lions1984stochastic} together with an admissibility condition that essentially requires $\DD$ to be approximated in some way by smooth domains. This admissibility condition was later eliminated by Saisho in \cite{saisho1987stochastic}.  Finally, for a comprehensive review of stochastic differential equations with reflection, we refer the reader to \cite{pilipenko2014introduction}.

A direct application of It\^{o}'s formula, the law $\rho_t:=\rho(t,\cdot)$ of $\OX$ at time $t$ is a weak solution to the following nonlinear Vlasov--Fokker--Planck equation with non-flux boundary condition
\begin{subequations}\label{MFPDE}
	\begin{numcases}{}
		\partial_{t} \rho_t=\lambda \nabla\cdot((x-X_\alpha(\rho_t))\rho_t)+\frac{\sigma^2}{2}\sum_{k=1}^{d}\partial_{x_k,x_k}^2((x-X_\alpha(\rho_t))_{k}^2\rho_t),\quad x\in\DD,~t>0\\
		\la (x-X_\alpha(\rho_t))\rho_t,n(x)\ra +\sum_{k=1}^{d}\partial_{x_k}((x-X_\alpha(\rho_t))_{k}^2\rho_t) n_k(x)=0,\quad x\in\partial \DD\,.
	\end{numcases}
\end{subequations}
with the initial data $\rho_0(x)=\mbox{Law}(\OX_0)$. Indeed for any test function $\phi\in \mc{C}^2(\bar\DD)$ with $\partial_{x_k}\phi(x)n_k(x)=0$ on $\partial \DD$ for all $k=1,\dots,d$, using It\^{o}'s lemma deduces
\begin{align}\label{ito}
	d\phi(\OX_t)=&\nabla\phi(\OX_t)\cdot[-\lambda(\OX_t-X_\alpha(\rho_t))dt+\sigma D(X_t-X_\alpha(\rho_t))dB_t]\nn\\
	&-\nabla\phi(\OX_t)\cdot n(\OX_t)\mathbf{I}_{\partial\DD}(\OX_t)d|\OL|_t +\frac{\sigma^2}{2}\partial_{x_kx_k}^2\phi(\OX_t)(X_t^i-X_\alpha(\rho_t))_k^2dt\,.
\end{align}
Then taking expectation on both sides implies that $\rho_t(x)=\mbox{Law}(\OX_t)$ is a weak solution to \eqref{MFPDE}.
 \subsection{Step 3: large time analysis of the mean-field model}
The combination of the above mentioned quantitative approximations for $\Delta t \to 0$, $N \to \infty$, and the analysis of the large time behavior of the solution $\rho_t$ of \eqref{MFPDE} or its underlying process $\OX_t$ of \eqref{Xbareq} yields the estimate \eqref{quantest} of Theorem \ref{thmmain0}. The details of the proof follow in Section \ref{sec:globconv}.\\

\tikzset{
  box/.style = {rectangle, draw, rounded corners, minimum width=3.4cm, minimum height=8mm, align=center, font=\sffamily},
  arrow/.style = {-{Stealth[length=6pt,width=6pt]}, thick}
}

\begin{figure}
\centering 
\begin{tikzpicture}[node distance=6mm and 6mm]
  \node[box, minimum height=1.25cm, minimum width=2.5cm] (in) {Theorems \\ \ref{thmparticle} $\&$ \ref{wpxbar}};
  \node[box, minimum height=1.25cm, minimum width=2.5cm, right=9mm of in] (proc) {Theorem \\ \ref{thmmean}};
  \node[box, minimum height=1.25cm, minimum width=2.5cm, right=9mm of proc] (model) {Theorems \\ \ref{thmconvergence} $\&$ \ref{thmconvergence'}};
  \node[box, minimum height=1.25cm, minimum width=2.5cm, right=9mm of model] (out) {Theorem \\  \ref{thmmain}};

  \node[box, minimum height=1.25cm, minimum width=2.5cm, above=2mm of in] (wp) {Well-posedness}; 
  \node[box, minimum height=1.25cm, minimum width=2.5cm, above=2mm of proc] (pa) {Particle \\ approximation}; 
  \node[box, minimum height=1.25cm, minimum width=2.5cm, above=2mm of model] (lta) {Large time \\ analysis}; 
  \node[box, minimum height=1.25cm, minimum width=2.5cm, above=2mm of out] (gc) {Global \\ convergence}; 

  \draw[arrow] (wp) -- (pa);
  \draw[arrow] (pa) -- (lta);
  \draw[arrow] (lta) -- (gc);
\end{tikzpicture}
\caption{Scheme of results: in Section \ref{sec:wpmfl} we address the well-posedness of \eqref{particle} and \eqref{Xbareq} and their relationship via mean-field limit $N\to \infty$. In Section \ref{sec:globconv} we analyze the large time behavior and we derive the global convergence of the numerical scheme \eqref{numericeq0}.} 
\end{figure}

During the revision of this paper, we got aware of the independently developed work \cite{M3AS25}, which addressed the same class of mean-field models as  \eqref{Xbareq}.  While well-posedness and particle approximation of the model were developed also in this latter paper, the proof of  convergence to minimizers for CBO is based on a variance analysis, which requires bounded objectives $f$, bounded domain $\mathcal D$, and have a suboptimal scaling of the parameters $\lambda, \sigma$ depending on $\alpha$ (compare the condition  $2\lambda > \sigma^2 (1 + e^{2 \alpha (f_{max}-f_{min})})$ in \cite[Theorem 5.2]{M3AS25} with the optimal one $2\lambda > \sigma^2$ as in our Theorem \ref{thmmain0}). Moreover, in contrast to Theorem \ref{thmmain0}, in \cite{M3AS25}  no proof of convergence of the numerical scheme \eqref{numericeq0} was achieved.  

\section{Main technical challenges}\label{sec:relevance}

This paper provides the first rigorous proof of {\it global} convergence for CBO in convex domains with boundaries by developing a new mean-field formulation and analytic framework that extends Wasserstein contractivity and the quantitative Laplace principle to constrained settings. 

The main insights of our novel analysis is that the presence of a boundary does not destroy the exponential-in-time contractivity of CBO towards global minimizers and that the algorithm \eqref{numericeq0} is precisely the right one to discretize the dynamics. 
 
In particular, we aim to address a significant gap in the literature: although the numerical scheme \eqref{numericeq0} has been known since its initial formulation in \cite{bae2022constrained}, no proof of global convergence has been provided to date. This gap arises for two main reasons. First, a proper analytical approach was developed only very recently in \cite{fornasier2021consensus} for unconstrained optimization. Second, the approach in \cite{fornasier2021consensus} relies on contractivity estimates involving the Wasserstein distance, as in \eqref{eq:contr}. While this argument works well for unconstrained optimization, it has remained elusive how to adapt it to domains with boundaries, and how the presence of boundaries would affect estimates involving Wasserstein-like distances. To emphasize such difficulty, let us reiterate that in all existing CBO literature on constrained optimization, authors have consistently sought to avoid dealing directly with boundaries \cite{bae2022constrained,borghi2021constrained,carrillo2021consensus,carrillo2024constrained, FHPS1,FHPS2,fornasier2021anisotropic}. (The only exception is the recent work \cite{M3AS25} already discussed above at the end of the previous section.)

One of the main contributions of this paper is the identification of a proper formulation of the mean-field approximation \eqref{Xbareq}, which enables inequalities such as \eqref{eq:favor} that are crucial in estimates involving Wasserstein-like distances, such as \eqref{eq:def_V}. Without finding the right formulation and carefully developing the entire argument, it is by no means obvious that adapting the system to reflective boundary conditions would permit contractivity of the functional \eqref{eq:def_V}, making this a central and nontrivial achievement of our work. Indeed, for minimizers precisely at the boundary, one could have been concerned that reflecting boundary conditions could contribute to counteract the convergence, an effect that the theoretical results of the present paper eventually excludes, at least for convex domains (for nonconvex domains, this issue remains open).
\\
Another central technical tool used in the paper \cite{fornasier2021consensus} is the so-called {\it quantitative Laplace principle}, which was again formulated to work for problems defined on the entire Euclidean space.  Proposition \ref{propX},  Proposition \ref{propositive1}, Proposition \ref{propositive1'}, and Corollary \ref{propositive} finally allow for the adaptation to bounded domains. Without such a proper formulation, it would have not been possible to deal with the analysis for the global minimizers laying precisely at the boundary of the domain, a situation very common in optimization.
\\
As a further contribution, the proofs of convergence as in Theorem \ref{thmconvergence} and Theorem \ref{thmconvergence'} are simpler  compared to the formulation in \cite{fornasier2021consensus}, offering a more straightforward approach to further developments.
\\
Although these theoretical results are already significant advances for the reliable use of CBO for constrained optimization, we felt interesting to explore how flexible and effective CBO can be for constrained optimization, by pushing a bit further the boundaries of numerical experiments beyond classical benchmarks \cite{JYZ} as reported so many times in other previous studies \cite{borghi2021constrained,BORGHI2023113859,pinnau2017consensus,totzeck2020consensus,M3AS25}.
Therefore, we introduce novel heuristics tailored to few-particle regimes, including adaptive region control and geometry-dependent noise. The numerical experiments we report demonstrate that, even with a small number of particles, with such heuristics CBO can successfully solve problems that previous studies could not address. We also explore new challenging applications, such as the reliable computation of ground states of nonconvex energies, and show for the first time in the CBO literature that the $p$-Allen–Cahn variational problem with obstacles can be solved efficiently and robustly, highlighting the adaptability of CBO to various geometries and problems.

\section{Well-posedness and mean-field limit}\label{sec:wpmfl}

Let us now study in details the well-posedness of the models \eqref{XaN}-\eqref{Xbareq} and their mean-field relationship. First, we introduce the Condition (B) from \cite{tanaka1979stochastic} for a convex domain $\DD$.
\begin{assum}\label{assum}
    There exists $\varepsilon>0$ and $\delta>0$ such that for any $x\in\partial D$ we can find an open ball  $B_\varepsilon(x_0)=\{y\in \R^d:~|y-x_0|<\varepsilon\}$ satisfying $B_\varepsilon(x_0)\subset \DD$ and $|x-x_0|\leq \delta$.
\end{assum}
One can see that the above assumption always holds if $\DD$ is bounded or if $d=2$.

\begin{assum}\label{assum1}
	Throughout this section we are interested in the objective function $\TE\in\mc{C}(\ODD)$, for which
	\begin{enumerate}
		\item it holds that $\underline \TE\leq\TE(x)$ for all $x\in\ODD$ and satisfies
  \begin{equation}\label{eqlip}
      \forall x,y\in \ODD,\quad |\TE(x)-\TE(y)|\leq L_{\TE}(1+|x|+|y|)|x-y|\,;
  \end{equation}
		\item {
 there exist some constants $c_u,c_\ell,C_u,C_\ell>0$ such that
		\begin{align}
		   &\forall x\in \ODD,\quad c_\ell|x|^2-C_\ell\leq \TE(x)-\underline\TE \mbox{ and }\notag \\
           &\forall x\in \ODD,\quad \TE(x)-\underline\TE\leq c_u|x|^2+C_u\,.
		\end{align}}
	\end{enumerate}
\end{assum}

\subsection{Well-posedness}
To keep the notation concise in what follows, let us denote the state vector of the entire particle system~\eqref{particle} by $\mathbf{X}\in\mc{C}([0,\infty),\R^{Nd})$ with $\mathbf{X}(t) = \mathbf{X}_t = \left((X_t^1)^T, \dots, (X_t^N)^T\right)^T$ for every $t\geq0$.
Equation~\eqref{particle} can then be reformulated as
\begin{align} \label{particlecompact}
	d\mathbf{X}_t
	= -\diag{(\lambda)}\mathbf{F}(\mathbf{X}_t)\,dt + \diag{(\sigma)}\mathbf{M}(\mathbf{X}_t)\,d\mathbf{B}_t-d\mathbf{L}_t
\end{align}
with $(\mathbf{B}_t)_{t\geq0}$ being a standard Brownian motion in $\R^{Nd}$ and  definitions
\begin{align*}
	&\mathbf{L}_t:= \left((L_t^1)^T, \dots, (L_t^N)^T\right)^T\\
	&\mathbf{F}(\mathbf{X}_t)
	:= \,\left({F}^{1}(\mathbf{X}_t)^T, \dots, {F}^{N}(\mathbf{X}_t)^T\right)^T \quad \text{ with }\quad {F}^{i}(\mathbf{X}_t) = \left(X_t^i - X_\alpha(\rho_t^{N})\right)\in \R^d \\
	&\mathbf{M}(\mathbf{X}_t)
	:= \,\diag\left({M}^{1}(\mathbf{X}_t), \dots, {M}^{N}(\mathbf{X}_t)\right) \quad\text{ with }\quad {M}^{i}(\mathbf{X}_t) = D\!\left(X_t^i - X_\alpha(\rho_t^{N})\right)\in\R^d\,.
\end{align*}
The $\diag$ operator in the definition of $\mathbf{M}$ maps the input matrices onto a block-diagonal matrix with them as its diagonal, and $\diag{(\lambda)}$ and $\diag{(\sigma)}$ are $Nd \times Nd$-dimensional diagonal matrices, whose entries are $\lambda$ and $\sigma$.

\begin{lem}\label{lemlocal}
	Assume the cost function $\TE$ satisfies Assumption \ref{assum1}-$(1)$. Let $N\in\NN$, $\alpha>0$ be arbitrary. Then for any $\mb{X},\hat{\mb{X}}\in\ODD^{N}$ satisfying $|\mb{X}|,|\hat{\mb{X}}|\leq R$, it holds that
	\begin{align}\label{lemeq1}
		|F^{i}(\mb{X})|\leq 2|X^i|+|\mb{X}|\leq 2|\mb{X}|
	\end{align}
and
	\begin{align}\label{lemeq2}
			|F^{i}(\mb{X})-F^{i}(\hat{\mb{X}})|\leq 2\left(2+RL_\TE(1+2R)\alpha \exp\left(\alpha(\overline{\TE}_R-\underline{\TE}_R)\right)\right)|\mb{X}-\hat{\mb{X}}|_\infty\,,
	\end{align}
	all $i\in[N]$, where $L_\TE$ comes from \eqref{eqlip}, $\overline{\TE}_R:=\max_{|x|\leq R}\TE(x)$, $\underline{\TE}_R:=\min_{|x|\leq R}\TE(x)$, $|\mb X|_\infty:=\max_{i=1,\dots,N}|X^i|$ for any $\mb X\in \R^{Nd}$, and  $|\cdot|$ denotes the standard Euclidean norm.
\end{lem}
\begin{proof}
Recall that, 
\begin{equation}
    F^{i}(\mb{X})=X^i-\frac{\sum_{j=1}^N(X^i-X^j)\omega_{\alpha}^{\mc{E}}(X^j)}{\sum_{j=1}^N\omega_{\alpha}^{\mc{E}}(X^j)}\,,
\end{equation}
then the estimate \eqref{lemeq1} follows directly. Now we prove \eqref{lemeq2}.
 For any $i\in[N]$, we have
	\begin{align}
		&F^{i}(\mb{X})-F^{i}(\hat{\mb{X}})=X^i-\hat X^i+\frac{\sum_{j=1}^N(X^i-X^j)\omega_{\alpha}^{\mc{E}}(X^j)}{\sum_{j=1}^N\omega_{\alpha}^{\mc{E}}(X^j)}-\frac{\sum_{j=1}^N(\hat X^i-\hat X^j)\omega_{\alpha}^{\mc{E}}(\hat X^j)}{\sum_{j=1}^N\omega_{\alpha}^{\mc{E}}(\hat X^j)}\nn\\
		=&X^i-\hat X^i+\frac{\sum_{j=1}^N(X^i-\hat X^i+\hat X^j-X^j)\omega_{\alpha}^{\mc{E}}(X^j)}{\sum_{j=1}^N\omega_{\alpha}^{\mc{E}}(X^j)}
		+\frac{\sum_{j=1}^N(\hat X^i-\hat X^j)\left(\omega_{\alpha}^{\mc{E}}(X^j)-\omega_{\alpha}^{\mc{E}}(\hat X^j)\right)}{\sum_{j=1}^N\omega_{\alpha}^{\mc{E}}(X^j)}\nn\\
		&+\sum_{j=1}^N(\hat X^i-\hat X^j)\omega_\alpha^\TE(\hat X^j)\frac{\sum_{j=1}^N\left(\omega_{\alpha}^{\mc{E}}(X^j)-\omega_{\alpha}^{\mc{E}}(\hat X^j)\right)}{\sum_{j=1}^N\omega_{\alpha}^{\mc{E}}(X^j)\sum_{j=1}^N\omega_{\alpha}^{\mc{E}}(\hat X^j)}\nn\\
		=:&X^i-\hat X^i+I_1+I_2+I_3\,.
	\end{align}
	It is not difficult to see that
	\begin{align}
		|I_1|\leq |X^i-\hat X^i|+|\mb{X}-\hat{\mb{X}}|_\infty\leq 2|\mb{X}-\hat{\mb{X}}|_\infty\,.
	\end{align}
 Using assumption \eqref{eqlip} one can further compute that
	\begin{align}
		|I_2|&\leq \frac{\sum_{j=1}^N|\hat X^i-\hat X^j|\alpha \exp\left(-\alpha((1-c)\TE(X^j)+c\TE(\hat X^j))\right)L_\TE(1+|X^j|+|\hat X^j|)|X^j-\hat X^j|}{N\exp(-\alpha\overline{\TE}_R)}\nn\\
  &\leq 2RL_\TE(1+2R)\alpha \exp\left(\alpha(\overline{\TE}_R-\underline{\TE}_R)\right)\frac{\sum_{j=1}^N|X^j-\hat X^j|}{N}\nn\\
  &\leq 2RL_\TE(1+2R)\alpha \exp\left(\alpha(\overline{\TE}_R-\underline{\TE}_R)\right)|\mb{X}-\hat{\mb{X}}|_\infty\,,
	\end{align}
 where $c\in[0,1]$ comes from the mean-value theorem. Similarly, we also have
	\begin{align}
		|I_3|\leq 2RL_\TE(1+2R)\alpha \exp\left(\alpha(\overline{\TE}_R-\underline{\TE}_R)\right)|\mb{X}-\hat{\mb{X}}|_\infty\,.
	\end{align}
	Putting all the above terms together yields the required estimate.
\end{proof}

Now for any fixed number $N$ of particles, one obtains the well-posedness of the particle system \eqref{particlecompact}.
\begin{thm}\label{thmparticle} Let $\DD$ satisfy Assumption \ref{assum} and $\TE$ satisfy Assumption \ref{assum1}-$(1)$. Then
	for each $N\in\NN$ and any initial data $\{X_0^i\}_{i=1}^N\in \ODD^N$ satisfying $\EE[|\textbf{X}_0|^2]<\infty $, there exists a pathwise unique solution $(\mathbf{X}_t,\mathbf{L}_t)_{t\in[0,T]}$ to the particle system \eqref{particle} or \eqref{particlecompact} for any $T>0$, and it satisfies $\EE\left[\sup_{0\leq t\leq T}|\mb X_t|^2\right]<\infty$.
\end{thm}
\begin{proof}
    Lemma \ref{lemlocal} deduce that $F^i$, $i\in[N]$ is locally Lipschitz continuous and has sublinear growth. Consequently, both $\mb{F}$ and $\mb{M}$ are locally Lipschitz continuous and have sublinear growth. If additionally we have the non-explosion criterion \cite[Theorem 3.5]{khasminskii2012stochastic} or \cite[Theorem 3.1]{durrett2018stochastic}, then it guarantees the well-posedness of the particle system \eqref{particle} or \eqref{particlecompact} according to \cite[Theorem 4.1]{tanaka1979stochastic}. Indeed, applying It\^{o}'s formula leads to
\begin{align}\label{firstito}
	d |X_t^i-X_0^i|^2&=2(X_t^i-X_0^i)\cdot[-\lambda(X_t^i-X_\alpha(\rho_t^{N}))dt+\sigma D(X_t^i-X_\alpha(\rho_t^{N}))dB_t^{i}-dL_t^i]\nn\\
&+\sigma^2|X_t^i-X_\alpha(\rho_t^{N})|^2dt\,.
\end{align}
Using the convexity of $\DD$ one has 
\begin{align}
(X_t^i-X_0^i)\cdot dL_t^i=(X_t^i-X_0^i)\cdot n(X_t^i)\mathbf{I}_{\partial\DD}(X_t^i)d|L^i|_t\geq 0\,.
\end{align}
Then \eqref{lemeq1} implies that for all $\tau>0$
\begin{align*}
    \EE\left[\sup_{0\leq t\leq \tau}|X_t^i-X_0^i|^2\right]
    \leq& \lambda \int_0^\tau \EE[|X_s^i-X_0^i|^2]+\EE[|F^i(\mb{X}_s)|^2]ds+\sigma^2\int_0^\tau\EE[|F^i(\mb{X}_s)|^2]ds\nn\\
    & +2\sigma \EE\left[\sup_{0\leq t\leq \tau}\left|\int_0^t(X_s^i-X_0^i)\cdot D(X_s^i-X_\alpha(\rho_s^{N}))dB_s^{i}\right|\right]\\
    \leq& C\int_0^\tau\EE[|X_s^i-X_0^i|^2]ds+C\int_0^\tau\EE[|\mb{X}_s|^2]ds \\ \nonumber
   & +2\sigma \EE\left[\sup_{0\leq t\leq \tau}\left|\int_0^t(X_s^i-X_0^i)\cdot D(X_s^i-X_\alpha(\rho_s^{N}))dB_s^{i}\right|\right]\,.
\end{align*}
It follows from Burkholder--Davis--Gundy inequality that
\begin{align*}
    &\EE\left[\sup_{0\leq t\leq \tau}\left|\int_0^t(X_s^i-X_0^i)\cdot D(X_s^i-X_\alpha(\rho_s^{N}))dB_s^{i}\right|\right]\\
    \leq& C\EE\left[\left(\int_0^\tau |X_s^i-X_0^i|^2|F^i(\mb{X}_s)|^2ds\right)^{1/2}\right] \\ 
    \leq& C\EE\left[\left(\sup_{0\leq t\leq \tau}|X_t^i-X_0^i|^2\int_0^\tau |F^i(\mb{X}_s)|^2ds\right)^{1/2}\right]\\
    \leq & \frac{1}{2} \EE\left[\sup_{0\leq t\leq \tau}|X_t^i-X_0^i|^2\right]+C\int_0^\tau\EE[|F^i(\mb{X}_s)|^2]ds \\ \leq &  \frac{1}{2} \EE\left[\sup_{0\leq t\leq \tau}|X_t^i-X_0^i|^2\right]+C\int_0^\tau\EE[|\mb{X}_s|^2]ds\,.
\end{align*}
This yields that
\begin{align*}\label{moment}
    \EE\left[\sup_{0\leq t\leq \tau}|X_t^i-X_0^i|^2\right]\leq C\int_0^\tau\EE[|X_s^i-X_0^i|^2]ds+C\int_0^\tau\EE[|\mb{X}_s-\mb{X}_0|^2]ds+C\EE[|\mb{X}_0|^2]\,,
\end{align*}
which leads to
\begin{align}
     \EE\left[\sup_{0\leq t\leq \tau}|\mb X_t-\mb X_0|^2\right]\leq 
   C\int_0^\tau\EE[|\mb X_s-\mb X_0|^2]ds+C\EE[|\mb{X}_0|^2]\,.
\end{align}
Therefore Gronwall's inequality concludes that
\begin{equation}
  \EE\left[\sup_{0\leq t\leq \tau}|\mb X_t-\mb X_0|^2\right]\leq  C\EE[|\mb{X}_0|^2]\exp(C\tau)\quad \forall \tau> 0
\end{equation}
where $C>0$ depend only on $\lambda,\sigma$ and $N$. Namely, the solution exists globally in time for each fixed $N$.
\end{proof}

In order to prove the well-posedness of the mean-field dynamics \eqref{Xbareq}, we need the following stability estimate from \cite[Lemma 3.2]{carrillo2018analytical}:
\begin{lem}\label{lemLip}
	Let $\TE$ satisfy Assumption \ref{assum1}-$(1)$ and $\rho,\hat \rho\in \mc{P}_4(\ODD)$ with
 $$\int |x|^4d \rho (dx),\quad \int |x|^4d \hat \rho (dx)\leq K.$$
	Then it holds that
	\begin{equation}
		|X_\alpha(\rho)-X_\alpha(\hat \rho)|\leq c_0W_2(\rho,\hat \rho)\,,
	\end{equation}
	where $c_0$ depends only on $\alpha,L_\TE$ and $K$.
\end{lem}
We will also need the following estimate on $X_\alpha(\rho)$ \cite[Proposition A.3]{gerber2023mean}:
\begin{lem}\label{lemXalpha}
    Let $\TE$ satisfy Assumption \ref{assum1}-$(2)$ and $\rho\in \mc{P}_p(\ODD)$. Then there exists some constant $c_1>0$ depending only on $c_u,c_\ell,C_u,C_\ell,\alpha$ such that it holds
    \begin{equation}
    |X_\alpha(\rho)|\leq \left(c_1\int_{\ODD}|x|^p\rho(dx)\right)^{\frac{1}{p}}\quad \mbox{ for }p\geq 1\,.
    \end{equation}
\end{lem}

\begin{thm}\label{wpxbar}
Let $\DD$ satisfy Assumption \ref{assum},  $\TE$ satisfy Assumption \ref{assum1}, and the initial data $\rho_0\in\mc{P}_4(\ODD)$. Then there exists a unique process $\OX\in\mc{C}([0,T];\ODD)$, $T>0$ satisfying the mean-field dynamic \eqref{Xbareq} in strong sense with $\rho\in \mc{C}([0,T];\mc P_2(\ODD))$.
\end{thm}
\begin{proof}
We follow the proof of \cite[Theorem 3.1, Theorem 3.2]{carrillo2010asymptotic}.
For any given $u\in \mc{C}([0,T];\R^d)$, using \cite[Theorem 4.1]{tanaka1979stochastic} we may uniquely solve the following linear SDE
	\begin{subequations}\label{linMVeq}
		\begin{numcases}{}
d\OX_t=-\lambda(\OX_t-u_t)dt+\sigma D(\OX_t-u_t)dB_t-dL_t\,,\\
L_t=\int_0^t n(\OX_s)d|L|_s,\quad |L|_t=\int_0^t \mathbf{I}_{\partial\DD}(\OX_s)d|L|_s\,,
		\end{numcases}
	\end{subequations}
	with the initial data $\OX_0$ distributed according to $\rho_0\in \mc{P}_4(\ODD)$.  In particular
 $\sup_{t\in[0,T]}\EE[|\OX_t|^4]\leq K$ for some $K<\infty$ depends only on $\EE[|\OX_0|^4],T,\lambda$ and $\sigma$.
 Let us denote by $g_t=\mbox{Law}(\OX_t)\in \mc{P}(\ODD)$ then $g\in \mc{C}([0,T];\mc{P}_2(\ODD))$. Indeed, for any $t<s, t,s\in(0,T)$,  it follows from \cite[Theorem 3.1]{tanaka1979stochastic} we have
    \begin{align}
|\OX_t-\OX_s|^2\leq 2\int_s^t(X_\tau-X_s)\cdot[-\lambda(\OX_\tau-u_\tau)d\tau+\sigma D(\OX_\tau-u_\tau)dB_\tau]+\int_s^t\sigma^2|\OX_\tau-u_\tau|^2d\tau\,.
    \end{align}
Then  taking expectation on both sides implies
	\begin{align}
		\EE[|\OX_t-\OX_s|^2]& \leq 2\lambda \int_s^t\EE[|X_\tau-X_s|\cdot| \OX_\tau-u_\tau|]d\tau +\sigma^2\int_s^t\EE[|\OX_\tau-u_\tau|^2]d\tau\nn\\
  &\leq \lambda \int_s^t\EE[|\OX_\tau-u_\tau|^2]d\tau+(\sigma^2+\lambda)\int_s^t\EE[|\OX_\tau-u^1_\tau|^2]d\tau\nn\\
		&\leq C|t-s|\,,
	\end{align}
	for some $C>0$ depending only on $\lambda,\sigma,K$ and $\|u\|_\infty$. Thus we have $W_2(g_t,g_s)\leq C|t-s|^{\frac{1}{2}}$.
	Set
	\begin{align}
		X_\alpha(g_t)=\frac{\int_{\ODD}x\omega_{\alpha}^{\mc{E}}(x)g(t,dx)}{\int_{\ODD}\omega_{\alpha}^{\mc{E}}(x)g(t,dx)}\,,
	\end{align}
	which provides the self-mapping property of the map
	\begin{equation}
		\mc{T}: \mc{C}([0,T];\R^d)\to \mc{C}([0,T];\R^d) \mbox{ with } u\mapsto \mc{T}(u)=X_\alpha(g)
	\end{equation}
	for which we show to be compact later.
	
	 Applying Lemma \ref{lemLip} we obtain
	\begin{equation}
		|X_\alpha(g_t)-X_\alpha(g_s)|\leq c_0W_2(g_t, g_s)\leq c_0C|t-s|^{\frac{1}{2}}\,,
	\end{equation}
	which indicates the H\"{o}lder continuity of $t\mapsto X_\alpha(g_t)$ with the exponent $1/2$. This implies the compactness of $\mc{T}$ because of the compact embedding $\mc{C}^{1/2}([0,T];\R^d)\hookrightarrow \mc{C}([0,T];\R^d)$.
	Now let $u\in  \mc{C}([0,T];\R^d)$ satisfy $u=\xi \mc{T}(u)$ for some $\xi\in[0,1]$. In particular, there exists $g\in \mc{C}([0,T];\mc{P}_2(\ODD))$, i.e. the law of solution $\OX$ to SDE \eqref{linMVeq} such that
	$u=\xi X_\alpha(g)$.  Then according to Lemma \ref{lemXalpha}, we have for all $t\in(0,T)$, it holds
	$
		|u_t|^2=\xi^2|X_\alpha(g_t)|^2\leq \tau^2 c_1\int_{\ODD}|x|^2g_t(dx)\,.
	$
 Following similar computations as in \eqref{firstito}-\eqref{moment} it yields an estimate on $\sup_{t\in[0,T]}\int_{\ODD}|x|^2g_t(dx)=\sup_{t\in[0,T]}\EE[|\OX_t|^2]<\infty$. This implies $\|u\|_\infty < \infty$. 
	Finally, applying the Leray--Schauder fixed point theorem provides a fixed point $u$ for $\mc{T}$ thereby a solution to \eqref{Xbareq}.	As for the uniqueness, we can follow similar arguments in \cite[Theorem 3.1]{carrillo2010asymptotic} by using Lemma \ref{lemLip}. The details are omitted here.
\end{proof}

\subsection{Mean-field limit}
Now let $\{(\OX_t^i)_{t\geq 0}\}_{i=1}^N$ be $N$ independent copies of solutions to the mean-field dynamics \eqref{Xbareq}, so they are i.i.d. with the same distribution $\rho$. First, one can prove the following moment estimates for empirical measures.
\begin{lem}\label{lem:moment}
    Let $\TE$ satisfy Assumption \ref{assum1}, and suppose that  $\rho_0\in\mc{P}_{2p}(\ODD)$  for any $p\geq 2$.
    \begin{enumerate}
        \item Consider particle system \eqref{particle} with $\rho_0^{\otimes N}-$ distributed initial data, and let $\rho_t^N$ be the corresponding empirical measure. Then there exists some constant $c_2>0$ independent of $N$ such that it holds
        \begin{equation}\label{eq(1)}
        \sup_{t\in[0,T]}\left\{\sup_{i\in[N]}\EE[|X_t^i|^p]+\EE[\rho_t^N[|x|^p]]+\EE[|X_\alpha(\rho_t^N)|^p]\right\}<c_2\,.
        \end{equation}
        \item Consider particles  $\{(\OX_t^i)_{t\geq 0}\}_{i=1}^N$  with $\rho_0^{\otimes N}-$ distributed initial data, and let $\bar \rho_t^N$ be the corresponding empirical measure. Then there exists some constant $c_2>0$ independent of $N$ such that it holds 
              \begin{equation}\label{eq(2)}
        \sup_{t\in[0,T]}\left\{\sup_{i\in[N]}\EE[|\OX_t^i|^p]+\EE[\bar\rho_t^N[|x|^p]]+\EE[|X_\alpha(\bar \rho_t^N)|^p]\right\}<c_2\,.
        \end{equation}
    \end{enumerate}
\end{lem}
\begin{proof}
    Arguments for $(1)$ and $(2)$ are parallel. Here we only prove $(1)$ for the case $p=4$.  Indeed, for any fixed $i\in[N]$ applying It\^{o}'s formula leads to
\begin{align}\label{itodiff'}
	d |X_t^i-X_0^i|^4&=4|X_t^i-X_0^i|^2(X_t^i-X_0^i)\cdot[-\lambda(X_t^i-X_\alpha(\rho_t^{N}))dt\\ &+\sigma D(X_t^i-X_\alpha(\rho_t^{N}))dB_t^{i}-dL_t^i]\nn + 6\sigma^2|X_t^i-X_0^i|^2|X_t^i-X_\alpha(\rho_t^{N})|^2dt\,.
\end{align}
Using the convexity of $\DD$ and taking expectation on both sides one has 
\begin{align*}
    & d \EE[|X_t^i-X_0^i|^4] \\ 
    \leq & 3(\lambda+\sigma^2)\EE[|X_t^i-X_0^i|^4dt+(\lambda+3\sigma^2)\EE[|X_t^i-X_\alpha(\rho_t^{N})|^4]dt \nn\\
    \leq & 3(\lambda+\sigma^2)\EE[|X_t^i-X_0^i|^4dt+8(\lambda+3\sigma^2)\EE[|X_t^i|^4]dt 
    +8(\lambda+3\sigma^2)\EE[|X_\alpha(\rho_t^{N})|^4]dt\nn\\
    \leq & 3(\lambda+\sigma^2)\EE[|X_t^i-X_0^i|^4dt+8(\lambda+3\sigma^2)\EE[|X_t^i|^2]dt +8(\lambda+3\sigma^2)c_1\EE[(\rho_t^N[|x|^4])]dt\,,
\end{align*}
where in the last inequality we used Lemma \ref{lemXalpha}. {Notice that here It\^{o}'s integral term disappears, namely
\begin{equation}
    \EE\left[4|X_t^i-X_0^i|^2(X_t^i-X_0^i)\cdot \sigma D(X_t^i-X_\alpha(\rho_t^{N}))dB_t^{i}\right]=0
\end{equation}
because it satisfies,
\begin{equation}
    \EE\left[\int_0^t\left|4|X_s^i-X_0^i|^2(X_s^i-X_0^i)\cdot \sigma D(X_s^i-X_\alpha(\rho_s^{N}))\right|^2ds\right]\leq C\int_0^t\EE\left[|X_s^i|^8+|X_0^i|^8\right]ds<\infty\,,
\end{equation}
which can be guaranteed by the assumption $\rho_0\in \mc{P}_{8}(\ODD)$.}

Since particles are exchangeable, it implies that
\begin{equation}\label{2.26}
   \EE[\rho_t^N[|x|^4]]=\EE[|X_t^i|^4]\,.
\end{equation}
Thus we have
\begin{align*}
  d \EE[|X_t^i-X_0^i|^4]  \leq & 3(\lambda+\sigma^2)\EE[|X_t^i-X_0^i|^4dt+8(\lambda+3\sigma^2)(1+c_1)\EE[|X_t^i|^4]dt\nn\\
  \leq & 3(\lambda+\sigma^2)\EE[|X_t^i-X_0^i|^4dt+64(\lambda+3\sigma^2)(1+c_1)\EE[|X_t^i-X_0^i|^4]dt\\
  + &64(\lambda+3\sigma^2)(1+c_1)\EE[|X_0^i|^4]dt
\end{align*}
Gronwall's inequality gives that
\begin{equation}
    \sup_{t\in[0,T],i\in[N]}\EE[|X_t^i|^4]\leq C
\end{equation}
for some $C>0$ independent of $N$.  Then estimate \eqref{eq(1)} follows from \eqref{2.26} and Lemma \ref{lemXalpha}.
\end{proof}

Then we recall a large deviation bound estimate from \cite[Lemma 3.7]{gerber2023mean}:
\begin{lem}\label{lemLLN}
	Let  $\TE$  satisfy Assumption \ref{assum1}, and $\bar \rho_t^N$ be the empirical measure associated to the particles $\{(\OX_t^i)_{t\in[0,T]}\}_{i=1}^N$ satisfying \eqref{Xbareq} up to any time $T>0$, which are i.i.d. with common distribution $\rho$ satisfying $\sup_{t\in[0,T]}\int_{\ODD}|x|^4\rho_t(dx)<\infty$. Then there exists some constant $c_3>0$ depending only on $\TE,\alpha, T$ and $\sup_{t\in[0,T]}\int_{\ODD}|x|^4\rho_t(dx)$ such that
	\begin{equation}
		\sup\limits_{t\in[0,T]}\EE[|X_\alpha(\bar \rho_t^N)-X_\alpha(\rho_t)|^2]\leq c_3N^{-1}\,.
	\end{equation}
\end{lem}
Moreover, we will also need the following improved stability estimate \cite[Corollary 3.3]{gerber2023mean} on $X_\alpha(\rho)$:
\begin{lem}\label{lemimprove}
    Suppose $\TE$ satisfy Assumption \ref{assum1}. Then for all $R>0$ there exists some constant $L>0$ depending only on $L_\TE,R,\alpha$ such that
    \begin{align}
        |X_\alpha(\rho)-X_\alpha(\hat \rho)|\leq LW_2(\rho,\hat\rho),\quad \forall~(\rho,\hat\rho)\in \mc{P}_{2,R}(\ODD)\times \mc{P}_{2}(\ODD)\,,
    \end{align}
    where $\mc{P}_{2,R}(\ODD):=\{\rho\in \mc{P}_{2}(\ODD):~\int_{\ODD}|x|^2\rho(dx)\leq R\}$.
\end{lem}
Compared to Lemma \ref{lemLip}, here $L$ depends only on the second moment bound of measure $\rho$. This lemma leads to the following mean-field limit estimate:
\begin{thm}\label{thmmean}
	Assume that  $\DD$ satisfies Assumption \ref{assum} and $\TE$ satisfies Assumption \ref{assum1}. For any $T>0$, let $\{(X_t^i)_{t\in[0,T]}\}_{i=1}^N$ and $\{(\OX_t^i)_{t\in[0,T]}\}_{i=1}^N$ be the solutions to the interacting particle system \eqref{particle} and the mean-field dynamics \eqref{Xbareq} respectively up to time $T$ with the same initial data $\{(X_0^i)\}_{i=1}^N$ (i.i.d. distributed according to $\rho_0\in \mc{P}_{16}(\ODD)$) and Brownian motions $\{(B_t^i)_{t\in[0,T]}\}_{i=1}^N$. Then there exists some $C_{\mathrm{MFA}}>0$ depending only on $\lambda$, $\sigma$, $\alpha$, $T$, $\TE$ and $\sup_{t\in[0,T]}\EE[|\OX_t^1|^8]$ such that
		\begin{equation}
		\sup\limits_{t\in[0,T]}\sup\limits_{i=1,\dots,N}\EE[|X_t^i-\OX_t^i|^2]\leq C_{\mathrm{MFA}}N^{-1}\,.
	\end{equation}
\end{thm}
\begin{proof}
	Notice that
	\begin{align}
		d(X_t^i-\OX_t^i)&=-\lambda(X_t^i-\OX_t^i)dt+\lambda(X_\alpha(\rho_t^{N})-X_\alpha(\rho_t))dt\nn\\
		&+\sigma D(X_t^i-\OX_t^i)dB_t^{i}+\sigma D(X_\alpha(\rho_t)-X_\alpha(\rho_t^{N}))dB_t^{i}-dL_t^i+d\bar L_t^i\,,
	\end{align}
which by applying It\^{o}'s formula leads to
\begin{align}\label{itodiff}
	d(X_t^i-\OX_t^i)^2&=2(X_t^i-\OX_t^i)\cdot[-\lambda(X_t^i-\OX_t^i)dt+\lambda(X_\alpha(\rho_t^{N})-X_\alpha(\rho_t))dt]\nn\\
	&+2(X_t^i-\OX_t^i)\cdot[\sigma D(X_t^i-\OX_t^i)dB_t^{i}+\sigma D(X_\alpha(\rho_t)-X_\alpha(\rho_t^{N}))dB_t^{i}-dL_t^i+d\bar L_t^i]\nn\\
	&+\sigma^2[(X_t^i-\OX_t^i+(X_\alpha(\rho_t)-X_\alpha(\rho_t^{N}))^2dt]\,.
\end{align}
Using the fact that
\begin{align}
	(X_t^i-\OX_t^i)\cdot dL_t^i=(X_t^i-\OX_t^i)\cdot n(X_t^i)\mathbf{I}_{\partial\DD}(X_t^i)d|L^i|_t\geq 0\,, \\
	(X_t^i-\OX_t^i)\cdot d\bar L_t^i=(X_t^i-\OX_t^i)\cdot n(\bar X_t^i)\mathbf{I}_{\partial\DD}(\OX_t^i)d|\bar L^i|_t\leq 0\,,
\end{align}
and taking expectation on both sides of \eqref{itodiff} deduce
\begin{align}
	&d\EE[|X_t^i-\OX_t^i|^2]\leq -(\lambda-2\sigma^2)\EE[|X_t^i-\OX_t^i|^2]dt+(\lambda+2\sigma^2) \EE[|X_\alpha(\rho_t^{N})-X_\alpha(\rho_t)|^2]dt\nn\\
	\leq& -(\lambda-2\sigma^2)\EE[|X_t^i-\OX_t^i|^2]dt+2(\lambda+2\sigma^2) \EE[|X_\alpha(\rho_t^{N})-X_\alpha(\bar \rho_t^N)|^2]dt+2(\lambda+2\sigma^2) \EE[|X_\alpha(\bar \rho_t^{N})-X_\alpha(\rho_t)|^2]dt\nn\\
	\leq&-(\lambda-2\sigma^2)\EE[|X_t^i-\OX_t^i|^2]dt+2(\lambda+2\sigma^2) \EE[|X_\alpha(\rho_t^{N})-X_\alpha(\bar \rho_t^N)|^2]dt+2(\lambda+2\sigma^2) c_3\frac{1}{N}dt\,,
\end{align}
where in the third inequality we have used Lemma \ref{lemLLN}.  
We now bound the term 
$$\EE[|X_\alpha(\rho_t^{N})-X_\alpha(\bar \rho_t^N)|^4]\leq 8\EE[|X_\alpha(\rho_t^{N})|^4]+8\EE[|X_\alpha(\bar \rho_t^{N})|^4]\leq 16c_2\,.$$
In order to deal with the non-Lipschitz  property of the weighted mean $X_\alpha(\rho)$, we introduce for each $t\in[0,T]$ the event
\begin{equation}
    A_{N,t}:=\left\{\omega\in\Omega:~\frac{1}{N}\sum_{i=1}^N|\OX_t^i(\omega)|^2\geq R\right\}\,,
\end{equation}
where $R>\sup_{t\in[0,T]}\EE[|\OX_t^1|^2]$ is fixed.
Then according to \cite[Lemma 2.5]{gerber2023mean} we have the following bound
\begin{equation}
    {
{\mathbb P}(A_{N,t})\leq C N^{-2}\,,\quad \forall t\in[0,T]}
\end{equation}
holds if $\sup_{t\in[0,T]}\EE[|\OX_t^1|^8]<\infty$, which is guaranteed by Lemma \ref{lem:moment} under the assumption that $\rho_0\in \mathcal{P}_{16}(\ODD)$.
Then one splits the term
\begin{align}
    \EE[|X_\alpha(\rho_t^{N})-X_\alpha(\bar \rho_t^N)|^2]&=\EE[|X_\alpha(\rho_t^{N})-X_\alpha(\bar \rho_t^N)|^2\textbf{I}_{\Omega /A_{N,t}}]+\EE[|X_\alpha(\rho_t^{N})-X_\alpha(\bar \rho_t^N)|^2\textbf{I}_{A_{N,t}}]\nn\\
    &\leq \EE[|X_\alpha(\rho_t^{N})-X_\alpha(\bar \rho_t^N)|^2\textbf{I}_{\Omega /A_{N,t}}]+(\EE[|X_\alpha(\rho_t^{N})-X_\alpha(\bar \rho_t^N)|^4])^{1/2}({
\mathbb P}(A_{N,t}))^{1/2}\nn\\
    &\leq \EE[|X_\alpha(\rho_t^{N})-X_\alpha(\bar \rho_t^N)|^2\textbf{I}_{\Omega /A_{N,t}}]+4c_2^{1/2}C\frac{1}{N}\,.
\end{align}
Meanwhile, the improved stability Lemma \ref{lemimprove} concludes that
\begin{equation}
   \EE[|X_\alpha(\rho_t^{N})-X_\alpha(\bar \rho_t^N)|^2\textbf{I}_{\Omega /A_{N,t}}]\leq C\EE[W_2(\rho_t^N,\bar\rho_t^N)^2]\leq C\EE\left[\frac{1}{N}\sum_{i\in[N]}|X_t^i-\OX_t^i|^2\right]=C\EE[|X_t^i-\OX_t^i|^2]\,,
\end{equation}
which leads to
\begin{align}
    d\EE[|X_t^i-\OX_t^i|^2]\leq C\EE[|X_t^i-\OX_t^i|^2]dt+4c_2^{1/2}C\frac{1}{N}dt\,.
\end{align}
Thus applying Gronwall's inequality concludes the desired estimate.
\end{proof}

\section{Global convergence to the minimizer}\label{sec:globconv}
In this section, we present our main result about the global convergence in mean-field law for cost functions satisfying the following conditions.
\begin{assum}\label{assum2}
	Throughout this section we are interested in the objective function $\TE\in\mc{C}(\ODD)$, for which
	\begin{enumerate}
		\item there exists a unique $x^*\in\ODD$ such that $\TE(x^*)=\min\limits_{x\in\ODD}\TE(x)=:\underline \TE$.
		\item there exists some $\TE_\infty,R_0,\eta,\nu>0$ such that
		\begin{equation}
			\eta|x-x^*|\leq |\TE(x)-\underline \TE|^\nu \mbox{ for all }x\in B_{R_0}(x^*)\cap \ODD\,,
		\end{equation}
		\begin{equation}
			\TE_\infty<\TE(x)-\underline \TE \mbox{ for all } x\in (B_{R_0}(x^*))^c\cap \ODD\,.
		\end{equation}
	\item For any $q>0$, there exists some $r\in(0,R_0]$ such that
	\begin{equation}
		|\TE(x)-\underline \TE|\leq q,\quad \mbox{ for all }x\in B_r(x^*)\cap\ODD\,.
	\end{equation}
	\end{enumerate}
\end{assum}
In this section we define
\begin{equation}\label{eq:def_V}
\mc{V} (t):=\EE[|\OX_t-x^*|^2]\,,
\end{equation}
then one has the following lemma
\begin{lem}\label{lemV}
	The functional $\mc{V} (t)$ defined in \eqref{eq:def_V} satisfies
\begin{align}
	\frac{d \mc{V} (t)}{dt}
	\leq -(2\lambda-\sigma^2)\mc{V} (t)+2(\lambda+\sigma^2)\mc{V} (t)^{\frac{1}{2}}|x^*-X_\alpha(\rho_t)|+\sigma^2|x^*-X_\alpha(\rho_t)|^2
\end{align}
\end{lem}
\begin{proof}
It is easy to compute that
	\begin{align}
d|\OX_t-x^*|^2=2(\OX_t-x^*)dX_t+\sigma^2|\OX_t-X_\alpha(\rho_t)|^2dt\,.
\end{align}
which implies
\begin{align}
d \mc{V} (t)&=-2\lambda\EE[(\OX_t-x^*)\cdot(\OX_t-X_\alpha(\rho_t))]dt-2\EE[ (\OX_t-x^*)\cdot n(\OX_t)\mathbf{I}_{\partial\DD}(\OX_t)d|\OL|_t] \nn \\ \nonumber 
& + \sigma^2\EE[|\OX_t-X_\alpha(\rho_t )|^2]dt\nn \\ \nonumber 
& \leq -2\lambda\EE[(\OX_t-x^*)\cdot(\OX_t-X_\alpha(\rho_t))]dt + \sigma^2\EE[|\OX_t-X_\alpha(\rho_t )|^2]dt\nn\\
&=-2\lambda\mc{V} (t)-2\lambda\EE[(\OX_t-x^*)\cdot(x^*-X_\alpha(\rho_t))]+\sigma^2\EE[|\OX_t-X_\alpha(\rho_t)|^2]\nn\\
&\leq -(2\lambda-\sigma^2)\mc{V} (t)+2(\lambda+\sigma^2)\mc{V} (t)^{\frac{1}{2}}|x^*-X_\alpha(\rho_t)|+\sigma^2|x^*-X_\alpha(\rho_t)|^2\,,
\end{align}
where in the second inequality we have used the fact that
\begin{equation}\label{eq:favor}
-2 (\OX_t-x^*)\cdot n(\OX_t)\leq 0
\end{equation}
for all $\OX_t\in \partial  \DD$ and $x^*\in \ODD$, since $\DD$ is convex and $n(\OX_t)$ is the outward normal vector at $\OX_t$. This completes the proof.
\end{proof}

In the following, we shall use the notation for a $\ell^\infty$-ball with a radius $r$ and the center $x^*$, i.e. $B_r(x^*):=\{x:~|x-x^*|_\infty\leq r\}$. Let us quantify the maximum discrepancy for the objective function $\TE$ around the minimizer $x^*$, namely for $s>0$,
\begin{equation}\label{eqTEr}
	\TE_s:=\sup_{x\in B_s(x^*)\cap\ODD}|\TE(x)-\underline \TE|\,.
\end{equation}
\begin{proposition}\label{propX}
	Assume that $\TE$ satisfies Assumption \ref{assum2}. For any $t>0$ and $s\in (0,R_0]$ let $\TE_s$ be defined as in \eqref{eqTEr}, $0<q\leq \frac{\TE_\infty}{2}$, and $r:=\max\{s\in(0,R_0]:~\TE_s\leq q\}$
	then we have
	\begin{align}
		|x^*-X_\alpha(\rho_t)|\leq \frac{(2q)^\nu}{\eta}+\frac{\exp\left(-\alpha q\right)}{ \rho_t \left(\{x\in\ODD:x\in B_r(x^*)\}\right)} \int_{\ODD}|x-x^*|\rho_t(dx)\,.
	\end{align}
\end{proposition}
\begin{proof}
  Let $\tilde r\geq r>0$, and using Jensen's inequality one can deduce
   \begin{equation}
   	|x^*-X_\alpha(\rho_t)|\leq \int_{B_{\tilde r}(x^*)\cap\ODD}|x-x^*|\frac{\omega_\alpha(x)\rho_t(dx)}{\|\omega_\alpha(\cdot)\|_{L_1(\rho_t)}}+\int_{(B_{\tilde r}(x^*))^c\cap\ODD}|x-x^*|\frac{\omega_\alpha(x)\rho_t(dx)}{\|\omega_\alpha(\cdot)\|_{L_1(\rho_t)}}\,.
   \end{equation}
The first term is bounded by $\tilde r$ since $|x-x^*|\leq \tilde r$ for all $x\in B_{\tilde r}(x^*)$. Moreover, it follows from Markov's inequality that
\begin{align}
	\|\omega_\alpha(\cdot)\|_{L_1(\rho_t)}
	\geq& \exp(-\alpha(\TE_r+\underline \TE))\rho_t(\{x\in\ODD:\exp(-\alpha\TE(x))\geq \exp(-\alpha(\TE_r+\underline \TE))\})\nn\\
=&	\exp(-\alpha(\TE_r+\underline \TE))\rho_t(\{x\in\ODD: \TE(x)\leq (\TE_r+\underline \TE)\})\nn\\
\geq& \exp(-\alpha(\TE_r+\underline \TE))\rho_t (\{x\in \ODD:x\in B_r(x^*)\})\,.
\end{align}
Then for the second term we have
\begin{align}
&\int_{(B_{\tilde r}(x^*))^c\cap\ODD}|x-x^*|\frac{\omega_\alpha(x)\rho_t(dx)}{\|\omega_\alpha(\cdot)\|_{L_1(\rho_t)}}\nn\\
\leq&\frac{1}{ \exp(-\alpha(\TE_r+\underline \TE))\rho_t (\{x\in\ODD:x\in B_r(x^*)\})} \int_{(B_{\tilde r}(x^*))^c\cap\ODD}|x-x^*|\omega_\alpha(x)\rho_t(dx)\nn\\
\leq&\frac{\exp(-\alpha\min_{x\in (B_{\tilde r}(x^*))^c\cap\ODD}\TE(x))}{ \exp(-\alpha(\TE_r+\underline \TE))\rho_t (\{x\in\ODD:x\in B_r(x^*)\})} \int_{(B_{\tilde r}(x^*))^c\cap\ODD}|x-x^*|\rho_t(dx)\nn\\
=&\frac{\exp\left(-\alpha(\min_{x\in (B_{\tilde r}(x^*))^c\cap\ODD}\TE(x)-\TE_r-\underline \TE)\right)}{ \rho_t(\{x\in\ODD:x\in B_r(x^*)\})} \int_{(B_{\tilde r}(x^*))^c\cap\ODD}|x-x^*|\rho_t(dx)\,.
\end{align}
Thus for any $\tilde r\geq r>0$ we obtain
\begin{align}\label{14}
	& |x^*-X_\alpha(\rho_t)|
	\\ \nonumber \leq & \tilde r+\frac{\exp\left(-\alpha(\min_{x\in (B_{\tilde r}(x^*))^c\cap\ODD}\TE(x)-\TE_r-\underline \TE)\right)}{ \rho_t  (\{x\in\ODD:x\in B_r(x^*)\})} \int_{(B_{\tilde r}(x^*))^c\cap\ODD}|x-x^*|\rho_t(dx)\,.
\end{align}
Next we choose $\tilde r=\frac{(q+\TE_r)^\nu}{\eta}$, then it holds that
\begin{align}
	\tilde r=\frac{(q+\TE_r)^\nu}{\eta}\geq \frac{\TE_r^\nu}{\eta}=\frac{\left(\sup_{x\in B_r(x^*)\cap\ODD}|\TE(x)-\underline \TE|\right)^\nu}{\eta}\geq \sup_{x\in B_r(x^*)\cap\ODD}|x-x^*|=r
\end{align}
by using Assumption \ref{assum2}. Additionally one notice that
\begin{align}
	\min_{x\in (B_{\tilde r}(x^*))^c\cap\ODD}\TE(x)-\underline \TE\geq \begin{cases}
		\TE_\infty,\quad \mbox{ if }x\in (B_{R_0}(x^*))^c\cap(B_{\tilde r}(x^*))^c\cap\ODD\,,\\
		(\tilde r\eta)^{\frac{1}{\nu}},\quad \mbox{ if }x\in B_{R_0}(x^*)\cap(B_{\tilde r}(x^*))^c\cap\ODD\,.
	\end{cases}
\end{align}
Thus, since $(\tilde r\eta)^{\frac{1}{\nu}}=q+\TE_r\leq 2q\leq \TE_\infty$, we have
\begin{align}
	\min_{x\in (B_{\tilde r}(x^*))^c\cap\ODD}\TE(x)-\TE_r-\underline \TE\geq (\tilde r\eta)^{\frac{1}{\nu}}-\TE_r
	=q+\TE_r-\TE_r=q\,.
\end{align}
Inserting this and the definition of $\tilde r$ into \eqref{14}, we conclude the result.
\end{proof}

To eventually apply the above proposition, one needs to ensure that $ \rho_t (\{x\in\ODD:x\in B_r(x^*)\})=\PP(\OX_t\in  B_r(x^*)\cap \ODD)$ is
bounded away from $0$ for a finite time horizon $T$. To do so, we employ a rather technical argument inspired from \cite[Proposition 23]{fornasier2021consensus}, and introduce the mollifier $\phi_r^{x_0}:~ \R^d \to \R$, with $r > 0$ and $x_0\in \DD$ defined by
	\begin{align}\label{eqphi}
	\phi_{r}^{x_0}(x):= \begin{cases}\prod_{k=1}^{d} \exp \left(1-\frac{r^{2}}{r^{2}-\left(x-x_0\right)_{k}^{2}}\right), & \text { if } x\in B_r(x_0)\cap\DD  \\ 0, & \text { else }.\end{cases}
	\end{align}
So we have $\mbox{Im}(\phi_r) = [0, 1], \mbox{supp}(\phi_r) = B_r(x_0)\cap\DD,  \phi_r\in \mc{C}_c^\infty(\RR^{d})$ and
	\begin{align}
	\begin{aligned}
	\partial_{x_k} \phi_{r}^{x_0}(x) &=-2 r^{2} \frac{\left(x-x_0\right)_{k}}{\left(r^{2}-\left(x-x_0\right)_{k}^{2}\right)^{2}} \phi_{r}(x), \\
	\partial_{x_k^2}^{2} \phi_{r}^{x_0}(x) &=2 r^{2}\left(\frac{2\left(2\left(x-x_0\right)_{k}^{2}-r^{2}\right)\left(x-x_0\right)_{k}^{2}-\left(r^{2}-\left(x-x_0\right)_{k}^{2}\right)^{2}}{\left(r^{2}-\left(x-x_0\right)_{k}^{2}\right)^{4}}\right) \phi_{r}(x)\,.
	\end{aligned}
	\end{align}
 The proof will be split into two cases: 
\subsection{Global convergence proof when $\ODD$ is bounded} In this subsection we assume that $\ODD$ is bounded. Namely, there exists some diameter $0<|\DD|<\infty$ such that $|\DD| = \max_{x,y\in \ODD}|x-y|.$ First, we get the following result, whose proof is postponed to the Appendix.
\begin{proposition}\label{propositive1}
	Let  $(\OX_t)_{0\leq t\leq T}$ be the solution to \eqref{Xbareq} up to any time $T>0$. Then for any  $x_0\in\DD$ and $r>0$, there exists some constant $\vartheta>0$ depends only on $d,r,\lambda,\sigma$ and $|\DD|$ such that
	\begin{align}\label{propeq1}
		\PP(\OX_t\in  B_r(x_0)\cap \ODD)\geq \EE[\phi_r^{x_0}(\OX_0)]\exp(-\vartheta t)\,
	\end{align}
	holds for all $t\in[0,T]$, where $\phi_r^{x_0}$ is defined as in \eqref{eqphi}.
\end{proposition}
This result directly implies
\begin{cor}\label{propositive}
	Let  $(\OX_t)_{0\leq t\leq T}$ be the solution to \eqref{Xbareq} up to any time $T>0$. Then for any $r>0$, let $\tilde x^*\in \DD$ and $0<\tilde r\leq r$ be choosen such that $B_{\tilde r}(\tilde x^*)\subset B_{ r}( x^*)$. Then there exists some constant $\vartheta>0$ depends only on $d,\lambda,\sigma$ and $|\DD|$ such that
\begin{align}\label{propeq}
	\PP(\OX_t\in  B_r(x^*)\cap \ODD)\geq \EE[\phi_{\tilde r}^{\tilde{x}^*}(\OX_0)]\exp(-\vartheta t)\,
\end{align}
holds for all $t\in[0,T]$.  Especially when $x^*\in \DD$ one can choose $\tilde x^*=x^*$ and $\tilde r=r$.
\end{cor}
\begin{proof}
\textit{ Case 1}: if $x^*\in \DD$, then one may let $x_0=x^*$ in Proposition \ref{propositive1} and immediately gets
	\begin{align}
	\PP(\OX_t\in  B_r(x^*)\cap \ODD)\geq \EE[\phi_r^{x^*}(\OX_0)]\exp(-\vartheta t)\,.
\end{align}
\textit{ Case 2}: if $x^*\in \partial \DD$, then we can always find some $x_0=\tilde x^*\in \DD$ and $0<\tilde r<r$ such that $B_{\tilde r}(\tilde x^*)\subset B_{ r}( x^*)$. Using Proposition \ref{propositive1} again one yields
	\begin{align}
	\PP(\OX_t\in  B_r(x^*)\cap \ODD)\geq 	\PP(\OX_t\in  B_{\tilde r}(\tilde x^*)\cap \ODD)\geq \EE[\phi_{\tilde r}^{\tilde{x}^*}(\OX_0)]\exp(-\vartheta t)\,.
\end{align}
\end{proof}

Now we are ready to prove the global convergence result, which provides a rate for the variance function  $\mc{V} (t)=\EE[|\OX_t-x^*|^2]$  within a prescribed time-range.
\begin{thm}\label{thmconvergence}
	Assume that $\TE$ satisfies Assumption \ref{assum2}, and $\lambda,\sigma$ satisfy $2\lambda>\sigma^2$. Furthermore assume the initial data satisfies $\mc{V} (0)>2\varepsilon$ for any prescribed accuracy $\varepsilon>0$, and $\EE[\phi_{\tilde r}^{\tilde{x}^*}(\OX_0)]>0$ with $\tilde x^*, \tilde r$ to be determined later. Let $\tau\in(0,1)$ and choose $\alpha$ to be sufficiently large satisfying
	\begin{equation}
	\alpha\geq \frac{\vartheta T_\varepsilon-\log(\frac{1}{2}c_5\EE[\phi_{\tilde r}^{\tilde{x}^*}(\OX_0)])}{c_4}\,,
	\end{equation}
where $\vartheta$ comes from Corollary \ref{propositive}, 
$$T_\varepsilon:=\frac{1}{(1-\tau)(2\lambda-\sigma^2)}\log\left(\frac{\mc{V}(0)}{\varepsilon}\right),\quad  c_5:=\min\left\{\frac{\tau(2\lambda-\sigma^2)}{4(\lambda+\sigma^2)},\sqrt{\frac{\tau(2\lambda-\sigma^2)}{2\sigma^2}} \right\}\,,$$
and
 $$c_4:=\frac{1}{2}\min\left\{\left(\eta\frac{c_5\sqrt{\varepsilon}}{2}\right)^{\frac{1}{\nu}},\TE_\infty\right\},\quad
	r:=\left\{\max_{s\in(0,R_0]}:~\TE_s\leq c_4\right\}\,,
$$
and $\tilde x^*\in D$ and $0<\tilde r<r$ are choosen such that $B_{\tilde r}(\tilde x^*)\subset B_{ r}( x^*)$.
Then there exists some $0<T_\ast \leq T_\varepsilon$ such that $\mc{V} (t)$ satisfies
\begin{equation}\label{expdecay}
\mc{V} (t)\leq \mc{V} (0)\exp(-(1-\tau)(2\lambda-\sigma^2)t)\quad \mbox{ for all }t\in[0,T_\ast)\,.
\end{equation}
and it reaches the prescribed accuracy at time $T_\ast$, namely $\mc{V} (T_\ast)= \varepsilon$.
\end{thm}

	\begin{proof}
	We define 
		\begin{equation}
			T_\alpha:=\inf\left\{t\geq0:~\mc{V}(t)= \varepsilon\right\}\,,
		\end{equation}
	and for all $t\in[0,T_\alpha]$ let
		\begin{equation}
		C_\alpha(t):=c_5\sqrt{\mc{V} (t)}=\min\left\{\frac{\tau(2\lambda-\sigma^2)}{4(\lambda+\sigma^2)},\sqrt{\frac{\tau(2\lambda-\sigma^2)}{2\sigma^2}} \right\}\sqrt{\mc{V} (t)}\,.
		\end{equation}
		Here the time $T_\alpha$ represents the first time when the variance $\mc V(t)$ reaches the prescribed accuracy $\varepsilon$, and it depends on $\alpha$ because $\mc V$ does.
	It is obvious that $T_\alpha>0$ is because of the assumption on the initial data that $\mc{V}(0)>2\varepsilon$.
	According to the definition of $T_\alpha$ one has  $\mc{V} (t)>\varepsilon$ for all $t\in[0,T_\alpha)$ and $\mc{V}(T_\alpha)=\varepsilon$. Next we will prove that $T_\alpha\leq  T_\varepsilon$ and $\mc{V}(t)$ decreases exponentially on $[0,T_\alpha)$.
	
	\textbf{Case $T_\alpha\leq T_\varepsilon$}: It follows from Proposition \ref{propX} that for all $t\in[0,T_\alpha)$
	\begin{align}
		&|x^*-X_\alpha(\rho_{t})|\leq \frac{(2q_{t}^\alpha)^\nu}{\eta}+\frac{\exp\left(-\alpha q_{t}^\alpha\right)}{ \rho_{t} (\{x\in\ODD:x\in B_r(x^*)\})} \int|x-x^*|\rho_{t}(dx)\nn\\
		&\leq \frac{(2q_{t}^\alpha)^\nu}{\eta}+\frac{\exp\left(-\alpha q_{t}^\alpha\right)}{ \rho_{t} (\{x\in\ODD:x\in B_r(x^*)\})} \sqrt{\mc{V} (t)}
		\,,
	\end{align}
	where we have constructed
	\begin{equation}
		q_{t}^\alpha:=\frac{1}{2}\min\left\{ \left(\eta\frac{C_\alpha({t})}{2}\right)^{\frac{1}{\nu}},\TE_\infty\right\} \,.
	\end{equation}
	Then it holds
	\begin{equation}
		c_4=\frac{1}{2}\min\left\{ \left(\eta\frac{c_5\sqrt{\varepsilon}}{2}\right)^{\frac{1}{\nu}},\TE_\infty\right\} < q_{t}^\alpha\leq \frac{1}{2}\TE_\infty
	\end{equation}
	because $C_\alpha(t)>c_5\sqrt{\varepsilon}$.
	Moreover, here we let
	\begin{equation}
		r=\left\{\max_{s\in[0,R_0]}:~\TE_s\leq c_4\right\}\,.
	\end{equation}
	By construction, these choices satisfy $r\leq R_0$ and $\frac{(2q_{t}^\alpha)^\nu}{\eta}\leq \frac{C_\alpha({t})}{2}$.

	Moreover, $\rho_{t} (\{x\in\ODD:x\in B_r(x^*)\})=	\PP(\OX_{t}\in B_r(x^*)\cap\ODD)\geq \EE[\phi_{\tilde r}^{\tilde{x}^*}(\OX_0)]\exp(-\vartheta t )$ holds according to \eqref{propeq}, where $\vartheta$ depends only on $d,r,\lambda,\sigma$ and $|\DD|$.
	This concludes that for all $t\in[0,T_\alpha)$
	\begin{align}\label{329}
		|x^*-X_\alpha(\rho_{t} )|\leq \frac{C_\alpha({t})}{2}+\frac{\exp\left(-\alpha c_4\right)\exp\left( \vartheta T_\varepsilon\right)}{\EE[\phi_{\tilde r}^{\tilde{x}^*}(\OX_0)]}\sqrt{\mc{V} (t)} 
		\leq C_\alpha(t)\,,
	\end{align}
	where we choose $\alpha\geq\alpha_0$ with 
	\begin{equation}
		\alpha_0=\frac{\vartheta T_\varepsilon-\log(\frac{1}{2}c_5\EE[\phi_{\tilde r}^{\tilde{x}^*}(\OX_0)])}{c_4}\,.
	\end{equation}
		Let us recall the upper bound for the time derivative of $\mc{V} (t)$ given in Lemma \ref{lemV}:
\begin{align}
	\frac{d \mc{V} (t)}{dt}
	\leq-(2\lambda-\sigma^2)\mc{V} (t)+2(\lambda+\sigma^2)\mc{V} (t)^{\frac{1}{2}}|x^*-X_\alpha(\rho_t)|+\sigma^2|x^*-X_\alpha(\rho_t)|^2\,.
\end{align}
	Then  using the definition of $c_5$ and the estimate from \eqref{329} one can deduce that
		\begin{equation}
		\frac{d \mc{V} (t)}{dt}\leq -(1-\tau)(2\lambda-\sigma^2)\mc{V} (t)\,,
		\end{equation}
		which by Gronwall's inequality leads to
		\begin{equation}\label{decrease}
		\mc{V} (t)\leq \mc{V} (0)\exp(-(1-\tau)(2\lambda-\sigma^2)t)\mbox{ for }t\in[0,T_{\alpha})\,.
		\end{equation}
		
\textbf{Case $T_\varepsilon<T_\alpha$:} By the definition of $T_\alpha$ we know  $\mc{V} (t)>\varepsilon$ for all $t\in[0,T_\varepsilon]$.
Then following the same argument as in the first case, one can conclude
		\begin{equation}\label{decrease1}
	\mc{V} (t)\leq \mc{V} (0)\exp(-(1-\tau)(2\lambda-\sigma^2)t)\mbox{ for }t\in[0,T_{\varepsilon}]\,.
\end{equation}
Using the this estimate the fact that $T_\varepsilon=\frac{1}{(1-\tau)(2\lambda-\sigma^2)}\log(\frac{\mc{V}(0)}{\varepsilon})$ implies $\mc{V}(T_\varepsilon)\leq \varepsilon$, which is a contradiction. Thus this case can never happen.
		\end{proof}
		
\subsection{Global convergence proof when $\ODD$ is unbounded}		
In this subsection we consider the case where $\DD$ is unbounded, namely a finite diameter $|\DD|$ does not exist.  Therefore, Proposition \ref{propositive1} and Corollary \ref{propositive} need some adjustments.
Firstly, following the proof of  Proposition \ref{propositive1} it is easy to get
\begin{proposition}\label{propositive1'}
	Let  $(\OX_t)_{0\leq t\leq T}$ be the solution to \eqref{Xbareq} up to any time $T>0$. For any fixed $x_0\in \DD$, assume that $\sup_{t\in[0,T]}|x_0-X_\alpha(\rho_t)|\leq B$ for some $B>0$.
	Then for any  $r>0$, there exists some constant $\vartheta>0$ depends only on $d,r,\lambda,\sigma$ and $B$ such that
	\begin{align}\label{propeq1'}
		\PP(\OX_t\in  B_r(x_0)\cap \ODD)\geq \EE[\phi_r^{x_0}(\OX_0)]\exp(-\vartheta t)\,
	\end{align}
	holds for all $t\in[0,T]$, where $\phi_r^{x_0}$ is defined as in \eqref{eqphi}.
\end{proposition}
\begin{proof}
Compared to Proposition \ref{propositive1}, the additional assumption on the upper bound $B$ compensates for the lack of an upper bound on $|\DD|$.
The proof is basically identical to the one of Proposition \ref{propositive1} except we shall bound $|(\OX_t-X_\alpha(\rho_t))_k|$ in \eqref{eqdiff} differently  as follows
	\begin{equation*}
	|(\OX_t-X_\alpha(\rho_t))_k|\leq 	|(\OX_t-x_0)_k|+	|(x_0-X_\alpha(\rho_t))_k|\leq \sqrt{c}r+B\,.
	\end{equation*}
\end{proof}
This implies the following result immediately similar to the bounded domain case.
\begin{cor}\label{propositive'}
Let  $(\OX_t)_{0\leq t\leq T}$ be the solution to \eqref{Xbareq} up to any time $T>0$. For any $r>0$, let $\tilde x^*\in D$ and $0<\tilde r<r$ be choosen such that $B_{\tilde r}(\tilde x^*)\subset B_{ r}( x^*)$. Assume that  $\sup_{t\in[0,T]}|\tilde x^*-X_\alpha(\rho_t)|\leq B$ for some $B>0$. Then
there exists some constant $\vartheta>0$ depends only on $d,r,\lambda,\sigma$ and $B$ such that
	\begin{align}\label{propeq'}
		\PP(\OX_t\in  B_r(x^*)\cap \ODD)\geq \EE[\phi_{\tilde r}^{\tilde{x}^*}(\OX_0)]\exp(-\vartheta t)\,
	\end{align}
	holds for all $t\in[0,T]$. Especially when $x^*\in \DD$ one can choose $\tilde x^*=x^*$ and $\tilde r=r$.
\end{cor}

Now we are ready to prove the global convergence result as  in Theorem \ref{thmconvergence}
\begin{thm}\label{thmconvergence'}
	Assume that $\TE$ satisfies Assumption \ref{assum2}, and $\lambda,\sigma$ satisfy $2\lambda>\sigma^2$. Furthermore, assume the initial data satisfies $\mc{V} (0)>2\varepsilon$ for any prescribed accuracy $\varepsilon>0$, and $\rho_{0}(\{x\in\ODD:x\in B_{r_0}(x^*)\})>0$,  $\EE[\phi_{\tilde r}^{\tilde{x}^*}(\OX_0)]>0$ with $\tilde x^*, \tilde r$ to be determined later. Let $\tau\in(0,1)$ and choose $\alpha$ to be sufficiently large satisfying
	\begin{equation}
		\alpha\geq \max\left\{\frac{-\log\left(\frac{1}{2}c_5\rho_{0}(\{x\in\ODD:x\in B_{r_0}(x^*)\})\right)}{q_0}, \frac{\vartheta T_\varepsilon-\log(\frac{1}{2}c_5\EE[\phi_{\tilde r}^{\tilde{x}^*}(\OX_0)])}{c_4}\right\}\,,
	\end{equation}
	where $\vartheta$ comes from Proposition \ref{propositive'}, 
	$$T_\varepsilon:=\frac{1}{(1-\tau)(2\lambda-\sigma^2)}\log\left (\frac{\mc{V}(0)}{\varepsilon}\right ),\quad  c_5:=\min\left\{\frac{\tau(2\lambda-\sigma^2)}{4(\lambda+\sigma^2)},\sqrt{\frac{\tau(2\lambda-\sigma^2)}{2\sigma^2}} \right\}\,,$$
	$$q_{0}:=\frac{1}{2}\min\left\{ \left(\eta\frac{c_5\sqrt{\mc{V} (0)}}{2} \right)^{\frac{1}{\nu}},\TE_\infty\right\} ,\quad c_4:=\frac{1}{2}\min\left\{\left(\eta\frac{c_5\sqrt{\varepsilon}}{2}\right)^{\frac{1}{\nu}},\TE_\infty\right\},\quad
	\,
	$$
		$$r_0:=\left\{\max_{s\in(0,R_0]}:~\TE_s\leq q_0\right\},\quad r:=\left\{\max_{s\in(0,R_0]}:~\TE_s\leq c_4\right\}\,,
	$$
	and $\tilde x^*\in \DD$ and $0<\tilde r<r$ are chosen such that $B_{\tilde r}(\tilde x^*)\subset B_{ r}( x^*)$.
	Then there exists some $0<T_\ast \leq T_\varepsilon$ such that $\mc{V} (t)$ satisfies
	\begin{equation}\label{expdecay'}
		\mc{V} (t)\leq \mc{V} (0)\exp(-(1-\tau)(2\lambda-\sigma^2)t)\quad \mbox{ for all }t\in[0,T_\ast)\,.
	\end{equation}
	and it reaches the prescribed accuracy at time $T_\ast$, namely $\mc{V} (T_\ast)= \varepsilon$.
\end{thm}

\begin{proof}
	We define 
	\begin{equation}
		T_\alpha:=\inf\left\{t\geq0:~\mc{V}(t)= \varepsilon\quad \mbox{or}\quad |x^*-X_\alpha(\rho_t)|=2c_5 \sqrt{\mc{V} (0)}\right\}\,,
	\end{equation}
	and for all $t\in[0,T_\alpha]$ define
	\begin{equation}
		C_\alpha(t):=c_5\sqrt{\mc{V} (t)}=\min\left\{\frac{\tau(2\lambda-\sigma^2)}{4(\lambda+\sigma^2)},\sqrt{\frac{\tau(2\lambda-\sigma^2)}{2\sigma^2}} \right\}\sqrt{\mc{V} (t)}\,.
	\end{equation}
	Here the time $T_\alpha$ represents the first time when the variance $\mc V(t)$ reaches the prescribed accuracy $\varepsilon$ or $|x^*-X_\alpha(\rho_t)|=2c_5 \sqrt{\mc{V} (0)}$, and it depends on $\alpha$ because $\mc V$ does.

	Now we show that $T_\alpha>0$. 
	 It follows from Proposition \ref{propX} that 
	\begin{align}
		&|x^*-X_\alpha(\rho_{0})|\leq \frac{(2q_{0})^\nu}{\eta}+\frac{\exp\left(-\alpha q_{0}\right)}{ \rho_{0} (\{x\in\ODD:x\in B_r(x^*)\})} \int|x-x^*|\rho_{0}(dx)\nn\\
		&\leq \frac{(2q_{0})^\nu}{\eta}+\frac{\exp\left(-\alpha q_{0}\right)}{ \rho_{0} (\{x\in\ODD:x\in B_r(x^*)\})} \sqrt{\mc{V} (0)}
		\,,
	\end{align}
	where we have constructed
	\begin{equation}
		q_{0}:=\frac{1}{2}\min\left\{ \left(\eta\frac{c_5\sqrt{\mc{V} (0)}}{2}\right)^{\frac{1}{\nu}},\TE_\infty\right\} \,.
	\end{equation}
	Then it holds $q_{0}\leq \frac{1}{2}\TE_\infty$. Moreover, here we let
	\begin{equation}
		r_0=\left\{\max_{s\in[0,R_0]}:~\TE_s\leq q_0\right\}\,.
	\end{equation}
	By construction, these choices satisfy $r_0\leq R_0$ and $\frac{(2q_{0})^\nu}{\eta}\leq \frac{c_5\sqrt{\mc{V} (0)}}{2}$. This means that
	\begin{equation}
|x^*-X_\alpha(\rho_{0})|\leq \frac{c_5\sqrt{\mc{V} (0)}}{2}+\frac{\exp\left(-\alpha q_{0}\right)}{ \rho_{0} (\{x\in\ODD:x\in B_{r_0}(x^*)\})} \sqrt{\mc{V} (0)}\,.
	\end{equation}
	If now we choose $\alpha$ sufficiently large, e.g., $\alpha\geq \alpha_0$ with
	\begin{equation*}
		\alpha_0=\frac{-\log\left(\frac{1}{2}c_5\rho_{0}(\{x\in\ODD:x\in B_{r_0}(x^*)\})\right)}{q_0}
	\end{equation*}
	then one has
	\begin{equation*}
		|x^*-X_\alpha(\rho_{0})|\leq c_5\sqrt{\mc{V} (0)}.
	\end{equation*}
	This together with $\mc{V}(0)\geq 2\varepsilon$ implies $T_\alpha>0$.

	According to the definition of $T_\alpha$ one has  
	\begin{equation*}
		\mc{V} (t)>\varepsilon\quad \mbox{ and }\quad |x^*-X_\alpha(\rho_{t})|< 2c_5\sqrt{\mc{V} (0)},\quad \mbox{for all } t\in[0,T_\alpha)\,,
	\end{equation*}
	and at $t=T_\alpha$, it holds 	$\mc{V}(T_\alpha)=\varepsilon$ or $|x^*-X_\alpha(\rho_{T_\alpha})|=2c_5\sqrt{\mc{V} (0)}$. Next we will prove that $T_\alpha\leq  T_\varepsilon$ and $\mc{V}(t)$ decreases exponentially on $[0,T_\alpha)$.
	
	\textbf{Case $T_\alpha\leq T_\varepsilon$}: It follows from Proposition \ref{propX} that for all $t\in[0,T_\alpha)$
	\begin{align}
		|x^*-X_\alpha(\rho_{t})|\leq \frac{(2q_{t}^\alpha)^\nu}{\eta}+\frac{\exp\left(-\alpha q_{t}^\alpha\right)}{ \rho_{t} (\{x\in\ODD:x\in B_r(x^*)\})} \sqrt{\mc{V} (t)}
		\,,
	\end{align}
	where we have constructed
	\begin{equation}
		q_{t}^\alpha:=\frac{1}{2}\min\left\{ \left(\eta\frac{C_\alpha({t})}{2}\right)^{\frac{1}{\nu}},\TE_\infty\right\} \,.
	\end{equation}
	Then it holds
	\begin{equation}
		c_4=\frac{1}{2}\min\left\{ \left(\eta\frac{c_5\sqrt{\varepsilon}}{2}\right)^{\frac{1}{\nu}},\TE_\infty\right\} < q_{t}^\alpha\leq \frac{1}{2}\TE_\infty
	\end{equation}
	because $C_\alpha(t)>c_5\sqrt{\varepsilon}$.
	Moreover here we let
	\begin{equation}
		r=\left\{\max_{s\in[0,R_0]}:~\TE_s\leq c_4\right\}\,.
	\end{equation}
	By construction, these choices satisfy $r\leq R_0$ and $\frac{(2q_{t}^\alpha)^\nu}{\eta}\leq \frac{C_\alpha({t})}{2}$.
	
	Next we choose $\tilde x^*\in D$ and $0<\tilde r<r$ such that $B_{\tilde r}(\tilde x^*)\subset B_{ r}( x^*)$. Then we have
	\begin{equation*}
		\sup_{t\in \in[0,T_\alpha)}|\tilde x^*-X_\alpha(\rho_t)|\leq |\tilde x^*-x^*|+\sup_{t\in \in[0,T_\alpha)}|x^*-X_\alpha(\rho_t)|<r+2c_5\sqrt{\mc{V} (0)}=:B\,.
	\end{equation*}
		Moreover according to \eqref{propeq'} it holds that 
		$$\rho_{t} (\{x\in\ODD:x\in B_r(x^*)\})=	\PP(\OX_{t}\in B_r(x^*)\cap\ODD)\geq \EE[\phi_{\tilde r}^{\tilde{x}^*}(\OX_0)]\exp(-\vartheta t )\,,$$ where $\vartheta$ depends only on $d,r,\lambda,\sigma$ and $B$.
	This concludes that for all $t\in[0,T_\alpha)$
	\begin{align}\label{329'}
		|x^*-X_\alpha(\rho_{t} )|\leq \frac{C_\alpha({t})}{2}+\frac{\exp\left(-\alpha c_4\right)\exp\left( \vartheta T_\varepsilon\right)}{\EE[\phi_{\tilde r}^{\tilde{x}^*}(\OX_0)]}\sqrt{\mc{V} (t)} 
		\leq C_\alpha(t)\,,
	\end{align}
	where we choose $\alpha\geq\alpha_0$ with 
	\begin{equation}
		\alpha_0=\frac{\vartheta T_\varepsilon-\log(\frac{1}{2}c_5\EE[\phi_{\tilde r}^{\tilde{x}^*}(\OX_0)])}{c_4}\,.
	\end{equation}
	Let us recall the upper bound for the time derivative of $\mc{V} (t)$ given in Lemma \ref{lemV}:
	\begin{align}
		\frac{d \mc{V} (t)}{dt}
		\leq-(2\lambda-\sigma^2)\mc{V} (t)+2(\lambda+\sigma^2)\mc{V} (t)^{\frac{1}{2}}|x^*-X_\alpha(\rho_t)|+\sigma^2|x^*-X_\alpha(\rho_t)|^2\,.
	\end{align}
	Then  using the definition of $c_5$ and the estimate from \eqref{329'} one can deduce that
	\begin{equation}
		\frac{d \mc{V} (t)}{dt}\leq -(1-\tau)(2\lambda-\sigma^2)\mc{V} (t)\,,
	\end{equation}
	which by Gronwall's inequality leads to
	\begin{equation}\label{decrease'}
		\mc{V} (t)\leq \mc{V} (0)\exp(-(1-\tau)(2\lambda-\sigma^2)t)\,, \quad \mbox{for} \ t\in[0,T_{\alpha})\,.
	\end{equation}
	This implies
	\begin{equation*}
	|x^*-X_\alpha(\rho_{T_\alpha} )|\leq C_\alpha(T_\alpha)=c_5\sqrt{\mc{V}(T_\alpha)}\leq c_5\sqrt{\mc{V}(0)}\,.
	\end{equation*}
	Then, by the definition of $T_\alpha$ we must have $\mc{V}(T_\alpha)=\varepsilon$.
	
	\textbf{Case $T_\varepsilon<T_\alpha$:} By the definition of $T_\alpha$ we know  $\mc{V} (t)>\varepsilon$ and $|x^*-X_\alpha(\rho_{t})|< 2c_5\sqrt{\mc{V} (0)}$ for all $t\in[0,T_\varepsilon]$.
	Then following the same argument as in the first case, one can conclude
	\begin{equation}\label{decrease1'}
		\mc{V} (t)\leq \mc{V} (0)\exp(-(1-\tau)(2\lambda-\sigma^2)t)\,, \quad \mbox{for} \ t\in[0,T_{\varepsilon}]\,.
	\end{equation}
	Using the this estimate the fact that $T_\varepsilon=\frac{1}{(1-\tau)(2\lambda-\sigma^2)}\log(\frac{\mc{V}(0)}{\varepsilon})$ implies $\mc{V}(T_\varepsilon)\leq \varepsilon$, which is a contradiction. Thus this case can never happen.
\end{proof}

\subsection{Quantitative convergence result for the
numerical scheme}
We will use the Euler--Maruyama scheme from \cite{slominski2001euler} to solve the particle system \eqref{particle} numerically.
For this purpose we define the {
orthogonal projection operator
\begin{equation}
	\Pi_{\ODD}(x):=\arg\min_{z\in \ODD}|x-z|\,, \quad \text{for} \ x \in \R^d\,.
\end{equation}
 Note that for a given $x \in \mathbb R^d$, the projection $\Pi_{\ODD}(x)$ exists and is uniquely determined due to the convexity of $\ODD$.} Given a time horizon $T>0$ and a time discretization $t_0=0<\Delta t<\dots<K\Delta t=T$ of $[0,T]$. Then the Euler--Maruyama scheme amounts to:
\begin{align}\label{numericeq}
	X_{(k+1)\Delta t}^i&=\Pi_{\ODD}\Big(X_{k\Delta t}^i-\Delta t\lambda(X_{k\Delta t}^i-X_\alpha(\rho_{k\Delta t}^{N}))+\sigma D(X_{k\Delta t}^i-X_\alpha(\rho_{k\Delta t}^{N}))N^i(0,\Delta t)\Big)\notag\\
	X_0^i&\sim \rho_0,\quad i=1,\dots,N\,,
\end{align}
where $\{N^i(0,\Delta t)\}_{i=1}^N$ are independent Gaussian random vectors with zero mean and covariance matrix $\Delta t \textbf{Id}_{d}$. Now collecting results from Theorem \ref{thmmean} and Theorem \ref{thmconvergence} we can establish a quantitative convergence result for the numerical scheme \eqref{numericeq}. We do so by controlling the following discrete counterpart of \eqref{eq:def_V}:
\begin{equation}\label{eq:def_VN}
    \mathcal{V}_N(t):=W_2^2(\rho_{t}^N, \delta_{x^*})=\EE\left[\left|\frac{1}{N}\sum_{i=1}^NX_{t}^i-x^*\right|^2\right]\,, \quad t=k\Delta t, k\in[K]\,.
\end{equation}
{
We are now ready to prove the main result of this paper (cf.~Theorem \ref{thmmain0}):} 
\begin{thm}\label{thmmain}
Under the assumptions of Theorem \ref{thmmean} and Theorem \ref{thmconvergence} (or Theorem \ref{thmconvergence'}), let $\{(X_{k\Delta t}^i)_{k=1,\dots,K}\}_{i=1}^N$ be the iterations generated by Euler--Maruyama scheme \eqref{numericeq} with $K\Delta t =T_*$, where $T_*$ comes from Theorem \ref{thmconvergence} (or Theorem \ref{thmconvergence'}) such that $\mc{V}(T_*)=\varepsilon$ for any prescribed accuracy $\varepsilon>0$. Then the final iterations fulfill the following quantitative error estimate 
\begin{align}
	\mathcal{V}_N\left(T_*\right)\leq 3C_{\mathrm{NA}}\Delta t\log(1/\Delta t)+3C_{\mathrm{MFA}}\frac{1}{N}+3\varepsilon\,,
\end{align}
where $C_{\mathrm{MFA}}$ comes from Theorem \ref{thmmean}, and $C_{\mathrm{NA}}$ depends on $\lambda$, $\sigma$, $\alpha$, $d$, $T_*$, $N$ and $\TE$.
\end{thm}
\begin{proof}
	Recall that $\{(X_t^i)_{t\in[0,T_*]}\}_{i=1}^N$ and $\{(\OX_t^i)_{t\in[0,T_*]}\}_{i=1}^N$ are are solutions to the CBO particle system \eqref{particle} and $N$ independent copies of solutions to the mean-field dynamics \eqref{Xbareq} up to time $T_*=K\Delta t$ respectively. Then
	we split the error
	\begin{align}\label{eqerror}
	\mathcal{V}_N\left(T_*\right)
		\leq & 3 \EE\left[\left|\frac{1}{N}\sum_{i=1}^N(X_{K\Delta t}^i-X_{T_*}^i)\right|^2\right]+3 \EE\left[\left|\frac{1}{N}\sum_{i=1}^N(X_{T_*}^i-\OX_{T_*}^i)\right|^2\right] \\ \nonumber
        + & 3 \EE\left[\left|\frac{1}{N}\sum_{i=1}^N\OX_{T_*}^i-x^*\right|^2\right]\,,
	\end{align}
	which divide the overall error into an approximation error of the Euler scheme, the mean-field limit estimate error and the optimization error in mean-field law. The first term on the right hand side of \eqref{eqerror} can be  estimated by applying the result from \cite[Theorem 3.2]{slominski2001euler}, which yields 
	\begin{equation*}
		\EE\left[\left|\frac{1}{N}\sum_{i=1}^N(X_{K\Delta t}^i-X_{T_*}^i)\right|^2\right]\leq C_{\mathrm{NA}}\Delta t\log(1/\Delta t)\,.
	\end{equation*}
	The second term can be bounded by using estimate from Theorem \ref{thmmean}, which establishes
	\begin{equation*}
		\EE\left[\left|\frac{1}{N}\sum_{i=1}^N(X_{T_*}^i-\OX_{T_*}^i)\right|^2\right]\leq C_{\mathrm{MFA}}\frac{1}{N}\,.
	\end{equation*}
	Finally the third term follows from Theorem \ref{thmconvergence} or Theorem \ref{thmconvergence'}, and it holds
	\begin{equation*}
	\EE\left[\left|\frac{1}{N}\sum_{i=1}^N\OX_{T_*}^i-x^*\right|^2\right]\leq \varepsilon\,.
	\end{equation*}
\end{proof}

\section{Numerical Experiments and Applications}\label{sec:num}

In this section, we present numerical experiments, which are performed in Python, partially based on CBXpy \cite{Bailo_CBX_Python_and_2024}, and are available for reproducibility at \href{https://github.com/echnen/CBO-with-boundaries}{https://github.com/echnen/CBO-with-boundaries}. Let us introduce now the rationale of these experiments.
\\

{
The success of the CBO algorithm lies in its unique blend of exploration and exploitation. It achieves exploration through the diffusion of a large number of particles, while exploitation is facilitated by harnessing collective information to form a consensus point that approximates the global minimizer. The synergy of these effects—exploration through particle diffusion and exploitation via the consensus point—becomes increasingly powerful as the number of particles grows. Theorem \ref{thmmain} guarantees that, as the number of particles \( N \) increases, the algorithm converges in expectation at a rate of \( \mathcal{O}(N^{-1}) \). This indicates that the accuracy improves significantly with more particles. However, the constants in the error bound \eqref{eqerror}, especially \( C_{\mathrm{MFA}} \), can increase exponentially with certain problem parameters, such as dimensionality. Consequently, it is challenging to directly apply these convergence results when using the CBO algorithm \eqref{numericeq} with a limited number of particles \( N \).  In the few-particle regime, the effectiveness of the CBO algorithm exploration and exploitation mechanisms diminishes. To ensure successful convergence even with fewer particles, we must augment these mechanisms with specialized heuristics. These enhancements should enable the algorithm to keep dynamics similar to those observed with a large number of particles, thereby preserving its effectiveness despite the reduced particle count.\\

In this section, we introduce two key strategies to enhance the exploration and exploitation capabilities of the CBO algorithm in a few-particle regime. 
\begin{itemize}
\item [1.] {\bf Enhanced Exploitation:} To improve exploitation, we leverage the insight that the consensus point typically remains within an exponentially shrinking region around the global minimizer. Indeed, combining  \eqref{329} or \eqref{329'}, for which 
$$
|x^*-X_\alpha(\rho_{t} )| \lesssim \sqrt{\mathcal V(t)}
$$
with the exponential decay in \eqref{expdecay} or \eqref{expdecay'}
$$
\mc{V} (t)\leq \mc{V} (0)\exp(-(1-\tau)(2\lambda-\sigma^2)t)\quad \mbox{ for all }t\in[0,T_\ast)\,.
$$
we conclude that there is an exponentially shrinking {\it trust-region} $\{x \in \mathbb{R}^d: |x- X_\alpha(\rho_{t})| \lesssim \exp(-(1-\tau)(\lambda-\sigma^2/2)t) \}$ in which we can assume that particles are moving with high-probability. Based on this observation, even with fewer particles, we will---optimistically---constrain them to stay within progressively shrinking balls centered around the consensus-point, following a carefully planned shrinking schedule. Admittedly, such optimism would be misplaced if few particles would not explore enough, but this is counteracted by the Heuristic 2. which enforces at the same time an enhanced exploration\footnote{Another reasoning supporting Heuristic 1. is that trust-region methods \cite{ConnGouldToint2000}, which are connected to CBO as explained in the Appendix, are known to converge (at least to local minimizers). While in trust-region methods the search is based on a quadratic interpolation and a guess-and-accept/reject mechanism for the trust radius, in CBO the search is based on stochastic exploration and the definition of consensus-point, which determine the trust-region. We refer to the Appendix for more insights.}.
\item [2.] {\bf Enhanced Exploration:} We propose a heuristic that boosts exploration by increasing the volatility in the noise. Specifically, we allow the noise variance \(\sigma^2\) to exceed \(2 \lambda\), contrary to the conditions set by Theorem \ref{thmconvergence}. This adjustment aims to compensate for the reduced number of particles by providing greater variability in their movements.
\end{itemize}
These enhancements are designed to ensure that the dynamics of the algorithm remain effective and similar to those observed with a larger number of particles, even when operating in a few-particle regime.\\

To our knowledge, these two enhancement mechanisms—boosting exploration through increased noise volatility and enhancing exploitation by constraining particles to shrinking regions around the global minimizer—have not been explored in the existing literature. However, they enable us to obtain numerical results that were previously unobserved in previous numerical experiments, see, e.g., \cite{borghi2021constrained,BORGHI2023113859,pinnau2017consensus, totzeck2020consensus}. 

For example, it has been well-documented that the standard CBO algorithm struggles to minimize the Rastrigin function in higher dimensions (e.g., \( d \geq 20 \)), regardless of how reasonably large is the number of particles used. Remarkably, our enhanced CBO implementation surpasses this dimensionality barrier, demonstrating consistent convergence with a moderate number of particles. This breakthrough suggests that our approach can significantly improve the algorithm performance in high-dimensional optimization problems. We illustrate these findings in Section \ref{sec:numerical_experiments} and Section \ref{sec:performance_over_nncvx_domains}. \\

Moreover, as noted in \cite{carrillo2019consensus,kalise2022consensus,fornasier2021convergence}, incorporating anisotropic noise or jump process noise can significantly enhance exploration in the few-particle regime, especially for high-dimensional optimization problems. This observation implies that adaptively shaping the noise based on the specific application is a crucial heuristic for optimizing the performance of the CBO algorithm. Tailoring the noise to suit the problem at hand can provide additional benefits, particularly when dealing with complex, high-dimensional landscapes.

Below, we will demonstrate how these adaptive noise strategies, combined with our previously discussed enhancements, contribute to improved convergence and robustness of the CBO algorithm even with a reduced number of particles. In particular, in Section \ref{sec:pde} we explore the efficient computation of global minimizers of 1D $p$-Allen--Cahn energies with boundary conditions and additional obstacle constraints. The problem is solved by combining a hierarchical approximation and noise shaping within a multigrid finite element method. Given the existence of continua of (local) minimizers of the  $p$-Allen--Cahn energy, the variational problems is extremely challenging. Yet, our hierarchical CBO solver does converge robustly to the global minimizer by using very few particles, despite the high-dimensionality of the optimization problem and, very importantly, without starting from initial data in the vicinity to the sought solution. To our knowledge, this is the first example in the literature of use of CBO for solving problems in scientific computing. These very promising results suggest that CBO could be employed to solve other challenging nonlinear equations, certainly a very interesting direction for future research.
}

\subsection{Effect of the parameters}\label{sec:numerical_experiments}

In this section, we explore the effect of the parameters for the convergence of \eqref{numericeq} considering a standard nonlinear multimodal function for testing, i.e., the Rastrigin function:
\begin{equation}
    R(x):=10 d + \sum_{i=1}^d (x_i^2 - 10\cos(2\pi x_i))\,, \quad \text{for} \ x:=(x_1,\dots, x_d)\in \R^d\,,
\end{equation}
which we aim to minimize over the convex and nonconvex domains depicted in Figure \ref{fig:domains}.

\begin{figure}
\includegraphics[width=\linewidth]{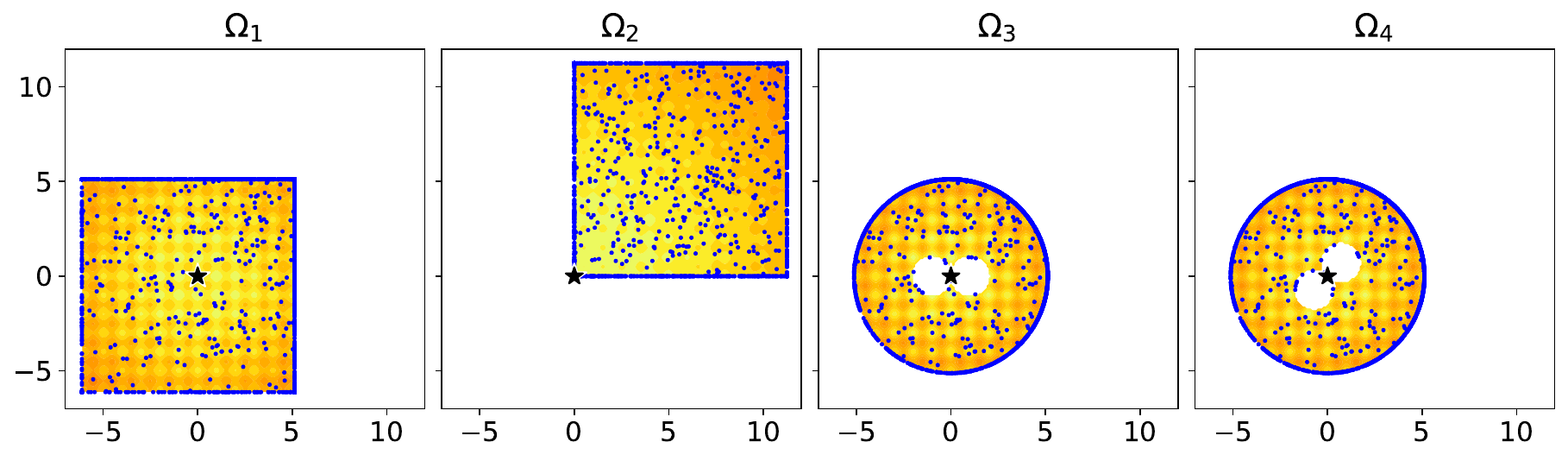}
\caption{Domains and initial particles considered in Section \ref{sec:numerical_experiments} in dimension $d=2$. The contoured orange shadow represents the Rastrigin function.}
    \label{fig:domains}
\end{figure}

In our numerical implementation, three key heuristics have been applied. The first, which is quite common, involves using an increasing sequence of parameters $\alpha$ generated by the update rule:
\begin{equation}\label{eq:increasing_alpha}
\alpha_{k} := \alpha_0 + \frac{k}{K} (\alpha_K - \alpha_0)\,, \quad \text{for all} \ k \in \{0, \dots, K\}\,,
\end{equation}
where $K\in \N$ is the final iteration of the method, and $\alpha_0, \alpha_K>0$ are starting and final parameters defined by the user. If not specified, in this work, we set $\alpha_K = 10^9$, while $\alpha_0$ is decided on a case-by-case basis. The next two are the two novel enhancements introduced in this paper, which we detail below.
\begin{heuristics}[Enhanced Exploitation]\label{heuristics:ball}
At each iteration number $k\in \N$, we consider a ball centered on the consensus point $X_\alpha(\rho_{k \Delta t}^N)$ with radius $R := \gamma \max_i\{|X^i_{k \Delta t} - X_\alpha(\rho_{k \Delta t}^N)|\}$ with $\gamma \in (0, 1]$ and project every particle $X^i_{(k+1)\Delta t}$ onto this ball. 
\end{heuristics}
It is important to note that the parameter $\gamma$ should be considered as an additional hyperparameter and should be tuned by the user. The second heuristic enhances particle exploration by allowing the noise variance $\sigma^2$ to exceed the theoretical bound $2\lambda$:
\begin{heuristics}[Enhanced Exploration]\label{heuristics:exploration}
	Given a drift parameter $\lambda>0$, we exceed the admissible noise level $\sigma = \sqrt{2\lambda}$ given by Theorem \ref{thmconvergence} by setting $\sigma = S\sqrt{2\lambda}$ for $S > 1$. 
\end{heuristics}
We anticipate that, based on our experiments, a good choice for $S$ is $5$, independently on the dimension of the problem. Higher values of $S$ can also be considered, especially when used in conjunction with Heuristic \ref{heuristics:ball}, as this combination helps mitigate excessive exploration away from the consensus point.

\begin{figure}[H]
	\centering
    \includegraphics[width=\linewidth]{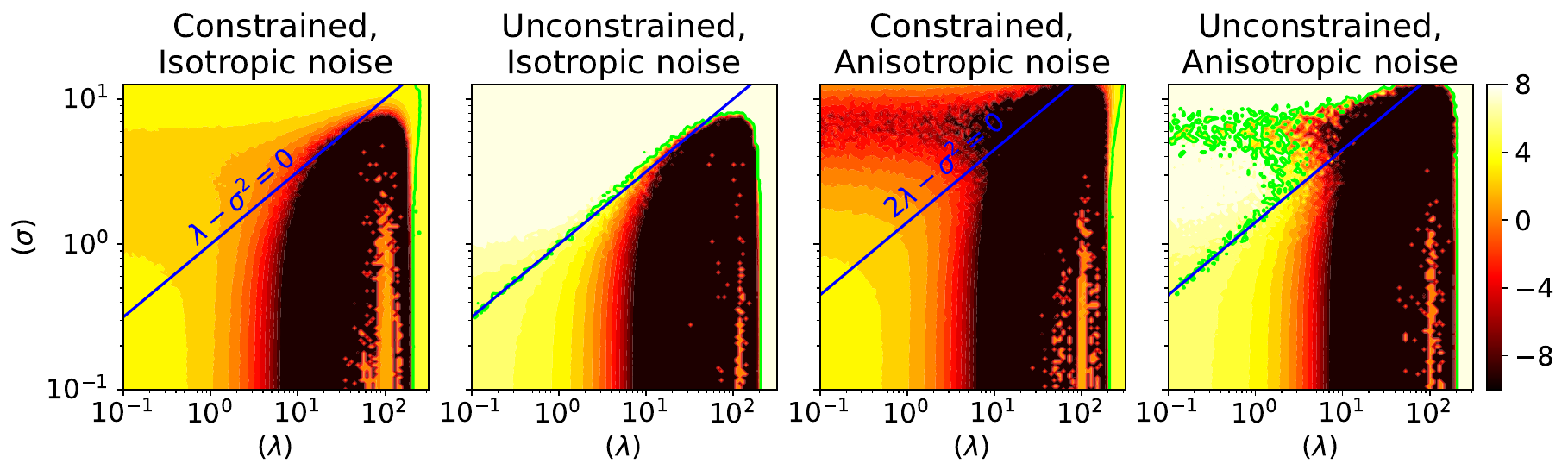}
    \caption{Logarithm of the residual to the solution at final iteration, i.e., $\ln\left(\hat{\mathcal{V}}_N(T)\right)$ for $\hat{\mathcal{V}}_N(t)$ defined as in \eqref{eq:def_VNhat} for different parameter's choices. The green line indicates the values of $\lambda$ and $\sigma$ for which the final residual coincides with the initial residual to the solution $\ln\left(\hat{\mathcal{V}}_N(0)\right)$ and therefore separates the (numerical) converging regime with the non-converging one. Confer Section \ref{sec:effect_lambda_and_sigma} for comments.}
    \label{fig:dependency_on_lambda_and_sigma}
\end{figure}

\subsubsection{The effect of $\lambda$ and $\sigma$}\label{sec:effect_lambda_and_sigma}

In this section, we investigate the effect of the parameters $\lambda$ and $\sigma$ in the presence of constraints. We consider\footnote{Note that the tiltedness of the domain is intentional.} $\Omega_1 = [-6.12, 5.12]^2$, fix $N=1000$, $\alpha_0 = \alpha_K =10^6$, $\Delta t = 10^{-2}$, and pick $100$ choices of $\lambda$ and $\sigma$ spaced evenly in a logarithmic scale from $10^{-1}$ to $10^{2.5}$ and from $10^{-1}$ to $10^{1.1}$, respectively. We initialize the particles sampling from a Gaussian distribution with mean $(5.12, 5.12) / \sqrt{2}$ and variance $10$, in this way the initial discrete distribution is likely to cover the whole domain while not being centered on the optimal solution, but actually on a local (non global) minimizer. We apply a further projection to the domain for those particles that fell outside $\Omega$. Once the initial measure is fixed, for each $\lambda$ and $\sigma$ we let the algorithm run for $100$ iterations and keep track of the error at the last iteration. This is computed in each iteration $t$ by an empirical approximation of $\mathcal{V}_N(t)$ according to \eqref{eq:def_VN}, namely,
\begin{equation}\label{eq:def_VNhat}
    \hat{\mathcal{V}}_N(t):=\frac{1}{N}\sum_{i=1}^N |X^i_{t} - x^*|^2\,, 
\end{equation}
where $x^*\in \R^d$ is the optimal solution, in this case $x^*=0$. For the sake of comparison, we run exactly the same experiment with $\Omega =\R^2$, i.e., simply dropping the additional projection onto the bounded domain and by considering the same initialization without the further projection onto $\Omega$. Additionally, we repeat the same experiment with isotropic noise, for which the critical line (at least for unbounded domains) is $2 \lambda - d \sigma^2=0$, compare, e.g., \cite{fornasier2021consensus}. We plot all the logarithmic value of these errors in Figure \ref{fig:dependency_on_lambda_and_sigma}.

We learn that utilizing anisotropic noise and introducing constraints do in fact influence the convergence of the CBO scheme allowing for larger diffusion parameters. Indeed, in Figure \ref{fig:dependency_on_lambda_and_sigma} we observe that in these cases, the critical line $2\lambda - \sigma^2=0$ is no longer sharp. This interesting behaviour, which gives a first numerical evidence of the effectiveness of Heuristic \ref{heuristics:exploration}, is completely unexpected and suggests that the analysis in these settings could be further refined. We leave these intriguing questions for a future work.

\subsubsection{Exploitation and Exploration: A Powerful Combination}\label{sec:effect_of_gamma} In this experiment, we show that the combination of Heuristics \ref{heuristics:ball} and \ref{heuristics:exploration}, can indeed be beneficial in practice especially in higher dimensions. Specifically, we consider dimensions ranging in $d = 2, 15, 20$ and consider Heuristics \ref{heuristics:exploration} with $S=5$, i.e., we make the non-feasible parameter choice:
\begin{equation}\label{eq:pars_non_feasibile}
    \lambda = 1\,, \quad \text{and}\quad \sigma = 5 \sqrt{2\lambda}\,.
\end{equation}
Note that these values fall outside the range of admissible parameters in Theorem \ref{thmconvergence}. To test the influence of the parameter $\gamma$ in Heuristics \ref{heuristics:ball}, we consider $N=1000$, $\alpha_0=10^6$, $\Delta t = 10^{-2}$. We set a maximum of $1000$ iterations, i.e., we consider a time horizon of $10$s and implement the increasing parameter heuristic as in \eqref{eq:increasing_alpha}. In this experiment, we considered two possible scenarios:
\begin{equation*}
    \Omega_1 := [-6.12, 5.12]^d\,, \quad \text{and} \quad \Omega_2:=[0, 11.24]^d\,.
\end{equation*}

\begin{figure}
    \centering
    \includegraphics[width=\linewidth]{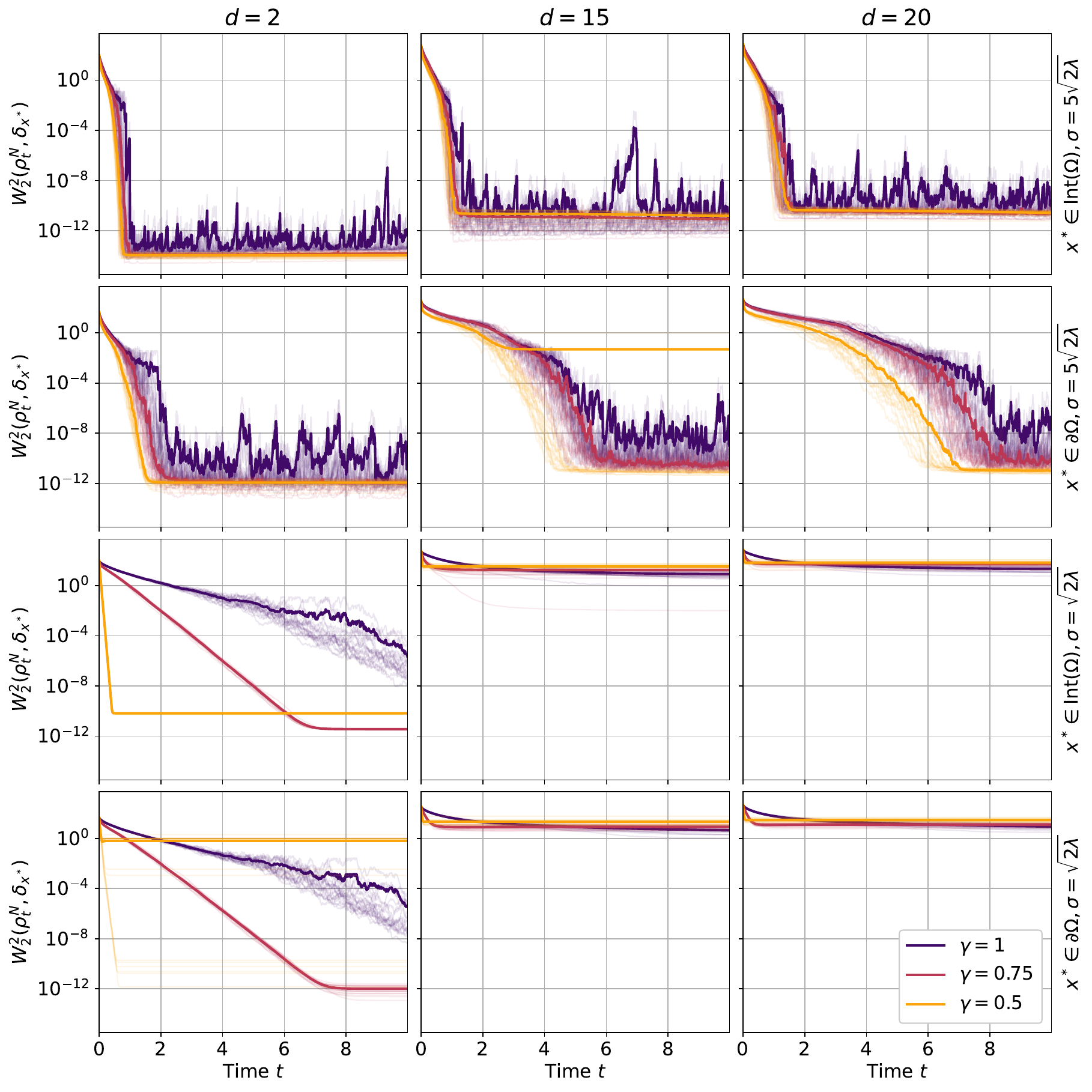}
    \caption{Error decrease as a function of time for the dimensions: $d=2, 15, 20$. In each case, we run the algorithm $20$ times and display the average in thicker lines. In the first and second row, Heuristics 1. and 2. are {\it active} for $\sigma = 5 \sqrt{2 \lambda}$ and $\gamma < 1$, while in the third and fourth row we show the behavior for the critical value $\sigma = \sqrt {2 \lambda}$, as in standard CBO.}
    \label{fig:middle_scale}
\end{figure}

Note that in the latter, the global minimizer lies on the boundary of $\Omega$. We initialize the particles sampling from a Gaussian with mean $(5.12, \dots, 5.12) / \sqrt{d}$, and variance $10$. Some initial particles might fall out of $\Omega$: We fix this by projecting all of these onto $\Omega$. We run the algorithm $20$ times in each case and show the results in the two top rows of Figure \ref{fig:middle_scale}. For the sake of comparison, we also repeat the experiment with a parameter choice that is indeed admissible in the sense of Theorem \ref{thmconvergence}, namely, $\sigma = \sqrt{2\lambda}$ and $\lambda=1$, and show the results in the two bottom rows of Figure \ref{fig:middle_scale}.

In this example, we observe that Heuristic \ref{heuristics:ball} has a beneficial effect on the performance of the proposed method. Specifically, as demonstrated in the top two rows of Figure \ref{fig:middle_scale}, it can significantly attenuate oscillations, which is a particularly desirable effect in optimization \cite{acfr20}. Additionally, in other cases, such as in the bottom row of Figure \ref{fig:middle_scale} (left graphic), it can markedly speed up convergence. It is also worth noting that we did not encounter any instance where Heuristics \ref{heuristics:ball} with $\gamma \geq 1/2$ deteriorated the convergence performance of the method. However, we did not present cases with $\gamma < 1/2$ as those indeed performed poorly.

To further corroborate our numerical heuristics, we run Algorithm \eqref{numericeq0} in the same setting, fixing the dimension to $d=5$ and choosing $\Omega=\R^d$. We consider $20$ values of $S$ evenly spaced in $[0.5, 10]$, first applying Heuristic \ref{heuristics:ball} ($\gamma=0.75$) and then disabling it ($\gamma=1$). For each value of $S$, we repeat the experiment five times and report the final residual at time $T=10$ in Figure \ref{fig:testing_heuristics}. The figure illustrates the positive impact of Heuristic \ref{heuristics:exploration}, which achieves its best performance for $S \simeq 5$, in agreement with the choice made in \eqref{eq:pars_non_feasibile}.

\begin{figure}
\centering
    \includegraphics[width=\linewidth]{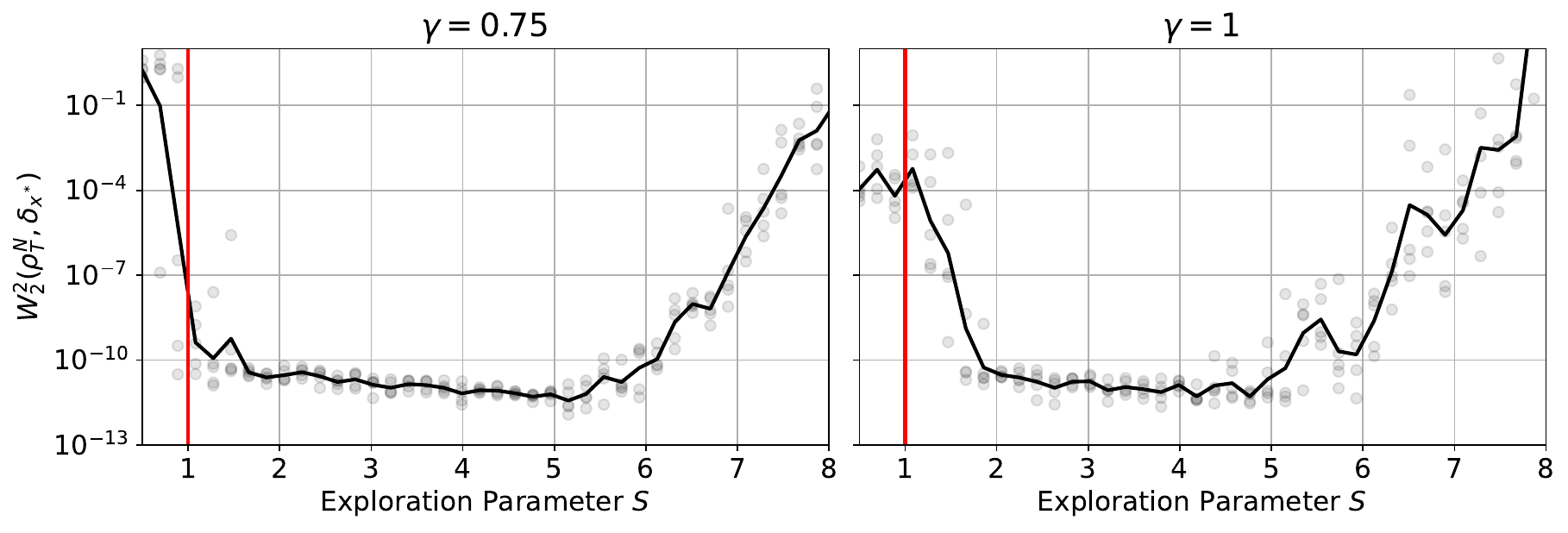}
    \caption{Final residual as a function of the exploration parameter $S$ in Heuristic \ref{heuristics:exploration}, first with Heuristic \ref{heuristics:ball} enabled and then without it. The vertical red line marks the critical stability threshold at $S=1$.}
    \label{fig:testing_heuristics}
\end{figure}

Eventually, we consider a relatively difficult high dimensional problem with $d = 100$. We consider the constraint $\Omega := \Omega_2$ and initialize the particles as above by sampling from a Gaussian distribution with mean $(5.12,\dots, 5.12) / \sqrt{d}$ and variance $10$ and projecting onto the domain. We run $20$ independent experiments by considering all enhancements: Increasing parameters $\alpha$ as in \eqref{eq:increasing_alpha}, Heuristic \ref{heuristics:ball} with $\gamma = 0.95$ and Heuristic \ref{heuristics:exploration} with $S=5$, and show the residual decrease as a function of the iterations in Figure \ref{fig:large_scale}. For comparison, we also show the performance in the unconstrained setting but employing Heuristic \ref{heuristics:ball} and Heuristic \ref{heuristics:exploration}, and in the unconstrained setting without employing any heuristic. While the latter does not converge to an optimal solution, in the first two cases the parameter choice \eqref{eq:pars_non_feasibile} allows us to optimize the Rastrigin function in $d=100$ dimensions, a result never obtained so far in the literature with CBO.

\begin{figure}
    \includegraphics[width=\linewidth]{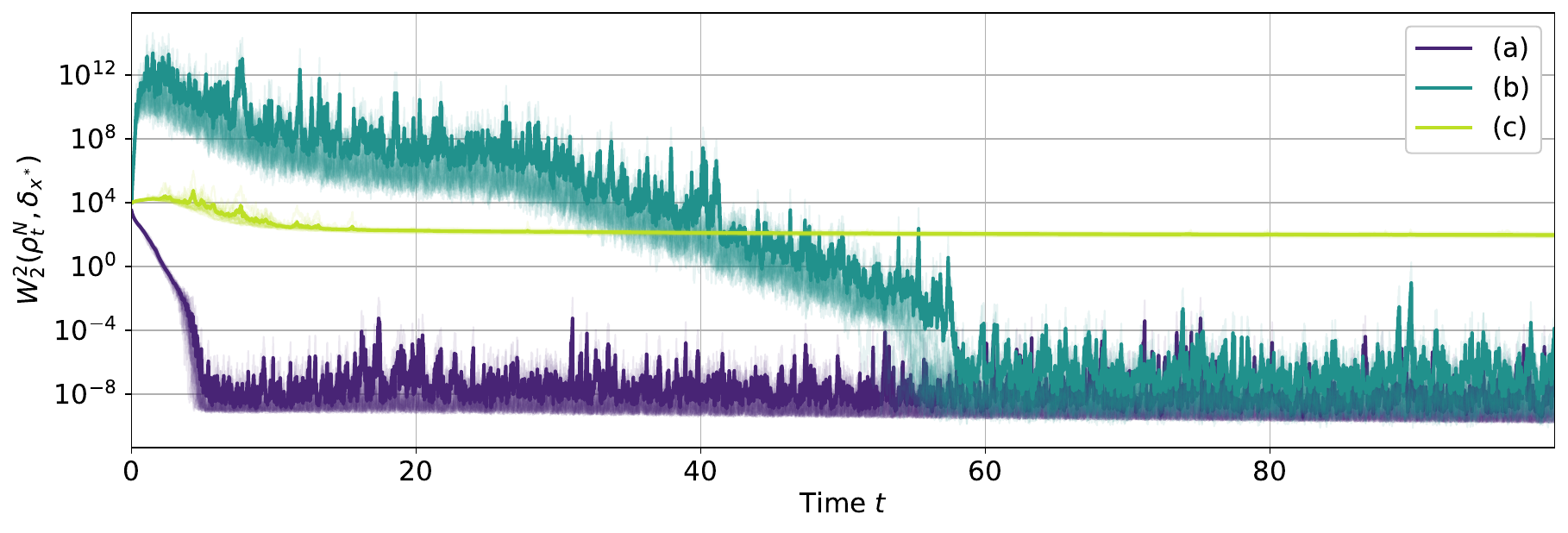}
    \caption{Error decrease as a function of time for $d=100$ with parameters \eqref{eq:pars_non_feasibile}. In each case, we run the algorithm $20$ times and display the average with a thicker line. Curves (a) and (b) represent the $\Omega=\Omega_2$ and $\Omega=\R^d$ cases employing Heuristics \ref{heuristics:ball} and \ref{heuristics:exploration}, while (c) represents the standard CBO for $\Omega = \R^d$.}
    \label{fig:large_scale}
\end{figure}

\subsubsection{Quantitative comparisons with respect to $N$}
\label{sec:quantitative_comparison}

We compare the performance of the standard CBO method ($\lambda=1$, $\sigma=\sqrt{2\lambda}$), CBO with enhanced exploration ($\lambda=1$, $\sigma=5\sqrt{2\lambda}$, $\gamma=1$), and CBO with both enhanced exploration and enhanced exploitation ($\lambda=1$, $\sigma=5\sqrt{2\lambda}$, $\gamma=0.75$). As a benchmark, we include a particle swarm optimization (PSO) algorithm based on \cite{kennedy1995particle}, which is implemented as part of the PySwarms package \cite{miranda2018pyswarms}. Unless specified otherwise, parameters are as chosen in Section \ref{sec:effect_lambda_and_sigma}. 
We study the performance of all methods using increasing particle numbers $N \in \{1,2,4,\ldots, 2^{14}\}$, applied to the Rastrigin problem in dimensions $d=5,10,15,20$ on the domain $\Omega_1$. 
The particles of each method are initialized and constrained to remain within the domain. 
Each simulation is run $20$ times, and performance is measured by the mean discrete Wasserstein distance of all runs evaluated after 1000 iterations. 
A low score on this measure indicates that all runs perform well, while a single poor result skews the performance measure towards a high score.
The results are displayed in Figure \ref{fig:quantitative_comparison}. 

While the overall performance of all methods improves with the number of particles, none of the algorithms reliably solves
the optimization problem with $N \leq 128$ particles.  
The CBO variants with enhanced exploration effectively handle the optimization problem when using more than $N=1024$ particles, with the variant with enhanced exploitation ($\gamma=0.75$) outperforming all other methods. 
The standard CBO performs well for $d=5$ and $N \geq 2048$, but its performance deteriorates in higher dimensions.
The PSO algorithm, on the other hand, struggles to reliably solve the optimization problem across all scenarios, while all considered CBO variants yield better or at least similar results. 

\begin{figure}
    \includegraphics[width=\linewidth]{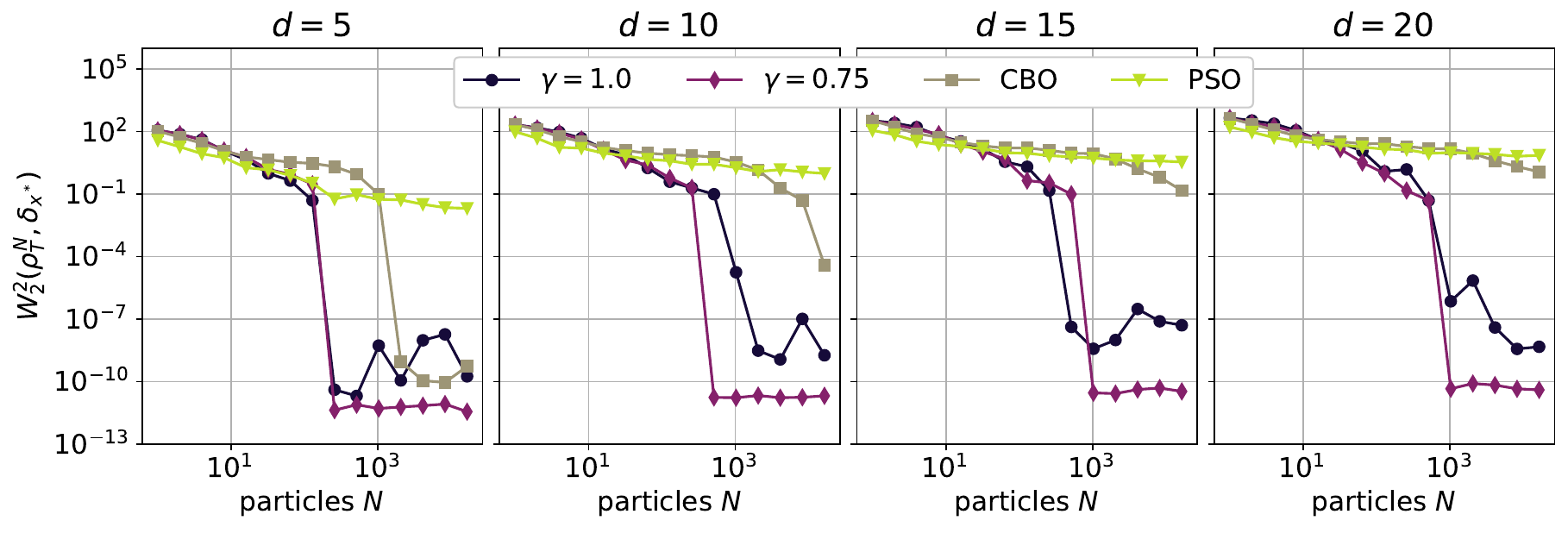}
    \caption{
        Quantitative comparison of CBO variants and PSO. In each case, we run the algorithm $20$ times and display the average of the discrete Wasserstein distance after $1000$ iterations. 
    }
    \label{fig:quantitative_comparison}
\end{figure}

\subsection{Performance over nonconvex domains}\label{sec:performance_over_nncvx_domains}

Motivated by the excellent results achieved, we decided to test the method further by exploring nonconvex domains, despite the current lack of theoretical backing. Specifically, for each dimension $d=2, 15, 20$ we consider two cases:
\begin{equation*}
    \Omega_{i + 2} := B_r(0) \setminus \left(B_{1}(\bar{x}_{i})\cup B_{1}(-\bar{x}_i)\right)\,, \quad i=1,2\,,
\end{equation*}
As shown in Figure \ref{fig:nncvx_domain}, the proposed method is able to deliver excellent performances also in the case of a nonconvex domain with shape $\Omega_3$. However, we believe that the anisotropic bias of the noise along axes makes this case particularly lucky and likely explains the significant difference in Figure \ref{fig:nncvx_domain} between optimizing on $\Omega_3$ or $\Omega_4$. This intriguing result suggests that optimal performance might be achievable only with an \textit{adaptive covariance} of the noise. Unfortunately, we have yet to model, implement, and analyze a version of CBO with adaptive covariance, but some techniques in this direction, such as those in \cite{burger2023covariancemodulated}, seem to open the door to such variations also for CBO diffusion.
\\

{
In the next section we move away from standard benchmarks, such as the Rastrigin function, and we present results of use of CBO for solving partial differential equations with constraints. }

\begin{figure}
    \includegraphics[width=\linewidth]{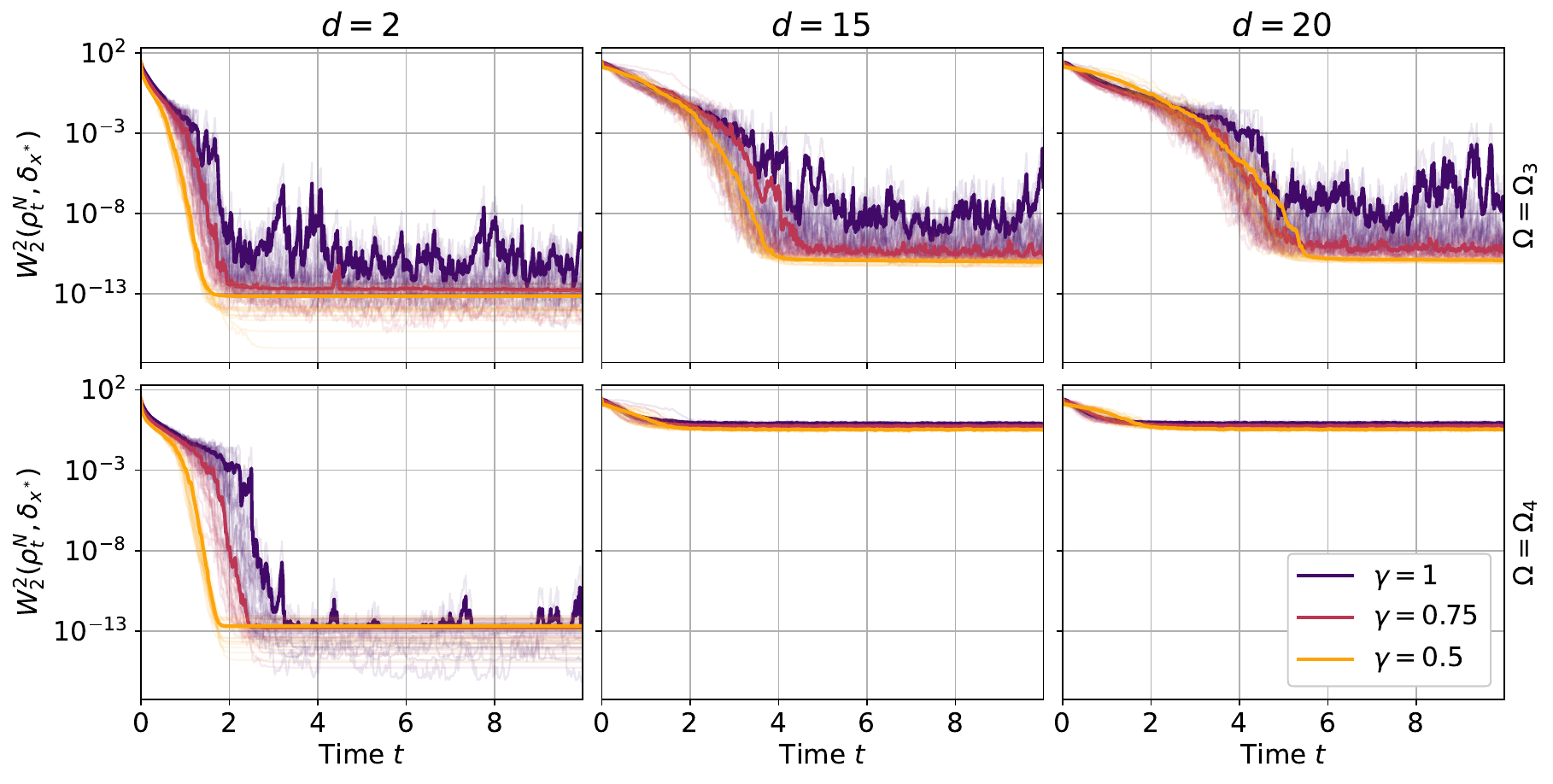}
    \caption{Residual decrease as a function of the iterations over two nonconvex domains, for details confer to Section \ref{sec:performance_over_nncvx_domains}.}
    \label{fig:nncvx_domain}
\end{figure}

\subsection{Solving a PDE problem with an obstacle} \label{sec:pde}

\begin{figure}
    \begin{subfigure}{0.3\linewidth}
        \includegraphics[width=\linewidth]{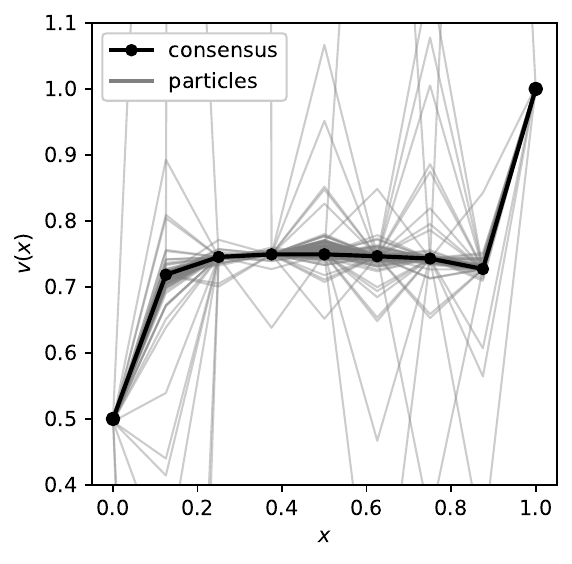}
        \caption{$i=3$}
    \end{subfigure}
    \begin{subfigure}{0.3\linewidth}
        \includegraphics[width=\linewidth]{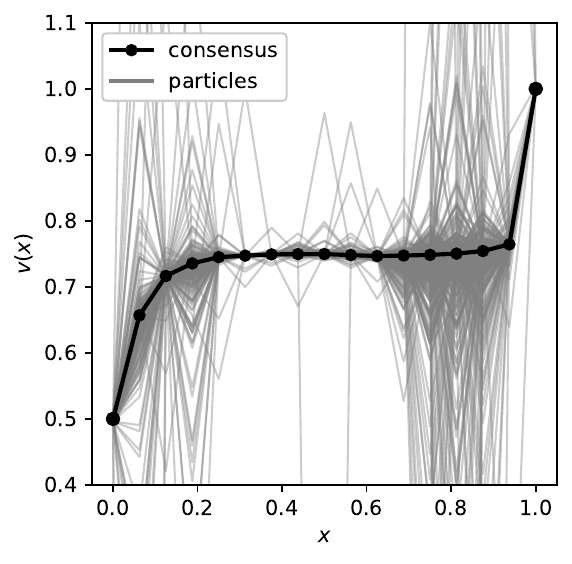}
        \caption{$i=4$}
    \end{subfigure}
    \begin{subfigure}{0.3\linewidth}
        \includegraphics[width=\linewidth]{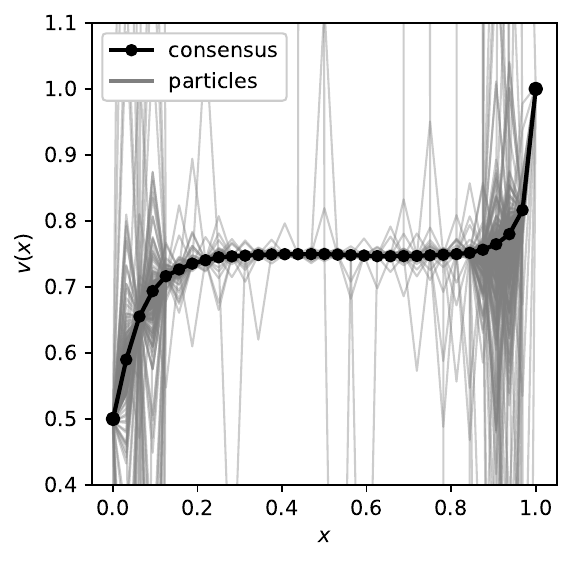}
        \caption{$i=5$}
    \end{subfigure}
    \caption{Particle distributions and consensus points at the end of the CBO runs for different $M = 2^i$.}
    \label{fig:AC_sc0_final_particle_distributions}
\end{figure}

Partial differential equations and optimization problems are inherently intertwined, as many PDE systems emerge as gradient flows of free energy functionals. Noticeably, the Allen--Cahn and Cahn--Hilliard equations can be derived as the $L^2$ and $H^{-1}$ gradient flows of the Ginzburg--Landau (GL) free energy functional, respectively, see, e.g., \cite{feng2003numerical,lee2014physical}. By employing CBO with hierarchically structured noise, we efficiently {
compute} the global minimizer of a {
variational} $p$-Allen--Cahn-type problem within a finite element framework. 

Let $\mathcal{E}_{GL}(v)$ be the $p$-GL free energy functional over the integral $[0,1]$
\begin{equation}
    \label{eq:GinzburgLandauFreeEnergyFunctional}
    \mathcal{E}_{GL}(v) 
    = \int_{[0,1]} \frac{1}{p} \left| \nabla v(s) \right|^p + \frac{1}{\epsilon^2} F(v(s))\mathrm{d}s, 
\end{equation} 
where $F(v) = (v-w_1)^2(v-w_2)^2$ is a so-called double well potential. The minimization problem reads as 
\begin{equation}
    \label{eq:AC_continuous_formulation}
    v^*= \arg\min_{v \in V} \mathcal{E}_{GL}(v) \, ,  
\end{equation} 
where $V = \{v \in W^{1,p} [0,1]\, | \,  v(0) = v_0, v(1) = v_1\}$. We discretize the continuous minimization problem using the finite element method. The interval $[0,1]$ is discretized into $M$ one-dimensional elements $T_j = [(j-1)h, jh]$, $j = 1, \ldots , M$ of equal length $h = M^{-1}$, and we define the finite element space $V_M$
\begin{alignat}{3}
    \label{eq:AC_FE_space}
    V_M & = \big\{v \in V \quad && v|_{T_j} \in \mathcal{P}_1(T_j), \forall j \in \{1, \ldots, M \} \big\} .  
\end{alignat}
Expanding $v \in V_M$ in terms of first-order Lagrange basis functions $\{\phi_j\}_{j=0}^{M}$, such that $v = \sum_{j=0}^M \phi_j x_j$, whereby $x_0 = v_0$, $x_M = v_1$, we obtain the discrete formulation 
\begin{equation}
    \label{eq:AC_discrete_formulation}
    x^* = \arg\min_{x \in \R^{M-1}} 
    \sum_{j=1}^M \left( 
        \frac{h}{p} \left| \frac{x_{j} - x_{j-1}}{h} \right|^p 
        + \frac{1}{\epsilon^2} \tilde{F}_j
    \right). 
\end{equation}
The integral over the double-well potential $\Tilde{F}_j = \int_{(j-1)h}^{jh} F(v(t)) \mathrm{d}t$ can be calculated numerically or even analytically. 
The optimization problem is high-dimensional with $d=M-1$ and can be associated with $V_{M,0} = \text{span }\{ \phi_j, j=1, \ldots, M-1 \}$. Due to the double-well potential, numerous local minima do exist, which, depending on $\epsilon$, retain arbitrarily steep basins of attraction. Since the GL free energy function depends on the gradient of $v$, Gaussian noise is badly suited for the optimization, as it does not reflect the smoothness of the solution. Therefore, naively applying CBO for fine discretizations requires a large number of particles, small time steps, and large times, and thus, becomes unfeasible due to the computational cost. 

\subsubsection{Hierarchical structured noise}
\label{subsubsec:hierarchical}
We overcome these limitations using a hierarchical approach similar to a multi-grid method. We set $M= 2^m$ and solve subsequently the optimization problem over a nested sequence of finite element spaces $V_{2} \subset V_{4} \subset V_{8} \subset \ldots \subset V_{2^m}$ using the final particle distribution of the previous run as the initial particle distribution of the following. Thereby, we exploit that CBO solves the coarse problem efficiently and can find the solution for fine resolutions, provided that the initial particle distribution is sufficiently close to the solution. Similar hierarchical approaches are well studied in the context of solving obstacle problems; see \cite{hue05,graser2009multigrid,woh11}. 

Alternatively, the hierarchical approach can be implemented using the highest resolution and solely changing the noise of the algorithm. During the $i$th run, at each time step and for every particle, we replace the standard anisotropic noise vector $\beta = \sigma D(X_{k\Delta t}^j-X_\alpha(\rho_{k\Delta t}^{N}))N^j(0,\Delta t) \in \R^d$ defined in \eqref{numericeq0} with ${\beta}^i \in \R^d$, such that 
\begin{equation}
    {\beta}_l^i 
    = \int_0^1 \phi_l^m(s) \sum_{k=1}^{2^i-1} \phi_k^i (s) \beta_{k 2^{(m-i)}} \mathrm{d}s, 
    \quad l = 1, \ldots, 2^m-1,  
\end{equation}
where $\{\phi_l^m\}_{l=1}^{2^m-1}$ and $\{\phi_k^i\}_{k=1}^{2^i-1}$ are first-order Lagrangian bases of $V_{2^m,0}$ and  $V_{2^i,0}$, respectively. Accordingly, the initial particle distribution of the first CBO run is drawn with respect to the coarsest resolution. We note that by this construction, we automatically decrease the correlation length of our noise with each refinement step. The resulting particle distributions are visualized in Figure \ref{fig:AC_sc0_final_particle_distributions}. 

For the numerical experiment, we choose $v_0 = 0.5$, $v_1 = 1$, $w_1 = 0.25$, $w_2 = 0.75$, $\epsilon^{-2} = 500$, $p=1.5$ and consider spatial resolutions with $M = 2^i$, $i = 2, \ldots, 7$ elements. The parameters of the algorithm are chosen as $\Delta t = 10^{-2}$, $N = 20d = 2540$, $\lambda=1$, $\sigma=7$, and $\alpha = 10^6$. We use 100 iterations for every resolution not on the finest level and 1000 iterations for the finest. As visualized in Figure \ref{subfig:AC_energy_development_over_CBO_runs}, the energy of the consensus of the CBO run quickly decreases with every refinement step, requiring only a few iterations before reaching the optimizer of the current resolution, and thus, stagnating until the next refinement step. Figure \ref{subfig:AC_consensus_development_over_CBO_runs} shows the convergence of the consensus points at the end of every refinement step, whereby the optimal solutions of coarse resolutions, e.g., $i=2$, may qualitatively differ from those of fine resolutions. By incorporating hierarchically structured noise and tailoring the algorithm's randomness to the specific optimization problem, we effectively solve a $127$-dimensional problem using only $2540$ particles and $1500$ iterations.

\begin{figure}
    \begin{subfigure}{0.45\linewidth}
        \includegraphics[width=\linewidth]{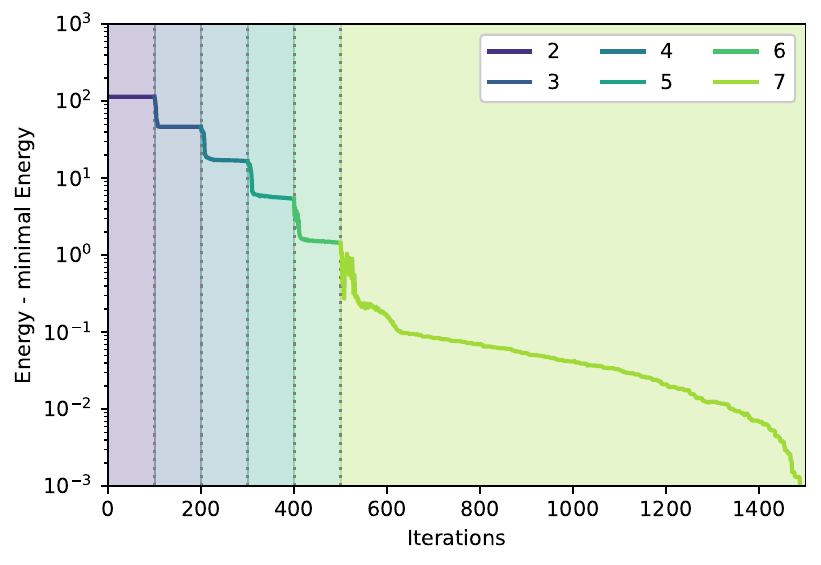}
        \caption{Energy for different $i$}
        \label{subfig:AC_energy_development_over_CBO_runs}
    \end{subfigure}
    \begin{subfigure}{0.45\linewidth}
        \includegraphics[width=\linewidth]{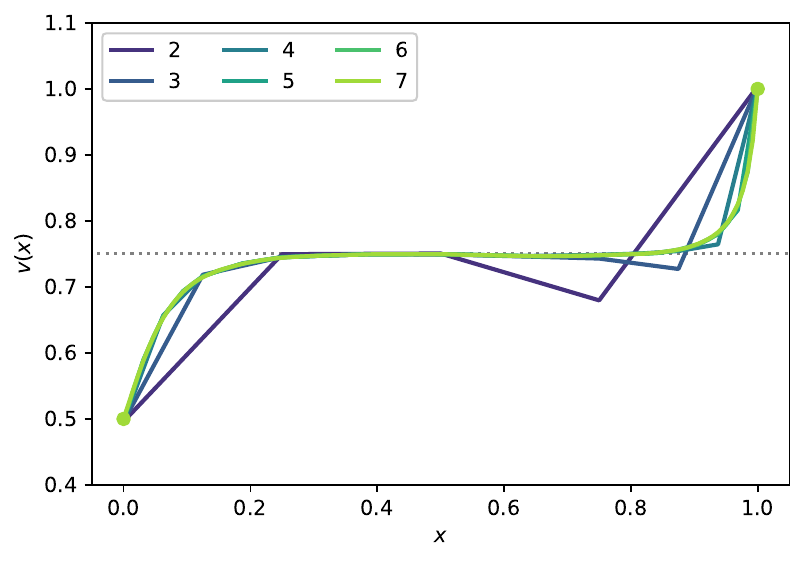}
        \caption{Consensus for different $i$}
        \label{subfig:AC_consensus_development_over_CBO_runs}
    \end{subfigure}
    \caption{Development of the energy \ref{subfig:AC_energy_development_over_CBO_runs} and the consensus points at the end of CBO runs with dimensions $d=2^i-1$, $i=2, \ldots, 7$ and $w_2 = 0.75$ (gray dotted line).}
    \label{fig:AC_development_over_CBO_runs}
\end{figure}

\subsubsection{Obstacle problem} Bridging the gap between global optimization on bounded domains and the finite element formulation previously discussed, we impose an obstacle constraint to the $p$-Allen--Cahn problem. Specifically, we enforce that the solution $v(x)$ satisfies $g(x) \leq v(x) \leq f(x)$ at all grid points of the finest level $x \in \{0, h, 2h \ldots, 1\}$. Thus, the resulting discrete optimization problem is defined within a convex $127$-dimensional hypercuboid. To reduce the complexity, we impose the constraints only at the grid points of the current resolution, beginning with a coarse obstacle and progressively refining it in tandem with the noise resolution. Thus, the solution obtained under coarse obstacle constraints can exhibit qualitative differences compared to those obtained with finer resolutions. Different obstacle constraints and the corresponding consensus points are visualized in Figure \ref{fig:AC_sc4_obstacle_constraints}, emphasizing the influence of the constraints on the consensus despite not being active. 

\begin{figure}
    \begin{subfigure}{0.3\linewidth}
        \includegraphics[width=\linewidth]{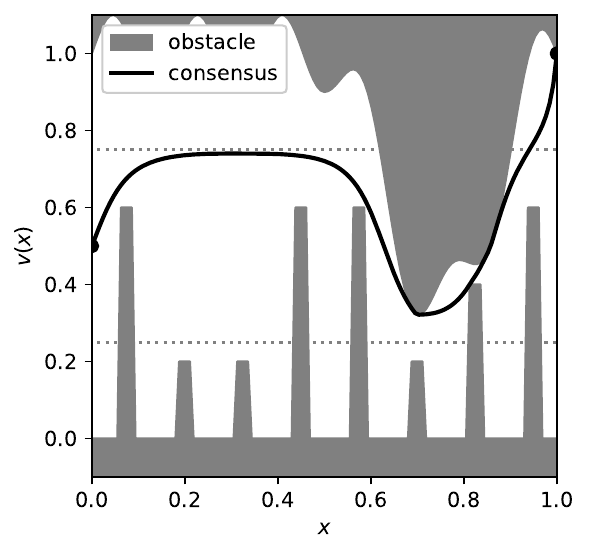}
    \end{subfigure}
    \begin{subfigure}{0.3\linewidth}
        \includegraphics[width=\linewidth]{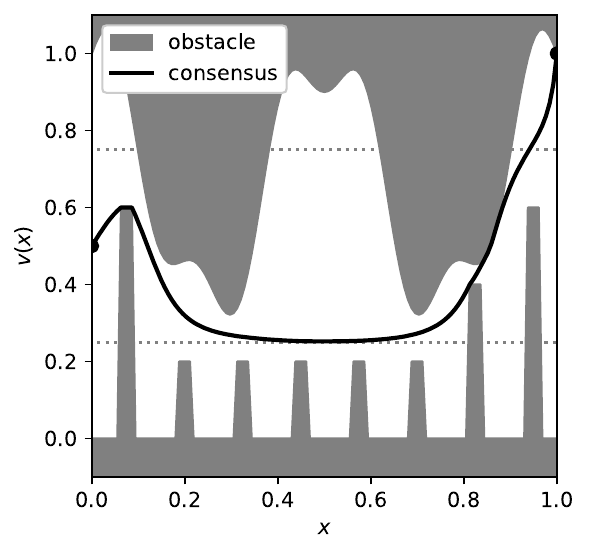}
    \end{subfigure}
    \begin{subfigure}{0.3\linewidth}
        \includegraphics[width=\linewidth]{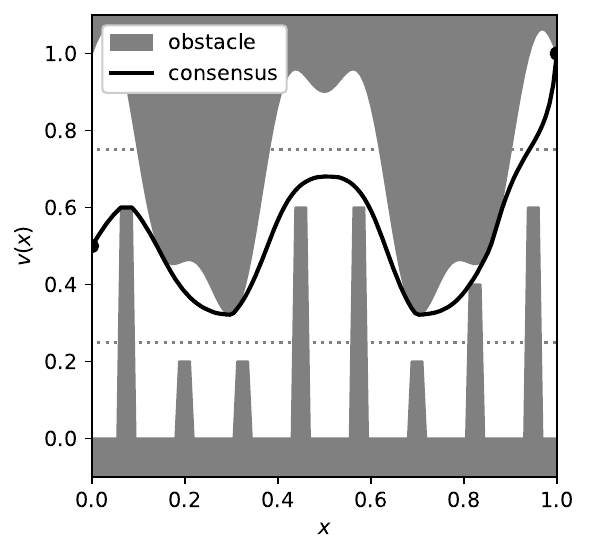}
    \end{subfigure}
    \caption{Various obstacles and consensus points for the $p$-Allen--Cahn problem described in Section \ref{subsubsec:hierarchical} with $p=1.5$, $\epsilon^{-2} = 500$. The obstacle constraints influence the consensus point despite not being active.}
    \label{fig:AC_sc4_obstacle_constraints}
\end{figure}

We apply the CBO to different parameter combinations for $p$ and $\epsilon$ with 1000 and 10000 iterations for the coarser and the finest resolution, respectively. Note that for $p=1, \epsilon^{-2}=2000$ and $p=1.5, \epsilon^{-2}=10000$, there exist multiple global optimizers. The development of the energy of the best particle and the final consensus point are visualized in Figure \ref{fig:AC_sc4_different_parameters}. As the obstacle evolves across varying refinement levels, a solution that was previously optimal may become infeasible or suboptimal following the refinement step, resulting in a non-monotonic energy curve. The progression of the energy associated with the best-performing particle is plotted relative to the minimal energy achieved at the current resolution. CBO converges fast for coarse resolution yet requires more iterations as the resolution of the obstacle and noise becomes finer. In Figure \ref{subfig:sc4_energies_various_p}, for $p=1$ and $p=1.25$, the energy curve exhibits a plateau before eventually declining as the algorithm overcomes a local minimizer. We obtain qualitatively good solutions, even for parameter settings characterized by local minima with steep basins of attraction.

\begin{figure}
    \begin{subfigure}{0.49\linewidth}
        \includegraphics[width=\linewidth]{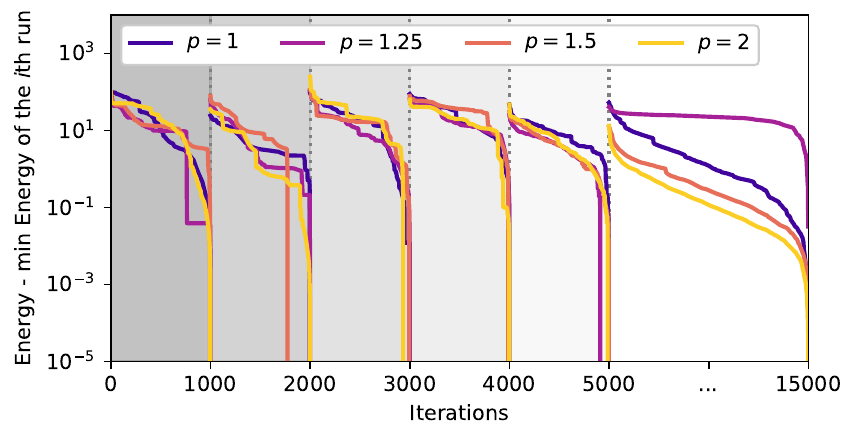}
        \caption{Energy development, $\epsilon^{-2}=2000$}
        \label{subfig:sc4_energies_various_p}
    \end{subfigure}
    \begin{subfigure}{0.49\linewidth}
        \includegraphics[width=\linewidth]{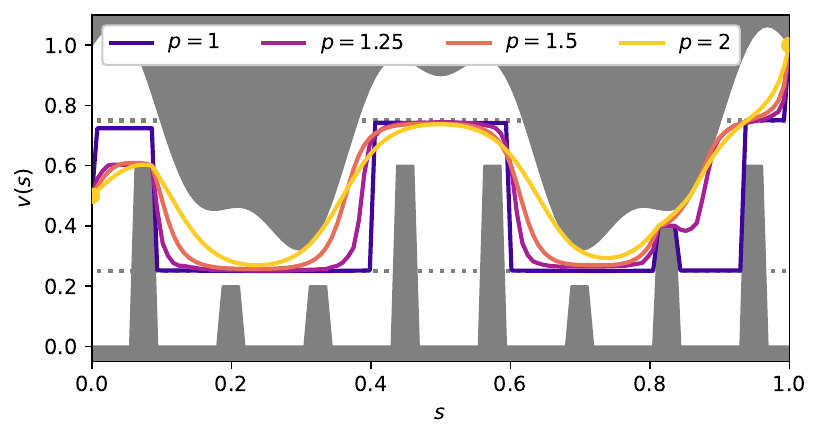}
        \caption{Consensus points, $\epsilon^{-2}=2000$}
        \label{subfig:sc4_consensus_various_p}
    \end{subfigure}
    \begin{subfigure}{0.49\linewidth}
        \includegraphics[width=\linewidth]{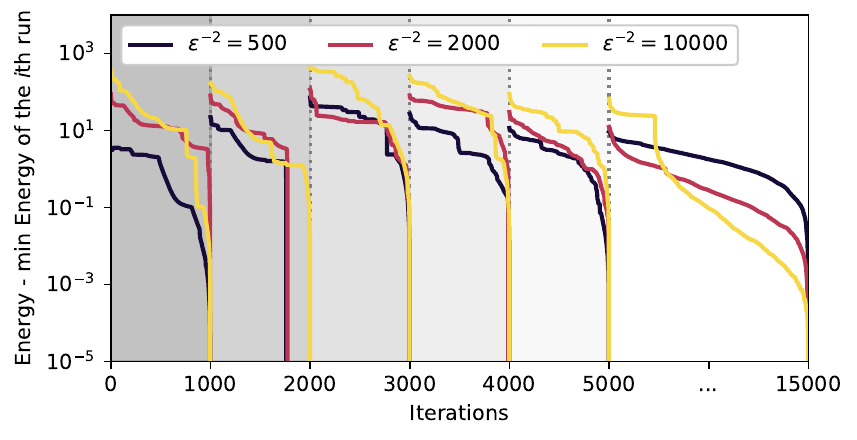}
        \caption{Energy development, $p=1.5$}
        \label{subfig:sc4_energies_various_ww}
    \end{subfigure}
    \begin{subfigure}{0.49\linewidth}
        \includegraphics[width=\linewidth]{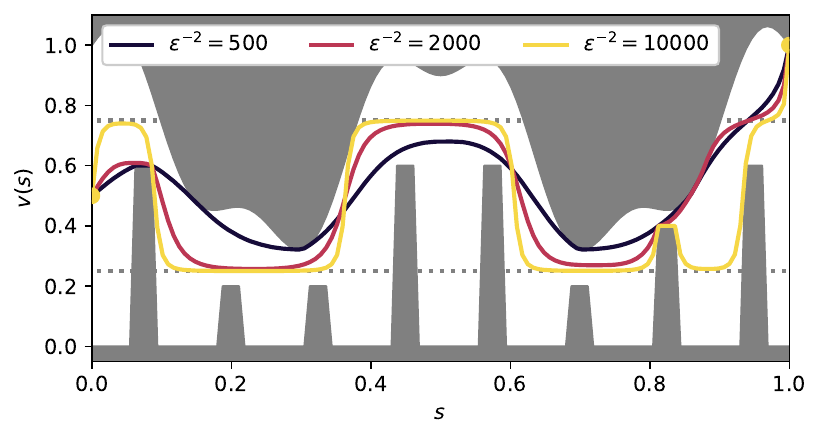}
        \caption{Consensus points, $p=1.5$}
        \label{subfig:sc4_consensus_various_ww}
    \end{subfigure}
    \caption{Development of the energy of the best particle and the consensus point of different $p$-Allen--Cahn problems with an obstacle. The gray dotted lines in (B) and (D) show the double-well potential with $w_1=0.25$ and $w_2=0.75$.}
    \label{fig:AC_sc4_different_parameters}
\end{figure}

By incorporating hierarchically structured noise, we obtain efficiently the global minimizer of the 1D $p$-Allen--Cahn energy discretized by low order $H^1$-conforming elements with a reasonably small number of particles. This advancement extends the applicability of CBO to a broader class of PDE-related problems. Furthermore, it highlights the potential of problem-specific noise to enhance algorithmic performance, to significantly decrease computational costs by reducing the number of particles, and to identify global optimizers in even higher-dimensional spaces.

\section*{Appendix}
\subsection*{Consensus-Based Optimization within the global optimization landscape}

Here, we refer to Figure \ref{fig:CBOworld}, which illustrates how Consensus-Based Optimization (CBO) fits within the broader landscape of global optimization. We begin with the original 2017 paper \cite{pinnau2017consensus}, where the authors introduced CBO as a hybrid approach, combining elements of Particle Swarm Optimization (PSO) and Simulated Annealing (SA). While the connections in \cite{pinnau2017consensus} were initially based on arguments by analogy, subsequent works have rigorously established them. In particular, \cite{grassi2020particle,cipriani2022zero} showed that CBO can be derived as the zero-inertia limit of a suitably formulated version of PSO. This insight enabled the first rigorous analysis of global convergence for PSO, as proposed in \cite{huang2022global}.\\
The fact that CBO is a form of SA was made rigorous in the paper \cite{fornasier2021consensus}, where one of us proved that as the number of particles $N \to +\infty$ and for $\alpha \to +\infty$, each particle tends to the solution of the following annealed Langevin dynamics:
$$
d X(t) = - \nabla \mathcal F(X(t)) + \sigma | X(t)- x^*| d B_t, 
$$
where $\mathcal F(x) = \frac{\lambda}{2} |x - x^*|^2$. \\
Further surprising connections of CBO with other relevant global optimization methods have been derived in \cite{riedl2023gradient} through the Euler-Maruyama approximation of the CBO dynamics \eqref{eq:B1-CBO1}: 
\begin{align}
\label{app1}
X^i_{(k+1)\Delta t} &= X^i_{k\Delta t} - \Delta t \lambda(X^i_{k\Delta t}-X_{\alpha}( \rho_{k\Delta t}^N)) + \sqrt{\Delta t}\sigma |X^i_{k\Delta t}-X_{\alpha}( \rho_{k\Delta t}^N)| N^i_k(0,1), 
\end{align}
By scaling $\lambda = \frac{1}{\Delta t}$ one obtains
\begin{align}
\label{app2}
X^i_{(k+1)\Delta t} &= X_{\alpha}( \rho_{k\Delta t}^N)) + \sqrt{\Delta t}\sigma |X^i_{k\Delta t}-X_{\alpha}( \rho_{k\Delta t}^N)| N^i_k(0,1), 
\end{align}
and, consequently, the so-called {\it Consensus Hopping scheme (CH)} as introduced in \cite{riedl2023gradient}
\begin{align}
\label{app3}
X_{k+1}^{CH} =  X_{\alpha}( \mathcal N(X_{k}^{CH}, \sigma_k^2)),
\end{align}
which is the consensus point of a Gaussian centered in the previous consensus point with a suitable variance $\sigma_k^2$. Let us first notice that the CH scheme is an instance of $(1,\lambda)$ Evolution Strategies (ES), with $\mu=1$ parent and $\lambda=N$ offspring\footnote{We recall that a $(\mu,\lambda)$-Evolution Strategy \cite{AugerHansen2005} maintains a population of $\mu$ candidate solutions, called parents. In each iteration, the algorithm generates $\lambda$ new candidate solutions, called offspring, by applying a  ``mutation" to the parents. The mutation operator typically adds random perturbations drawn from a normal distribution to the parent solutions. (Compare \eqref{app3}.)}. Proofs of global convergence for ES were obtained long ago, for instance in \cite{Rudolph1997,Rudolph1998}, using arguments based on Markov chains. However, most classical proofs of global convergence for ES do not provide explicit convergence rates. Based on the connection between CBO and ES obtained through \eqref{app1}--\eqref{app3}, it is now possible to also establish exponential convergence rates, as discovered in the forthcoming paper \cite{fornasier2025}. Connections between CBO and ES have been further investigated in \cite{roith25}.\\

By the Laplace principle, the explicit expression for the consensus point in \eqref{app3} can be approximated, for a suitable scaling of $\alpha \to \infty$ with respect to $\sigma_k \to 0$, by
\begin{eqnarray}
 X_{k+1}^{CH} &=&  X_{\alpha}( \mathcal N(X_{k}^{CH},\sigma_k^2) \nonumber\\
    &=& \fint_{\mathbb R^d} x \, e^{-\alpha \left (\mathcal{E}(x) + \frac{1}{2 \sigma_k^2 \alpha} |x-X_{k}^{CH}|^2 \right )} dx \nonumber\\
    &\approx& \arg\min_{x \in \mathbb R^d} \mathcal{E}(x) + \frac{1}{2 \tau} |x- X_{k}^{CH}|^2, \label{app4}
\end{eqnarray}
which can be readily recognized as a {\it Minimizing Movement Scheme (MMS)} or a {\it Proximal Point Method}. In other words, the iteration of CH follows approximately a Gradient Descent (GD), a surprising connection that was recently made rigorous in \cite{riedl2023gradient}, where it was shown that CBO can be interpreted as a suitable stochastic relaxation of GD.\\

Furthermore, the MMS is just one instance of a broader class of proximal point schemes of the form
\begin{align}\label{app5}
X_{k+1} = \arg \min_{x \in \mathbb R^d} \mathcal{E}(x) + \frac{1}{p \tau^{p-1}} |x- X_k|^p,
\end{align}
for any $1 < p < \infty$. As $p \to +\infty$, such schemes converge to
\begin{align}\label{app6}
X_{k+1} = \arg \min_{x \in B(X_k,R_\tau)} \mathcal{E}(x),
\end{align}
which restricts the search for the next approximation to an optimal point within a {\it Trust Region (TR)}. Notice that for $\tau \to 0$ suitable piecewise interpolations of both iterations in \eqref{app5} and \eqref{app6} converge to a gradient flow trajectory with different time reparametrizations, a well-known insight since the 1993 groundbreaking work of De Giorgi \cite{de1993new}. Hence, the conceptual path from \eqref{app3} to \eqref{app6} illustrates how CBO  is also connected  to both GD and TR methods \cite{ConnGouldToint2000}.
We conclude mentioning the work \cite{Borghi_2025} where {\it Genetic Algorithms (GA)} are also put in connection to CBO through kinetic  models \cite{borghi2024kinetic}.
\\

In conclusion, CBO occupies a very central position (Figure \ref{fig:CBOworld}) within a composite landscape of related methods for global optimization. In light of these connections, the powerful mathematical techniques very recently developed for analyzing the global convergence of CBO (starting with \cite{carrillo2018analytical,carrillo2019consensus,fornasier2021consensus,ha2021convergence,doi:10.1142/S0218202522500245,huang2025faithful}), as  further developed in this paper, have also enabled further breakthroughs, such as describing the global convergence of PSO \cite{huang2022global} and deriving convergence rates for ES \cite{fornasier2025}. We believe that, given their broad applicability, the significance of these mathematical techniques cannot be overstated, and they have the potential to inspire further advances in global optimization; this paper wants to represent a meaningful contribution in that direction.

\subsection*{Proof of Proposition \ref{propositive1}}
\begin{proof}
To simplify the notation, in the following, we drop the subscript $x_0$ in $\phi_r^{x_0}$, and it is easy to compute
	\begin{align}
		\frac{d \EE[\phi_r(\OX_t)]}{dt}=\sum_{k=1}^{d}\EE[T_{1,k}(\OX_t)+T_{2,k}(\OX_t)]
	\end{align}
	with
	\begin{equation}
		T_{1,k}:=-\lambda(\OX_t-X_\alpha(\rho_t))_{k}\partial_{x_k}\phi_r(\OX_t)\mbox{ and }T_{2,k}:=\frac{\sigma^2}{2}(\OX_t-X_\alpha(\rho_t ))_{k}^2\partial_{x_k^2}^2\phi_r(\OX_t)\,.
	\end{equation}
	Here the boundary term disappears because $\mbox{supp}(\phi_r)=B_r(x_0)\cap \DD \subset \DD$.
	For each $k=1,\dots,d$, let us define the subsets
	\begin{equation}
		K_{1,k}:=\{x\in\ODD:\,|(x-x_0)_k|>\sqrt{c}r\}
	\end{equation}
	and
	\begin{align}
		K_{2,k}:=\bigg\{x\in\ODD:&-\lambda(x-X_{\alpha}(\rho_{t} ))_{k}(x-x_0)_{k}(r^{2}-(x-x_0)_{k}^{2})^{2} \\
		&>\tilde{c} r^{2} \frac{\sigma^{2}}{2}(x-X_{\alpha}(\rho_{t} ))_{k}^{2}(x-x_0)_{k}^{2} 
		\bigg\}
	\end{align}
	where $\tilde c:=2c-1\in (0,1)$. Then we decompose $B_r(x_0)$ according to
	\begin{equation}
		B_r(x_0)=\left(K_{1,k}^{c} \cap  B_r(x_0)\right) \cup\left(K_{1,k} \cap K_{2,k}^{c} \cap  B_r(x_0)\right) \cup\left(K_{1,k} \cap K_{2,k} \cap  B_r(x_0)\right)
	\end{equation}
	and one considers each of these subsets separately.
	
	\textbf{Subset} $K_{1,k}^c\cap B_r(x_0):$ We have $|(\OX_t-x_0)_k|\leq \sqrt{c}r$ for each $\OX_t\in K_{1,k}^c$. For $T_{1,k}$, we use the expression of $\partial_{x_k}\phi_r$ and get
	\begin{align}
		T_{1,k}(\OX_t) &=2 r^{2} \lambda\left(\OX_t-X_{\alpha}\left(\rho_{t} \right)\right)_{k} \frac{\left(\OX_t-x_0\right)_{k}}{\left(r^{2}-\left(\OX_t-x_0\right)_{k}^{2}\right)^{2}} \phi_{r}(\OX_t) \\
		& \geq-2 r^{2} \lambda \frac{\left|\left(\OX_t-X_{\alpha}\left(\rho_{t} \right)\right)_{k}\right|\left|\left(\OX-x_0\right)_{k}\right|}{\left(r^{2}-\left(\OX_t-x_0\right)_{k}^{2}\right)^{2}} \phi_{r}(\OX_t) \geq-\frac{2 \lambda (\sqrt{c}r+|\DD|)\sqrt{c}}{(1-c)^{2} r} \phi_{r}(\OX_t) \\
		&=:-\vartheta _{1} \phi_{r}(\OX_t)\,,
	\end{align}
	where in the last inequality we have used 
	\begin{equation}\label{eqdiff}
		\left|\left(\OX_t-X_{\alpha}\left(\rho_{t} \right)\right)_{k}\right|\leq  \left|\left(\OX_t-x_0\right)_{k}\right|+\left|\left(X_{\alpha}\left(\rho_{t} \right)-x_0\right)_{k}\right| \leq  \sqrt{c}r+|\DD|.
	\end{equation}
	Similarly for $T_{2,k}$ one obtains
	\begin{align}
		& T_{2,k}(\OX_t) \\ \nonumber
        & =\sigma^{2} r^{2}\left(\OX_t-X_{\alpha}\left(\rho_{t} \right)\right)_{k}^{2} \frac{2\left(2\left(\OX_t-x_0\right)_{k}^{2}-r^{2}\right)\left(\OX_t-x_0\right)_{k}^{2}-\left(r^{2}-\left(\OX_t-x_0\right)_{k}^{2}\right)^{2}}{\left(r^{2}-\left(\OX_t-x_0\right)_{k}^{2}\right)^{4}} \phi_{r}(\OX_t) \\
		& \geq-\frac{ \sigma^{2}(\sqrt{c}r+|\DD|)^2(2 c+1)}{(1-c)^{4} r^{2}} \phi_{r}(\OX_t)=:-\vartheta_{2} \phi_{r}(\OX_t)\,.
	\end{align}

	\textbf{Subset} $K_{1,k}\cap K_{2,k}^c\cap B_r(x_0):$  As  $\OX_t\in K_{1,k}$ we have $|(\OX_t-x_0)_k|>\sqrt{c}r$. We observe that $T_{1,k}(\OX_t)+T_{2,k}(\OX_t)\geq 0$ for all $\OX_t$ satisfying
	\begin{align}\label{positve}
		&\left(-\lambda\left(\OX_t-X_{\alpha}\left(\rho_{t} \right)\right)_{k}\left(\OX_t-x_0\right)_{k}+\frac{\sigma^{2}}{2}\left(\OX_t-X_{\alpha}\left(\rho_{t} \right)\right)_{k}^{2}\right)\left(r^{2}-\left(\OX_t-x_0\right)_{k}^{2}\right)^{2} \\
		\leq &\sigma^{2}\left(\OX_t-X_{\alpha}\left(\rho_{t} \right)\right)_{k}^{2}\left(2\left(\OX_t-x_0\right)_{k}^{2}-r^{2}\right)\left(\OX_t-x_0\right)_{k}^{2}\,.
	\end{align}
	Indeed, this can be verified by first showing that
	\begin{align}
		&-\lambda\left(\OX_t-X_{\alpha}\left(\rho_{t} \right)\right)_{k}\left(\OX_t-x_0\right)_{k}\left(r^{2}-\left(\OX_t-x_0\right)_{k}^{2}\right)^{2} \\ 
        &\quad \leq \tilde{c} r^{2} \frac{\sigma^{2}}{2}\left(\OX_t-X_{\alpha}\left(\rho_{t} \right)\right)_{k}^{2}\left(\OX_t-x_0\right)_{k}^{2} \\ 
		&\quad =(2 c-1) r^{2} \frac{\sigma^{2}}{2}\left(\OX_t-X_{\alpha}\left(\rho_{t} \right)\right)_{k}^{2}\left(\OX_t-x_0\right)_{k}^{2} \\ 
        &\quad \leq \left(2\left(\OX_t-x_0\right)_{k}^{2}-r^{2}\right) \frac{\sigma^{2}}{2}\left(\OX_t-X_{\alpha}\left(\rho_{t} \right)\right)_{k}^{2}\left(\OX_t-x_0\right)_{k}^{2}\,,
	\end{align}
	where we have used the fact that $\OX_t\in K_{1,k}\cap K_{2,k}^c$ and $\tilde c=2c-1$. One also notice that
	\begin{align}
		\frac{\sigma^{2}}{2} &\left(\OX_t-X_{\alpha}\left(\rho_{t} \right)\right)_{k}^{2}\left(r^{2}-\left(\OX_t-x_0\right)_{k}^{2}\right)^{2} \leq \frac{\sigma^{2}}{2}\left(\OX_t-X_{\alpha}\left(\rho_{t} \right)\right)_{k}^{2}(1-c)^{2} r^{4} \\
		& \leq \frac{\sigma^{2}}{2}\left(\OX_t-X_{\alpha}\left(\rho_{t} \right)\right)_{k}^{2}(2 c-1) r^{2} c r^{2} \\
        & \leq \frac{\sigma^{2}}{2}\left(\OX_t-X_{\alpha}\left(\rho_{t} \right)\right)_{k}^{2}\left(2\left(\OX_t-x_0\right)_{k}^{2}-r^{2}\right)\left(\OX_t-x_0\right)_{k}^{2} 
	\end{align}
	by using $(1-c)^2\leq(2c-1)c$. Hence \eqref{positve} holds and we have $T_{1,k}(\OX_t)+T_{2,k}(\OX_t)\geq 0$.

	\textbf{Subset} $K_{1,k}\cap K_{2,k}\cap B_r(x_0):$ Notice that when $(\OX_t)_k=(X_\alpha(\rho_t ))_k$, we have $T_{1,k}=T_{2,k}=0$, so in this case there is nothing to prove. If $(\OX_t)_k\neq(X_\alpha(\rho_t ))_k$, or $\sigma^2(\OX_t-X_\alpha(\rho_t ))_k^2>0$ $(\sigma>0)$, we exploit
	$\OX_t\in K_{2,k}$ to get
	\begin{align}
		\frac{\left(\OX_t-X_{\alpha}\left(\rho_{t} \right)\right)_{k}\left(\OX_t-x_0\right)_{k}}{\left(r^{2}-\left(\OX_t-x_0\right)_{k}^{2}\right)^{2}} & \geq \frac{-\left|\left(\OX_t-X_{\alpha}\left(\rho_{t} \right)\right)_{k}\right|\left|\left(\OX_t-x_0\right)_{k}\right|}{\left(r^{2}-\left(\OX_t-x_0\right)_{k}^{2}\right)^{2}} \\
		&>\frac{2 \lambda\left(\OX_t-X_{\alpha}\left(\rho_{t} \right)\right)_{k}\left(\OX_t-x_0\right)_{k}}{\tilde{c} r^{2} \sigma^{2}\left|\left(\OX_t-X_{\alpha}\left(\rho_{t} \right)\right)_{k}\right|\left|\left(\OX_t-x_0\right)_{k}\right|} \geq-\frac{2 \lambda}{\tilde{c} r^{2} \sigma^{2}} .
	\end{align}
	Using this, $T_{1,k}$ can be bounded from below
	\begin{align}
		T_{1,k}(\OX_t) & =2 r^{2} \lambda\left(\OX_t-X_{\alpha}\left(\rho_{t} \right)\right)_{k} \frac{\left(\OX_t-x_0\right)_{k}}{\left(r^{2}-\left(\OX_t-x_0\right)_{k}^{2}\right)^{2}} \phi_{r}(\OX_t) \\ 
        & \geq-\frac{4 \lambda^{2}}{\tilde{c} \sigma^{2}} \phi_{r}(\OX_t)=:-\vartheta_{3} \phi_{r}(\OX_t)\,.
	\end{align}
	Moreover since $\OX_t\in K_{1,k}$ and $2(2c-1)c\geq(1-c)^2$ implied by the assumption, one has
	\begin{equation}
		2\left(2\left(\OX_t-x_0\right)_{k}^{2}-r^{2}\right)\left(\OX_t-x_0\right)_{k}^{2} \geq\left(r^{2}-\left(\OX_t-x_0\right)_{k}^{2}\right)^{2}\,,
	\end{equation}
	which yields that $T_{2,k}(\OX_t)\geq 0$. 
	
	\textbf{Concluding the proof:} Collecting estimates from above, we get
	\begin{align}
		\frac{d \EE[\phi_r(\OX_t)]}{dt} & =\sum_{k=1}^{d}\EE[T_{1,k}(\OX_t)+T_{2,k}(\OX_t)]\nn\\
		=&\sum_{k=1}^{d}\bigg(\EE[(T_{1,k}(\OX_t)+T_{2,k}(\OX_t))\textbf{I}_{K_{1,k}\cap K_{2,k}\cap B_r(x_0)}]\nn\\
        &+\EE[(T_{1,k}(\OX_t)+T_{2,k}(\OX_t))\textbf{I}_{K_{1,k}\cap K_{2,k}^c\cap B_r(x_0)}]\nn\\
		&+\EE[(T_{1,k}(\OX_t)+T_{2,k}(\OX_t))\textbf{I}_{K_{1,k}^c\cap B_r(x_0)}]\bigg)\nn\\
		\geq&-d(\vartheta_1+\vartheta_2+\vartheta_3)\EE[\phi_r(\OX_t)]=:-\vartheta\EE[\phi_r(\OX_t)]\,.
	\end{align}
	An application of Gronwall's inequality concludes that
	\begin{equation}
		\PP(\OX_t\in  B_r(x_0))\geq \EE[\phi_r(\OX_t)]\geq \EE[\phi_r(\OX_0)]\exp(-\vartheta t)\,.
	\end{equation}
\end{proof}

\section*{Data Availability Statement}
 All data and numerical experiments presented in this paper are partially based on CBXpy \cite{Bailo_CBX_Python_and_2024} and are available for reproducibility at \href{https://github.com/echnen/CBO-with-boundaries}{https://github.com/echnen/CBO-with-boundaries}. 

\section*{Acknowledgement}
E.C. acknowledges the support of the project P 34922-N of the Austrian Science Fund (FWF). M.F. acknowledges the support of the Munich Center for Machine Learning and the ERC Advanced Grant NEITALG, grant agreement No. 101198055.

\bigskip
\begin{center}
  \FundingLogos
  
  \vspace{0.5em}
  \begin{tcolorbox}\centering\small
    Funded by the European Union. Views and opinions expressed are however those of the author(s) only and do not necessarily reflect those of the European Union or the European Research Council Executive Agency. Neither the European Union nor the granting authority can be held responsible for them.
    This project has received funding from the European Research Council (ERC) under the European Union’s Horizon Europe research and innovation programme (grant agreement No. 101198055, project acronym NEITALG).
    
  \end{tcolorbox}
\end{center}

\bibliography{biblio}

\end{document}